\documentclass[12pt]{article}% Math and Physical Sciences Reference Style
\usepackage{color}
\usepackage{amsmath}
\usepackage{amsfonts,amssymb}
\usepackage{extarrows}
\usepackage{geometry}
\usepackage{authblk} % \affil
\usepackage{hyperref} % ref in PDF
\usepackage{mathrsfs}
\usepackage[utf8]{inputenc}
\def\0{\emptyset}

\usepackage{amsthm}
\newtheorem{theorem}{Theorem}

\newtheorem{lemma}[theorem]{Lemma}
\newtheorem{claim}[theorem]{Claim}
\newtheorem{observation}[theorem]{Observation}
\newtheorem{corollary}[theorem]{Corollary}
\newtheorem{proposition}[theorem]{Proposition}
\newtheorem{conjecture}[theorem]{Conjecture}

\newtheorem*{remark*}{Remark}
\usepackage{graphicx}

\usepackage{mathtools} % coloneqq
\usepackage{enumerate}
\usepackage{enumitem}

\usepackage{
	amsmath,			% Math Environments
	amssymb,			% Extended Symbols
	enumerate,		    % Enumerate Environments
	graphicx,			% Include Images
	lastpage,			% Reference Lastpage
	multicol,			% Use Multi-columns
	multirow,			% Use Multi-rows
	pifont,			    % For Checkmarks
}

\usepackage[numbers]{natbib}

\DeclareMathOperator{\ex}{\mathrm{ex}}
\DeclareMathOperator{\ind}{\mathrm{ind}}

\usepackage{algorithm}
\usepackage{algpseudocode}  % 更现代的伪代码包
\floatname{algorithm}{Procedure}

\begin{document}

% --- PAPER INFO ---

\title{The maximum number of odd cycles in planar graphs forbidding shorter odd cycles}

\author[1]{\small\bf Yichen Wang\thanks{E-mail: wangyich22@mails.tsinghua.edu.cn}}
\author[2]{\small\bf Ervin Gy\H{o}ri\thanks{E-mail: gyori.ervin@renyi.hu}}
\author[3]{\small\bf Zhen He\thanks{Corresponding author. E-mail: zhenhe@bjtu.edu.cn}}
% \author[3]{\small\bf Xin Cheng\thanks{E-mail: xincheng@mail.nwpu.edu.cn}}
% \author[2]{\small\bf Kitti Varga\thanks{E-mail: vkitti@cs.bme.hu}}
% \author[4,5]{\small\bf Yuanpei Wang\thanks{E-mail: boyuan@shu.edu.cn}}
% \author[1]{\small\bf Xiamiao Zhao\thanks{E-mail: zxm23@mails.tsinghua.edu.cn}}
% \author[4,5]{\small\bf Junpeng Zhou\thanks{E-mail: junpengzhou@shu.edu.cn}}

\affil[1]{\small Department of Mathematical Sciences, Tsinghua University, Beijing, P.R. China.}
\affil[2]{\small HUN-REN Alfréd Rényi Institute of Mathematics, 1053 Budapest, Hungary.}
\affil[3]{\small School of Mathematics and Statistics, Beijing Jiaotong University, Beijing, China.}
% \affil[3]{\small School of Mathematics and Statistics, Northwestern Polytechnical University and Xi'an-Budapest Joint Research Center for Combinatorics, Xi'an 710129, Shaanxi, P.R. China.}

% \affil[4]{\small Department of Mathematics, Shanghai University, Shanghai 200444, P.R. China.}
% \affil[5]{\small Newtouch Center for Mathematics of Shanghai University, Shanghai 200444, P.R. China.}

%\author{Yichen Wang}
\date{}

\maketitle\baselineskip 16.3pt

\begin{abstract}
    Given a graph $H$ and a family of graphs $\mathcal{F}$, the generalized planar Tur\'an number $\ex_\mathcal{P}(n, H, \mathcal{F})$ is the maximum number of copies of $H$ in an $n$-vertex planar graph that contains no graph $F \in \mathcal{F}$ as a subgraph.
    When only induced copies of $H$ are counted, we denote the corresponding generalized planar Tur\'an number by $\ex_\mathcal{P}(n, H^{\ind}, \mathcal{F})$.
    Gy\H{o}ri and Karim~\cite{gyHori2024generalized} determined $\ex_\mathcal{P}(n, C_{5}, \{C_3\})$.
    In this paper, we determine the exact value of $\ex_\mathcal{P}(n, C_{2k+1}, \{C_3,C_5,\ldots,C_{2k-1}\})$ for every $k \ge 3$.

    Since all shorter odd cycles are forbidden, every $C_{2k+1}$ is induced.
    This problem is closely related to the inducibility of odd cycles in planar graphs.
    Ghosh, Gy\H{o}ri, Janzer, Paulos, Salia and Zamora~\cite{ghosh2022maximum}~(and independently Savery~\cite{savery2024planar}) determined the exact value of $\ex_\mathcal{P}(n, C_5^{\ind}, \emptyset)$.
    Moreover, they established a conjecture for all odd cycles $C_{2k+1}$ with $k \ge 3$.
    Our result confirms their conjecture under the additional assumption that all shorter odd cycles are forbidden.
\end{abstract}

% %\textbf{AMS classification: }\textit{05C75, 05C65, 05C05}\vskip 0.3cm


% -------------------------------------------------------------
% --------- Here begins INTRODUCTION SECTION ------------------
\section{Introduction}
% \begin{figure}
%     \centering
%     \includegraphics[width=\textwidth]{src/PlanarCycle-01.png}
%     \caption{The logic of propositions and lemmas.}\label{fig: logic}
% \end{figure}
For an integer $k\geq 3$, let $C_k$ denote the cycle with $k$ vertices, and $P_k$ denote the path with $k$ vertices.
Let $\mathcal{C}^o_{<2k+1} = \{C_3, C_5, \ldots, C_{2k-1}\}$ be the set of all odd cycles of length less than $2k+1$.
Let $H$ be a graph and $\mathcal{F}$ a family of graphs.
The generalized Tur\'an number $\ex(n, H, \mathcal{F})$ is the maximum number of copies of $H$ in an $n$-vertex graph that contains no graph $F \in \mathcal{F}$ as a subgraph.
When $H$ is an edge, $\ex(n,H,G)$ is the classical Tur\'an number $\ex(n,G)$.
Restricting the host graph to be planar gives the planar Tur\'an problem.
Let $\ex_\mathcal{P}(n, H,  \mathcal{F})$ denote the maximum number of copies of $H$ in an $n$-vertex planar graph that contains no graph $F \in \mathcal{F}$ as a subgraph.
Let $f(n,H)$ denote the maximum number of copies of $H$ in an $n$-vertex planar graph, i.e. $f(n,H) = \ex_\mathcal{P}(n,H,\emptyset)$.

This problem was initiated by Hakimi and Schmeichel~\cite{hakimi1979number}, who determined $f(n,C_3)$ and $f(n,C_4)$ exactly.
$f(n, C_4)$ was also studied by Alameddine~\cite{alameddine1980number} independently.
In the general case, the order of magnitude of $f(n,H)$ was studied by many researchers, for example, when $H$ is a tree~\cite{gyHori2020generalized} and when $H$ is an arbitrary graph~\cite{huynh2022subgraph}.
There are also many results about specific graphs $H$.
Alon and Caro~\cite{alon1984number} studied $f(n,K_{2,t})$ for $t\geq 2$.
Wood~\cite{wood2007maximum} determined the exact value of $f(n,K_{4})$.
Grzesik, Gy\H{o}ri, Janzer, Paulos, Salia and Zamora gave the exact value of $f(n,P_4)$~\cite{grzesik2022maximum}.
For longer paths, Antonir and Shapira obtained good estimates for $f(n,P_{2k+1})$~\cite{antonir2024bounding}.

There is also a substantial body of work on the case in which $H$ is a cycle.
Hakimi and Schmeichel~\cite{hakimi1979number} gave an upper bound for $f(n,C_5)$ and conjectured the exact value of $f(n,C_5)$ in the same paper.
Later, Gy\H{o}ri, Paulos, Salia, Tompkins and Zamora~\cite{gyHori2019maximum} confirmed their conjecture.
Recently, Cox and Martin~\cite{cox2022counting, cox2023maximum} determined the asymptotic value of $f(n,C_{2k})$ for $k \in \{3,4,5,6\}$.
Lv, Gy\H{o}ri, He, Salia, Tompkins and Zhu~\cite{lv2024maximum} determined the asymptotic value of $f(n,C_{2k})$ for all $k$.
Heath, Martin and Wells~\cite{heath2025maximum} obtained a bound on $f(n,C_{2k+1})$, which is asymptotically tight up to a factor of $3/2$, and for $k\in \{3,4\}$, they gave an asymptotically tight result.

% Alon and Shikhelman~\cite{alon2016many}  introduced a generalized extremal function $ex(n,H,\mathcal{F})$, defined to be the maximum number of copies of a graph $H$ in an $n$-vertex graph with no $F\in \mathcal{F}$ as a subgraph, where $\mathcal{F}$ is a family of graphs.
% More recently,
Gy\H{o}ri, Paulos, Salia, Tompkins and Zamora~\cite{gyHori2020generalized} initiated the study of the general Tur\'an problem in planar graphs.
% Similarly, let $ex_\mathcal{P}(n,H,\mathcal{F})$ denote the maximum number of copies of a graph $H$ in an $n$-vertex, $\mathcal{F}$-free planar graph.
They determined the order of $\ex_\mathcal{P}(n,C_k,C_4)$ for all $k\geq 5$, and gave the exact value of $\ex_\mathcal{P}(n,C_5,C_4)$.
Hakimi and Schmeichel~\cite{hakimi1979number} determined the exact value of $\ex_\mathcal{P}(n,C_3,C_4)$.
Gy\H{o}ri and Karim~\cite{gyHori2024generalized} determined the exact value of $\ex_\mathcal{P}(n,C_\ell,C_3)$ for $\ell\in\{4,5,6\}$, and also the sharp upper bound of $\ex_\mathcal{P}(n,C_3,C_\ell)$ for $\ell\in \{4,5,6\}$.

In this paper, we determine the generalized planar Tur\'an number $\ex_\mathcal{P}(n,C_{2k+1},\mathcal{C}^o_{<2k+1})$ as follows.
\begin{theorem}\label{thm: main}
    For every $k \ge 3$, there exists a positive integer $N = N(k)$ such that, for all $n \ge N$,
    \[
        \ex_\mathcal{P}(n,C_{2k+1},\{C_3,C_5,\ldots,C_{2k-1}\}) = h_k(n),
    \]
    where
    \[
        h_k(n) = \max \{x_1x_2\cdots x_k \mid x_1 + x_2 + \cdots + x_k = n-k-1, \text{~each $x_i$ is a positive integer}\}.
    \]
    The bound is attained by the following construction~(see Figure~\ref{fig: extremal construction}):
    \begin{enumerate}
        \item Let $u_1v_1u_2v_2 \ldots u_{k}a_2a_1u_1$ be a $C_{2k+1}$.
        \item Blow up $v_1,v_2,\ldots,v_{k-1}$ to sets of sizes $x_1,x_2,\ldots,x_{k-1}$, respectively.
        \item Replace $a_1a_2$ with any tree with $x_k$ edges. Since any tree is bipartite, connect one part to $u_1$ and the other part to $u_k$.
        \item We have $x_1+x_2+\cdots+x_{k-1}+x_k = n-k-1$. It is straightforward to verify that the number of copies of $C_{2k+1}$ is at most $h_k(n)$, with equality when the $x_i$ differ by at most one.
    \end{enumerate}
    Note that the construction is not unique, since the tree may be chosen in many different ways.
\end{theorem}

\begin{figure}
    \centering
    \includegraphics[width=0.5\textwidth]{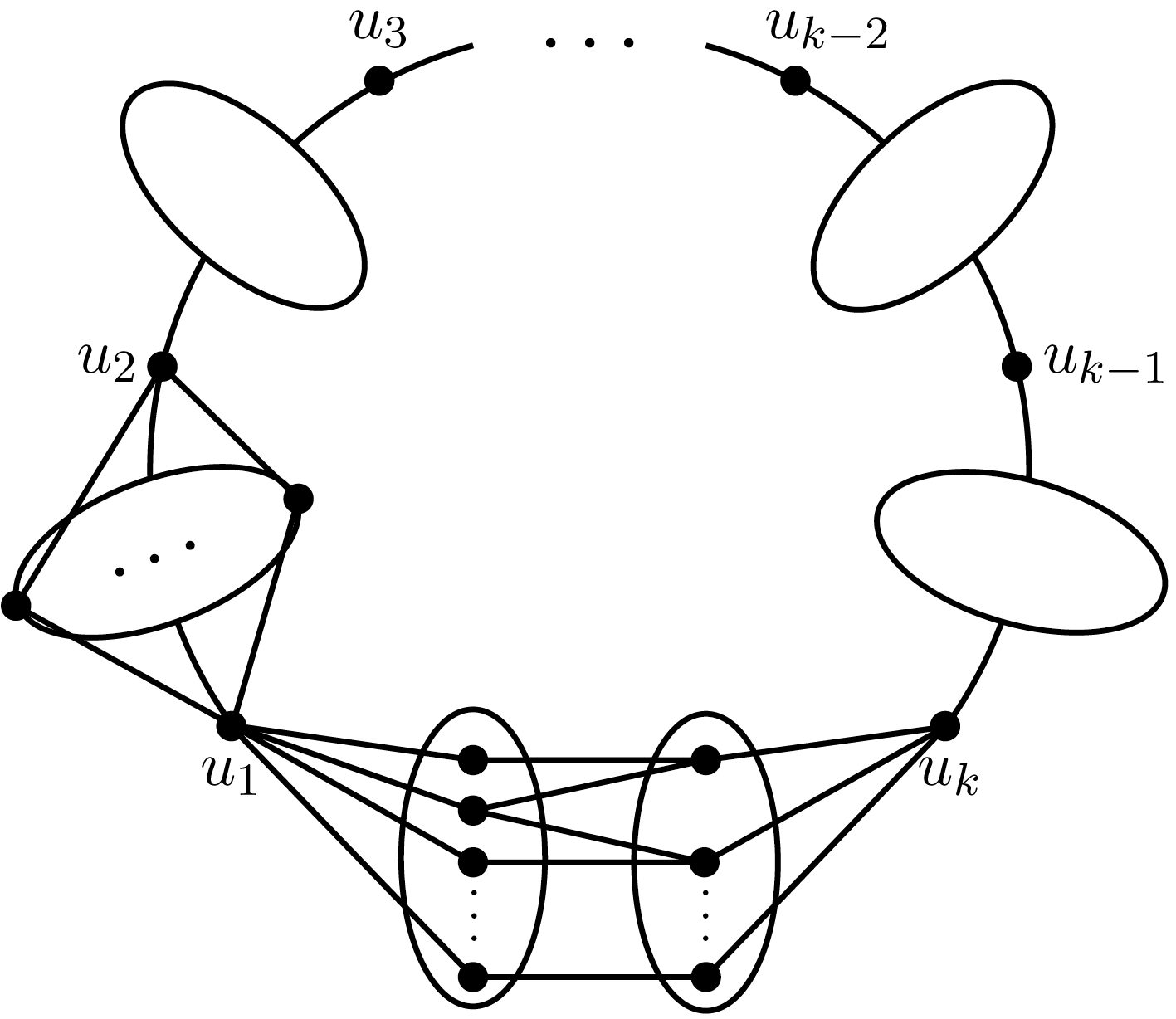}
    \caption{The extremal construction in Theorem~\ref{thm: main}.}\label{fig: extremal construction}
\end{figure}

In the non-planar case, Grzesik and Kielak~\cite{grzesik2019maximum} proved for $k \ge 3$,
\[
    \ex(n, C_{2k+1}, \mathcal{C}^o_{<2k+1}) \le {\left( \frac{n}{2k+1} \right)}^{2k+1},
\]
which is sharp for the blow-up of $C_{2k+1}$.
When all shorter odd cycles are forbidden, every $C_{2k+1}$ is induced.
% % And as a corollary, $\ex(n, C_{2k+1}, C_{2k-1}) = {(n/k)}^k + o(n^k)$.
% This problem is also closely related to the inducibility of odd cycles, since forbidding $\mathcal{C}^o_{<2k+1}$ ensures that all $C_{2k+1}$ are induced.
% Moreover, it proves the asymptotic value of $ind(C_{2k+1})$ if we forbid $C_{2k-1}$.
Both problems are closely related to the induced subgraph problem.
Let $F$ be a fixed graph.
The number of induced copies of $F$ in $G$, denoted by $\mathscr{H}(G,F)$, is the number of subsets $S$ of $V(G)$ such that $G[S] \cong F$.
The maximum number of induced copies of $F$ in an $n$-vertex graph is denoted by $i(n,F):=\max\{ \mathscr{H}(G,F)\mid |V(G)|=n\}.$
The inducibility of $F$ is denoted by $ind(F):=\lim_{n\to \infty}\frac{i(n,F)}{ \binom{n}{|V(F)|}}$.

The topic of inducibility was introduced by Pippenger and Golumbic~\cite{PIPPENGER1975189} in 1975. They proved that $ind(F) \ge \frac{k!}{k^k-k}$, where $k = v(F)$, and conjectured that this bound is tight when $H$ is $C_k$ with $k \ge 5$; that is, $ind(C_k)=\frac{k!}{k^k-k}$ for $k\geq 5$.
They also gave an upper bound that $ind(C_k) \le (2e+o(1))\frac{k!}{k^k}$.
Hefetz and Tyomkyn~\cite{hefetz2018inducibility} improved the extra factor in the upper bound from $2e$ to $128e/81$ in 2018.
The best known upper bound is now $(2+o(1))\frac{k!}{k^k}$, due to Kr\'al, Norin and Volec~\cite{norin2019bound}.
% Recently, thanks to the celebrated flag algebra method provided by Razborov~\cite{razborov2007flag}, there are some important results about the inducibility of small graphs.
Using the flag algebra method introduced by Razborov~\cite{razborov2007flag}, Balogh, J{\'o}zsef, Hu Ping and Lidick\'y~\cite{balogh2016maximum} gave the exact value of $ind(C_5)$.
% Hirst and James~\cite{hirst2014inducibility} calculate some inducibility of $4$-vertex graphs.

% Many studies considered the inducibility problem with a planar host graph. For a given graph $H$, let $f_I(n,H)$ denote the maximum number of copies of induced $H$ in an $n$-vertex planar graph.
There are also studies about the inducibility of cycles in planar graphs.
Let $f_I(n,H)$ denote the maximum number of induced copies of $H$ in an $n$-vertex planar graph.
The asymptotic value of $f_I(n, C_{2k})$ can be obtained as a corollary of the asymptotic value of $f(n,C_{2k})$~\cite{cox2022counting}.
Savery~\cite{savery2024planar} determined $f_I(n,C_4)$.
Ghosh, Gy\H{o}ri, Janzer, Paulos, Salia and Zamora~\cite{ghosh2022maximum} determined the value of $f_I(n, C_5)$ and, independently, Savery~\cite{savery2024planar} determined $f_I(n,C_5)$ and the extremal structures.
Savery~\cite{savery2021planar} later determined $f_I(n,C_6)$.

% Which means, the maximum number of copies of $C_{2k+1}$ in a planar graph on $n$ vertices forbidding $C_3, C_5, \ldots, C_{2k-1}$.

% Note that when forbidding $\mathcal{C}^o_{<2k+1}$, all $C_{2k+1}$ are induced, our result can also be viewed as a topic of the inducibility of odd cycles, which is a continuation of~\cite{ghosh2022maximum,savery2021planar,savery2024planar}.
% It is believed that the maximum number of induced copies of $C_{2k+1}$~($k \ge 3$) in a planar graph on $n$ vertices is $h_k(n)$, which is defined in the following conjecture.
As mentioned in~\cite{ghosh2022maximum,savery2024planar}, the best construction currently known for $f_I(n, C_{2k+1})$ is the following. We conjecture that it is optimal.
\begin{conjecture}\label{conj: induced odd cycle}
    For $k \ge 3$, the maximum number of induced copies of $C_{2k+1}$ in a planar graph on $n$ vertices is $h_k(n)$.
    This value can be attained by evenly blowing up $k$ pairwise non-adjacent vertices in a $C_{2k+1}$.
\end{conjecture}

Observe that our theorem confirms Conjecture~\ref{conj: induced odd cycle} under the additional assumption that all shorter odd cycles are forbidden.
Usually, the planar Tur\'an numbers are only known for small graphs.
Although forbidding all shorter odd cycles is a strong assumption, our result applies to odd cycles of arbitrary length.
Moreover, we obtain an exact value, which is also rare in planar Tur\'an problems.

The paper is organized as follows.
In Section~\ref{sec: preliminary}, we give some preliminaries and an outline of the proof.
In Section~\ref{sec: three properties}, we state Theorem~\ref{thm: new main}, which is sufficient to prove Theorem~\ref{thm: main}.
Section~\ref{sec: main proof} is devoted to proving Theorem~\ref{thm: new main}.

To keep the paper from becoming too long, we omit some routine details, especially in arguments that repeat the ideas of earlier lemmas.
Nevertheless, the proof of our main theorem is still technically involved.
One reason is that planar graphs require careful attention to how edges are embedded without crossings and to whether certain regions are empty.
We suggest that the reader first focus on the main outline of the proof; the details are then easier to follow.

% -------------------------------------------------------------
\section{Preliminaries and Proof Outline}\label{sec: preliminary}

Throughout the paper, we use the notation $O(\cdot)$, $\Omega(\cdot)$ and $\Theta(\cdot)$ in the standard way.
For example, $f(n) = O(g(n))$ means that there exists a constant $C$ such that $f(n) \le C g(n)$ for all $n$ sufficiently large.
Given a graph $G$ and a vertex $v \in V(G)$, we use $N_G(v)$ to denote the neighborhood of $v$, i.e. $N_G(v) = \{u \in V(G) \mid uv \in E(G)\}$.
Denote by $d_G(v)$ the degree of a vertex $v \in V(G)$, i.e. $d_G(v) = |N_G(v)|$.
When there is no ambiguity, we omit the subscript $G$.
For a vertex set $S$, denote by $G[S]$ the subgraph of $G$ induced by $S$.

The notation $a \ll b$ means that $a$ is relatively small compared to $b$.
In the proof, we use several constants $\ell, \epsilon, \epsilon', \xi$.
The relationship among these constants is given by the following inequality.
\begin{equation}\label{eq: relationship between constants}
    1/(2\ell+1) \le \epsilon \ll \epsilon' \ll \xi \ll 1/k.
\end{equation}

We first give rough estimates for $h_k(n)$.
\[
    {\left( \frac{n-2k-1}{k}\right)}^{k} \le h_k(n) \le {\left( \frac{n-k-1}{k}\right)}^{k}.
\]
Moreover,

\begin{equation}\label{eq: hk n - hk n-1}
\begin{aligned}
    h_k(n) - h_k(n-1) &\ge {\left(\frac{n-2k-2}{k} \right)}^{k-1};\\
    h_k(n) - h_k(n-2\ell) &\ge 2\ell {\left(\frac{n - 3\ell}{k}\right)}^{k-1}.
\end{aligned}
\end{equation}

For the rest of the paper, let $f_k(G)$ denote the number of copies of $C_{2k+1}$ in a graph $G$.
We first present a lemma about weighted bipartite forest.

\begin{lemma}\label{lem: bi tree}
    Let $H$ be a weighted bipartite forest with parts $X$ and $Y$.
    Every vertex $v$ has weight $w(v)$.
    Let $w(X) = \sum_{x \in X} w(x)$ and $w(Y) = \sum_{y \in Y} w(y)$.
    Let $w(H) = \sum_{xy \in E(H)} w(x) w(y)$.
    If $\Delta_X = \max_{x \in X} w(x)$ and $\Delta_Y = \max_{y \in Y} w(y)$, then
    \[
    w(H) \le w(X) w(Y) - (w(X) - \Delta_X)(w(Y) - \Delta_Y) = w(X) \Delta_Y + w(Y) \Delta_X - \Delta_X \Delta_Y.
    \]

\end{lemma}

\begin{proof}
    We prove the lemma by induction on the order of $H$.
    The assertion is immediate when $|V(H)| = 1$.
    Now assume $|V(H)| \ge 2$ and that the lemma holds for smaller graphs.
    When $H$ contains an isolated vertex, the lemma is trivial, so we may assume there is no isolated vertex in $H$.
    Since $H$ is a forest, let $v$ be a leaf vertex in $H$, without loss of generality, we may assume $v \in X$ and let $u$ be the neighbor of $v$.
    Let $H'$ be the graph $H-v$ with parts $X' = X\setminus \{v\}$ and $Y'=Y$ and the same weight for each vertex.
    By definition, $\Delta_{X'} \le \Delta_X$ and $\Delta_{Y'} = \Delta_Y$.
    By the induction hypothesis,
    \begin{equation*}
    \begin{aligned}
        w(H) &=  w(v)w(u) + w(H')  \\
        & \le (w(v) \Delta_Y + w(X')\Delta_Y) + (w(Y)\Delta_{X'} - \Delta_{X'}\Delta_Y) \\
        & \le w(X)\Delta_Y + w(Y)\Delta_X - \Delta_X \Delta_Y.
    \end{aligned}
    \end{equation*}
\end{proof}

% \begin{proposition}[\cite{2020arXiv200204579G}]\label{prop: path}
%    $ ex_P(n, P_t, \emptyset) = \Theta(n^{\lfloor (t-1)/2 \rfloor+1})$.
% \end{proposition}

We call a path from $u$ to $w$~(with endpoints $u$ and $w$) a $(u,w)$-path.
% Let ${\{A_i\}}_{1 \le i \le s}$ be a collection of disjoint vertex sets. We say a path is a ${(u,v)}_{A_1,A_2,\ldots,A_s}$-path if it is a $(u,v)$-path of length $s+1$ and its $i$-th vertex is in $A_i$ for all $1 \le i \le s$. Simply denote it by a ${(u,v)}_{\{A_i\}_{1 \le i \le s}}$-path.
For a pair $u,w \in V(G)$, let $p_i(u,w)$ be the number of $(u,w)$-paths of length $i$.
We call $(u,w)$ a $j$-valid pair if $u$ and $w$ are connected by a path of length $j$.
When studying $p_t(u,w)$, we define $A_1,\ldots, A_{t-1}$ as follows:
\begin{equation}\label{eq: Ai}
    A_i = \{v \mid \text{$\exists$ a $(u,w)$-path of length $t$ whose $i$-th vertex is $v$}\}, 1 \le i \le t-1.
\end{equation}
In this case, we call $\{A_1,\ldots, A_{t-1}\}$ a \textbf{$t$-path-decomposition} of $(u,w)$.

We call $v_1v_2\ldots v_t$ a \textbf{circuit} of length $t$ if $v_iv_{i+1}$ is an edge for all $1 \le i \le t-1$ and $v_t v_1$ is also an edge.
Unlike a cycle, a circuit is allowed to have repeated vertices.
The existence of a circuit of odd length $t$ implies the existence of an odd cycle of length at most $t$.

\begin{lemma}\label{lem: Ai Aj disjoint}
    Let $t$ be an integer with $2 \le t \le 2k$.
    Let $(u,w)$ be a $(2k+1-t)$-valid pair in a graph $G$ forbidding $\mathcal{C}^o_{<2k+1}$, and let $\{A_1,\ldots,A_{t-1}\}$ be the $t$-path-decomposition of $(u,w)$.
    Then $A_i \cap A_j = \emptyset$ for all $i \neq j$.
    Therefore, every $(u,w)$-path of length $t$ contains exactly one vertex from each $A_i$.
\end{lemma}

\begin{proof}
    The vertices $u,w$ are connected by a path of length $(2k+1-t)$.
    Assume the path is $u v_1 v_2 \ldots v_{2k-t} w$.

    First, if $i \equiv j \pmod 2$, $i < j$, and $x \in A_i \cap A_j$, then there exists a $(u,x)$-path of length $i$, and a $(x,w)$-path of length $t-j$.
    Combining the path $u v_1 v_2 \ldots v_{2k-t} w$ of length $(2k+1-t)$, we obtain a circuit of length $(2k+1-t) + i + (t - j) = (2k+1) + i-j$, which yields a shorter odd circuit, a contradiction.

    Second, if $i \not\equiv j \pmod 2$, $i < j$, and $x \in A_i \cap A_j$, then there exists a $(u,x)$-path of length $i$, and a $(u,x)$-path of length $j$.
    If $i+j < 2k+1$, then the circuit obtained by combining the two paths has odd length $i+j < 2k+1$, a contradiction.
    Otherwise, $i+j \ge 2k+1$, then there exists a $(x,w)$-path of length $t-i$ and a $(x,w)$-path of length $t-j$.
    Similarly, the circuit obtained by combining the two paths has odd length $2t - (i+j) \le 2t - (2k+1) < 2k+1$, a contradiction.
\end{proof}

\begin{lemma}\label{lem: t-decomp and 2k+1-t-decomp}
    Let $t$ be an integer with $2 \le t \le 2k-1$.
    Let $(u,w)$ be a pair which is both $(2k+1-t)$-valid and $t$-valid in a graph $G$ forbidding $\mathcal{C}^o_{<2k+1}$.
    Let $\{A_1,\ldots,A_{t-1}\}$ be the $t$-path-decomposition of $(u,w)$, and $\{B_1,\ldots,B_{2k-t}\}$ be the $(2k+1-t)$-path-decomposition of $(u,w)$.
    Then $A_i \cap B_j = \emptyset$ for all $i,j$.
\end{lemma}

\begin{proof}
    If there exists $x \in A_i \cap B_j$, then we can find a homomorphic image of a $C_{2k+1}$ in $G$, where the image has two identical vertices.
    Note that a homomorphic image of an odd cycle is a circuit and hence contains an odd cycle.
    This is a contradiction.
\end{proof}

% Let $(u,w)$ be a $2k+1-t$-valid pair, and $\{A_1,\ldots,A_{t-1}\}$ be a $t$-path-decomposition of $(u,w)$.
% Then we have the following observation.
By Lemma~\ref{lem: Ai Aj disjoint}, Lemma~\ref{lem: t-decomp and 2k+1-t-decomp}, and planarity, we have the following observation.

\begin{observation}\label{obs: forest between Ai and Ai+1}
    Let $3 \le t \le 2k$ be a constant and $(u,w)$ be a $(2k+1-t)$-valid pair, and $\{A_1,\ldots,A_{t-1}\}$ be the $t$-path-decomposition of $(u,w)$.
    For each $1 \le i \le t-2$, the edges between $A_i$ and $A_{i+1}$ form a forest.
\end{observation}

Define
\begin{equation}\label{eq: Delta}
\begin{aligned}
    \Delta_i =  \max \{p_i(u,w) \mid (u,w)\text{~is a $(2k+1-i)$-valid pair}\}, i=2,3,\ldots,2k.
\end{aligned}
\end{equation}
It is important to note that $(u,w)$ must be a $(2k+1-i)$-valid pair in the definition of $\Delta_i$.
% \begin{equation}\label{eq: Delta 2 and Delta 3}
% \begin{aligned}
%     &\Delta_2 = \{p_2(u,w) \mid (u,w)\text{~is a $5$-valid-pair}\},\\
%     &\Delta_3 = \{p_3(u,w) \mid (u,w)\text{~is a $4$-valid-pair}\}.
% \end{aligned}
% \end{equation}

\begin{lemma}\label{lem: size lower bound of A}
    Let $G$ be a planar graph without $\mathcal{C}^o_{<2k+1}$ on $n$ vertices.
    Let $(u,w)$ be a $(2k+1-t)$-valid pair in $G$ with $2 \le t \le 2k$.
    Let $\{A_1,\ldots,A_{t-1}\}$ be the $t$-path-decomposition of $(u,w)$ and $A = \bigcup_{i=1}^{t-1}A_i$.
    Then
    \[
        |A| \ge \left\lceil \frac{t-1}{2} \right\rceil {{p_t(u,w)}}^{1 / \lceil (t-1)/2 \rceil}.
    \]

    Consequently, for every integer $t$ with $2 \le t \le 2k$,
    \[
        {{p_t(u,w)}} \le {|A|}^{\lceil \frac{t-1}{2} \rceil}.
    \]
\end{lemma}

\begin{proof}

    Let $E_i$ be the set of edges between $A_{2i-1}, A_{2i}$, $1 \le i \le \lfloor (t-1)/2 \rfloor$.
    Since $(u,w)$ is $(2k+1-t)$-valid, there exists a $(u,w)$-path of length $(2k+1-t)$ avoiding $A$ by Lemma~\ref{lem: t-decomp and 2k+1-t-decomp}.
    By Observation~\ref{obs: forest between Ai and Ai+1}, $E_i$ forms a forest.
    Therefore, $|E_i| \le |A_{2i-1}| + |A_{2i}|$.

    \textbf{Case 1: $t$ is odd}.
    $p_t(u,w)$ is at most $\prod_{i=1}^{(t-1)/2}|E_i|$.
    Therefore, we have
    \[
    \sum_{i=1}^{(t-1)/2}|E_i| \ge  \frac{t-1}{2} {\left(\prod_{i=1}^{(t-1)/2}|E_i|\right)}^{1/((t-1)/2)} \ge \frac{t-1}{2} {p_t(u,w)}^{2/(t-1)}.
    \]

    \textbf{Case 2: $t$ is even}.
    $p_t(u,w)$ is at most $\left( \prod_{i=1}^{\lfloor (t-1)/2 \rfloor}|E_i| \right)  |A_{t-1}|$.
    Similarly, we have
    \[
    \sum_{i=1}^{t/2 - 1}|E_i| + |A_{t-1}| \ge  \frac{t}{2}{\left( \left( \prod_{i=1}^{\lfloor (t-1)/2 \rfloor}|E_i| \right)  |A_{t-1}| \right)}^{2/t} \ge \frac{t}{2} {{p_t(u,w)}}^{2/t}.
    \]

    Consequently,
    \[
    |A| = \sum_{i=1}^{t-1}|A_i| \ge \left\lceil \frac{t-1}{2} \right\rceil {{p_t(u,w)}}^{1 / \lceil (t-1)/2 \rceil}.
    \]
    % Note that when $t=1$, $p_t(u,w) \le 1 = |A|^{\lceil (t-1)/2 \rceil}$ which is trivial.
\end{proof}

Suppose there exists a $(2k-2)$-valid pair $(u_1,u_2)$, and let $\{A_1,A_2\}$ be the $3$-path-decomposition of $(u_1,u_2)$.
If there exists a path $x_{1}y_{1} \ldots x_{3\ell} y_{3\ell}$ such that $x_i \in A_1$ and $y_i \in A_2$ for all $1 \le i \le 3\ell$ and every $x_i,y_i$ has degree three, then we call the path $u_1x_{\ell} y_{\ell} u_2$ a \textbf{special path}.

We define three properties of $\mathcal{C}^o_{<2k+1}$-free planar graphs.
In Section~\ref{sec: three properties}, we prove that these properties may be assumed throughout the main proof.

\begin{enumerate}[label=(\arabic*)]
    \item For every vertex $v$, the number of copies of $C_{2k+1}$ containing $v$ is at least $h_k(n)-h_k(n-1)$.
    Moreover, for every set of $2\ell$ vertices $v_1,\ldots,v_{2\ell}$, the number of copies of $C_{2k+1}$ containing at least one of $v_1,\ldots,v_{2\ell}$ is at least $h_k(n)-h_k(n-2\ell)$.
    \item If $v_1$ and $v_2$ are two vertices of degree two in $G$ with different neighborhoods, then they must be contained in a common $C_{2k+1}$.
    \item Let $v_1$ and $v_2$ be two vertices of degree two in $G$ with the same neighbors, and let $(u_1,u_2)$ be a $(2k-2)$-valid pair in $G$.
    % Let $\{A_1,A_2\}$ be a $3$-path-decomposition of $(u_1,u_2)$.
    If there exists a special path $u_1x_{\ell} y_{\ell} u_2$, then $v_1$ and the path $u_1x_{\ell}y_{\ell} u_2$ must be contained in some common $C_{2k+1}$.
    % ~(the definition of a special path will be given later in this section).
    % $x_{1}y_{1} \ldots x_{3\ell} y_{3\ell}$ such that $x_i\in A_1$ and $y_i \in A_2$ for every $1 \le i \le 3\ell$ and every $x_i,y_i$ has degree three.
    % Then $v_1$, $u_1x_{\ell},y_{\ell} u_2$ must be contained in some common $C_{2k+1}$.
\end{enumerate}

The extremal graph has many vertices of degree two, which are very useful in our proof.
Indeed, once we find a vertex $u$ of degree two with neighbors $u_1,u_2$, property (1) gives $p_{2k-1}(u_1,u_2) \ge {((n-2k-2)/k)}^{k-1}$.
A large value of $p_{2k-1}(u_1,u_2)$ is precisely what we need.

% A large $p_5(u_1,u_2)$ is what we really need.
We will prove in Corollary~\ref{cor: large common neighbor} that if $u$ and $w$ have many common neighbors~($p_2(u,w) \ge 800k^{k}$), then one of these common neighbors has degree two.
Therefore, we hope to find such $u$ and $w$ with a large $p_2(u,w)$, that is, with a large $2$-path-decomposition.

However, it is not always possible to find such a structure.
The following lemma shows that if $u$ and $w$ satisfy suitable hypotheses relative to their $3$-path-decomposition (rather than having a large $2$-path-decomposition), then we can still obtain a comparable lower bound on $p_{2k-2}(u,w)$ by appealing to property~(1), at the expense of an $\epsilon$ factor.

\begin{lemma}\label{lem: with long path}
    Let $G$ be a planar graph forbidding $\mathcal{C}^o_{<2k+1}$ on $n$ vertices with property (1).
    Let $(u_1,u_2)$ be a $(2k-2)$-valid pair in $G$.
    Let $\{A_1,A_2\}$ be the $3$-path-decomposition of $(u_1,u_2)$.
    If there exists a path $x_{1}y_{1} \ldots x_{3\ell} y_{3\ell}$ such that $x_i \in A_1$ and $y_i \in A_2$ for all $1 \le i \le 3\ell$, and every $x_i,y_i$ is of degree three~($u_1x_{\ell} y_{\ell} u_2$ is a special path), then the number of $C_{2k+1}$ containing $x_\ell$~(or $y_{\ell}$) is exactly $2p_{2k-2}(u_1,u_2)$, and
    \[
        p_{2k-2}(u_1,u_2) \ge \left(1-\frac{1}{2\ell+1}\right) {\left( \frac{n-3\ell}{k} \right)}^{k-1} \ge (1-\epsilon){\left( \frac{n-3\ell}{k} \right)}^{k-1}.
    \]
\end{lemma}

\begin{proof}

    \begin{figure}
        \centering
        \includegraphics[width=0.4\textwidth]{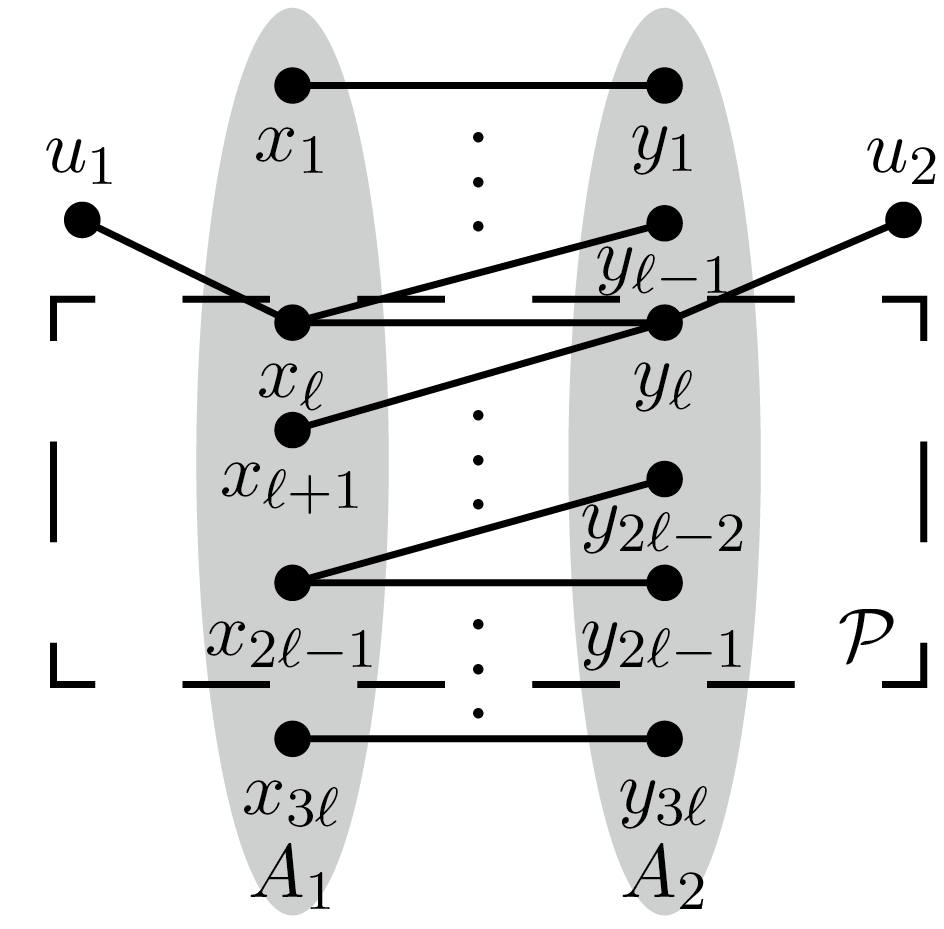}
        \caption{The path from $x_{1}$ to $y_{3\ell}$.}\label{fig: long path}
    \end{figure}

    Let $\mathcal{P} = \{x_{\ell}, y_{\ell}, \ldots, x_{2\ell-1}, y_{2\ell-1}\}$~(see Figure~\ref{fig: long path}).
    We now count the copies of $C_{2k+1}$ containing at least one vertex in $\mathcal{P}$.

    First consider the copies of $C_{2k+1}$ containing $x_\ell$.
    There are three possibilities as follows.
    \begin{enumerate}
        \item The cycle containing $u_1,x_{\ell},y_{\ell-1}$.
        \item The cycle containing $u_1,x_{\ell},y_{\ell}$.
        \item The cycle containing $x_{\ell},y_{\ell-1},y_{\ell}$.
    \end{enumerate}

    There is no cycle of the third type.
    Indeed, a cycle of the third type cannot contain $u_1$ and $u_2$, since it is induced.
    Then the neighbor of $y_{\ell}$ in the cycle must be $x_{\ell+1}$ and $x_{\ell}$.
    Repeating this argument for $x_{\ell+1}$, we eventually derive that all vertices in the path $\mathcal{P}$ must be contained in the cycle, a contradiction since $2k+1 \ll \ell$.

    Next consider a cycle of the first type.
    The cycle must contain $u_2$, otherwise, it must contain $x_{\ell-1}$ and then the cycle is not induced, a contradiction.
    Also for the second type, the cycle must contain $u_2$.
    If such a cycle contains both $u_1$ and $u_2$, it cannot contain other vertices in $\mathcal{P}$.

    The cycle containing $x_{\ell}$ must be of the form $u_1,x_{\ell},y_{\ell-1},u_2$~(or $u_1,x_{\ell},y_{\ell},u_2$) and a $(u_1,u_2)$-path of length $(2k-2)$.
    Hence, the number of $C_{2k+1}$ containing $x_\ell$~(or equivalently $y_\ell$) is $2p_{2k-2}(u_1, u_2)$.

    The same argument applies to each vertex in $\mathcal{P}$.
    Thus every cycle containing at least one vertex in $\mathcal{P}$ must contain exactly one edge in $\{y_{\ell-1}x_\ell, x_{\ell}y_{\ell}, \ldots, x_{2\ell-1}y_{2\ell-1}, y_{2\ell-1}x_{2\ell}\}$ and a $(u_1,u_2)$-path of length $2k-2$. 
    Therefore, the number of cycles containing at least one vertex in $\mathcal{P}$ is exactly $(2\ell+1) p_{2k-2}(u_1,u_2) \ge 2\ell {\left(\frac{n - 3\ell}{k}\right)}^{k-1}$.
    It follows that $p_{2k-2}(u_1,u_2) \ge (1-\frac{1}{2\ell+1}) {\left(\frac{n - 3\ell}{k}\right)}^{k-1}$.
\end{proof}

% That is, if $u_1 x_{\ell} y_{\ell} u_2$ is a special path, then there exists a path $x_{1}y_{1} \ldots x_{3\ell} y_{3\ell}$ such that $x_i \in A_1$ and $y_i \in A_2$ for all $1 \le i \le 3\ell$ and every $x_i,y_i$ is of degree three where $\{A_1,A_2\}$ is the $3$-path-decomposition of $(u_1,u_2)$ and $(u_1,u_2)$ is $(2k-2)$-valid.

We now have the necessary definitions and basic lemmas.
The proof of Theorem~\ref{thm: main} proceeds as follows.
\begin{enumerate}
    \item In Section~\ref{sec: three properties}, we first state a sufficient proposition. In proving this proposition, we may assume that $G$ has properties (1), (2) and (3).
    \item As mentioned above, the key is to find vertices of degree two or special paths.
    In Section~\ref{subsec: faces}, we present lemmas that help us analyze regions and thereby find vertices of degree two or special paths.
    \item In Section~\ref{subsec: special path or vertex of degree two}, we prove lemmas that find vertices of degree two or special paths.
    \item In Section~\ref{subsec: prep Delta 2 Delta 3} and~\ref{subsec: large Delta2}, we prove that $\Delta_2$ is large enough~(at least $\frac{1}{k}n - \epsilon' n$).
    \item Once $\Delta_2$ is large enough, we can find a subgraph $G'$ of $G$ on at least $n - k\xi n$ vertices that resembles the extremal graph. Thus only $k\xi n$ vertices remain undetermined, where $\xi$ is a small constant. Finally, in Section~\ref{subsec: Bootstrap}, we show that the remaining $k\xi n$ vertices also fit the extremal structure.
\end{enumerate}

% ------------------------------------
\section{Three properties}\label{sec: three properties}

In this section, we present a proposition, that shows that, in the proof, we may assume that $G$ has properties (1), (2) and (3).

If $G$ does not have property (1), then a natural idea is to delete vertices that violate property (1) and argue on the remaining graph.
If $G$ does not satisfy property (2) or (3), then we can perform operations that do not decrease the number of copies of $C_{2k+1}$ and eventually obtain a graph satisfying property (2) and (3).

Let $G$ be a graph on $n$ vertices forbidding $\mathcal{C}^o_{< 2k+1}$ that does not satisfy property (2).
Then there exist two vertices $v_1$ and $v_2$ of degree two with different neighborhoods such that $v_1$ and $v_2$ are not contained together in any $C_{2k+1}$.
Consider two operations: deleting $v_1$ and blowing up $v_2$, or deleting $v_2$ and blowing up $v_1$.
These operations preserve planarity, since the vertices involved have degree two.
If one of the operations strictly increases the number of $C_{2k+1}$, then we obtain a graph $G'$ with more $C_{2k+1}$ than $G$.
Otherwise, $v_1$ and $v_2$ are contained in the same number of $C_{2k+1}$.
In this case, we call such $v_1$ and $v_2$ a \textbf{symmetric pair}.
If there is a vertex of degree two adjacent to $v_1$ (~or equivalently $v_2$), then after deleting $v_1$ and blowing up $v_2$, we have a graph $G'$ with $f_{k}(G') \ge f_k(G)$ containing a degree-one vertex.
If a graph contains a degree-one vertex, then it cannot have property (1).
Otherwise, if no vertex of degree two is adjacent to $v_1$ or $v_2$, then deleting one of $v_1,v_2$ and blowing up the other does not decrease the number of vertices of degree two.
Assume we have an ordering of the vertices in $G$, $\phi: V(G) \rightarrow [n]$.
We choose $v_1,v_2$ to be a symmetric pair for which $\max\{\phi(v_1),\phi(v_2)\}$ is maximized.
Then we always delete the vertex with smaller index, say $v_1$, and blow up the other vertex $v_2$~(the new vertex has index $\phi(v_1)$).
Define
\[
    q(G) = \max \{ \phi(v) \mid v \in V(G),  \text{$v$ is in a symmetric pair}\}.
\]
After one such operation, either $q(G)$ decreases, or the number of symmetric pairs containing $\phi^{-1}(q(G))$ decreases.

% Now we define Procedure~\ref{alg: procedure 1} to transform $G$ to a graph with property (2).

% \begin{algorithm}[ht]
% \caption{Procedure to obtain property (2)}
% \label{alg: procedure 1}
% \begin{algorithmic}[1]
%     \State Start with $G_0 = G$.
%     \While {If $G_i$ does not have property (2)}
%         \State There exists two vertices of degree two $v_1$ and $v_2$ with different neighbor set, and $v_1,v_2$ are not contained together in any $C_{2k+1}$.
%         \State If deleting one vertex and blowing up the other vertex strictly increases the number of $C_{2k+1}$, then do the operation to obtain a graph $G_{i+1}$.
%         \State Otherwise, let $v_1,v_2$ be a symmetric pair such that $\max\{v_1,v_2\} = q(G_i)$.
%         Then delete the vertex with smaller index and blow up the other vertex to obtain a graph $G_{i+1}$.
%     \EndWhile
% \end{algorithmic}
% \end{algorithm}

We now define a procedure that transforms $G$ either into a graph with property (2) or into a graph without property (1).
We start with $G_0 = G$.

\begin{enumerate}[label=Step \arabic*:]
    \item If $G_i$ has property (2) or does not have property (1), then stop. Otherwise, there exist two vertices $v_1$ and $v_2$ of degree two with different neighborhoods, and $v_1,v_2$ are not contained together in any $C_{2k+1}$.
    \item If $v_1$ and $v_2$ do not form a symmetric pair, then deleting one vertex and blowing up the other strictly increases the number of copies of $C_{2k+1}$. Perform this operation to obtain a graph $G_{i+1}$ with more copies of $C_{2k+1}$, and return to Step 1.
    \item Otherwise, let $v_1,v_2$ be a symmetric pair such that $\max\{\phi(v_1),\phi(v_2)\} = q(G_i)$. If there is a vertex of degree two adjacent to $v_1$ (or equivalently $v_2$), then we can delete $v_1$ and blow up $v_2$ to obtain a graph $G_{i+1}$ with $f_{k}(G_{i+1}) \ge f_{k}(G_i)$ containing a degree-one vertex.
    Clearly, $G_{i+1}$ does not have property (1).
    Then return to Step 1.
    \item Otherwise, delete the vertex with smaller index~(say $v_1$) and blow up the other vertex to obtain a graph $G_{i+1}$~(the new vertex has index $\phi(v_1)$).
    In this case, the number of vertices of degree two does not decrease.
    Then return to Step 1.
\end{enumerate}

Throughout the process, the graph remains planar and $\mathcal{C}^o_{<2k+1}$-free.
In Step 2, the number of copies of $C_{2k+1}$ strictly increases.
In Step 3, we obtain a graph without property (1), and then stop.
In Step 4, either $q(G_{i+1}) < q(G_i)$, or the number of symmetric pairs containing $\phi^{-1}(q(G_i))$ strictly decreases.
Therefore, the procedure eventually stops at some graph $G_m$.
The following lemma follows directly from this procedure.

\begin{lemma}\label{lem: procedure 1}
    Let $G$ be a planar graph on $n$ vertices forbidding $\mathcal{C}^o_{<2k+1}$ that does not satisfy property (2).
    Then there exists a planar graph $G'$ on $n$ vertices forbidding $\mathcal{C}^o_{<2k+1}$ with $f_{k}(G') \ge f_{k}(G)$.
    Moreover, $G'$ has the following property.
    % property (2) whose number of vertices of degree two does not decrease, or has a degree-one vertex and thus does not satisfy property (1).
    \begin{enumerate}
        \item Either $G'$ does not have property (1), $f_k(G') \ge f_k(G)$;
        \item or $G'$ has property (2), $f_k(G') > f_k(G)$;
        \item or $G'$ has property (2), $f_k(G') = f_k(G)$, and the number of vertices of degree two in $G'$ is at least that in $G$.
    \end{enumerate}
\end{lemma}

    If $G$ has properties (1) and (2) but not property (3), then there exist two vertices $v_1$ and $v_2$ of degree two with the same neighbors, and a $(2k-2)$-valid pair $(u_1,u_2)$ such that there is a special path $u_1 x_{\ell} y_{\ell} u_2$, but $v_1$ and $u_1 x_{\ell} y_{\ell} u_2$ are not contained in any common $C_{2k+1}$.

By Lemma~\ref{lem: with long path}, the number of $C_{2k+1}$ containing $x_{\ell}$~(or $y_\ell$) is exactly $2 p_{2k-2}(u_1,u_2)$.
Assume the two neighbors of $v_1$ and $v_2$ are $w_1$ and $w_2$~(see Figure~\ref{fig: operation for property 3}).
If $p_{2k-1}(w_1,w_2) < p_{2k-2}(u_1,u_2)$, then we obtain a new graph $G'$ by deleting $v_1,v_2$, deleting the edge $y_{\ell}x_{\ell+1}$, and adding two new vertices $x$ and $y$.
Then add the edges $y_{\ell}x, xy, yx_{\ell+1},u_1x, yu_2$.
The number of copies of $C_{2k+1}$ in $G'$ is the number in $G$ minus $2 p_{2k-1}(w_1,w_2)$ plus $2 p_{2k-2}(u_1,u_2)$, and hence is larger.

If $p_{2k-1}(w_1,w_2) \ge p_{2k-2}(u_1,u_2)$, then we obtain a new graph $G''$ by deleting $x_{\ell},y_{\ell}$, adding the edge $y_{\ell-1}x_{\ell+1}$, and adding two new copies of $v_1$.
The number of copies of $C_{2k+1}$ in $G''$ is the number in $G$ minus $2 p_{2k-2}(u_1,u_2)$ plus $2 p_{2k-1}(w_1,w_2)$, which is larger when $p_{2k-1}(w_1,w_2) > p_{2k-2}(u_1,u_2)$~(see Figure~\ref{fig: operation for property 3}).
Since $G$ has property (1), $w_1,w_2$ must not be vertices of degree two.
Thus, when $p_{2k-1}(w_1,w_2) = p_{2k-2}(u_1,u_2)$, the number of copies of $C_{2k+1}$ remains the same, but the number of vertices of degree two strictly increases.

\begin{figure}
    \centering
    \includegraphics[width=\textwidth]{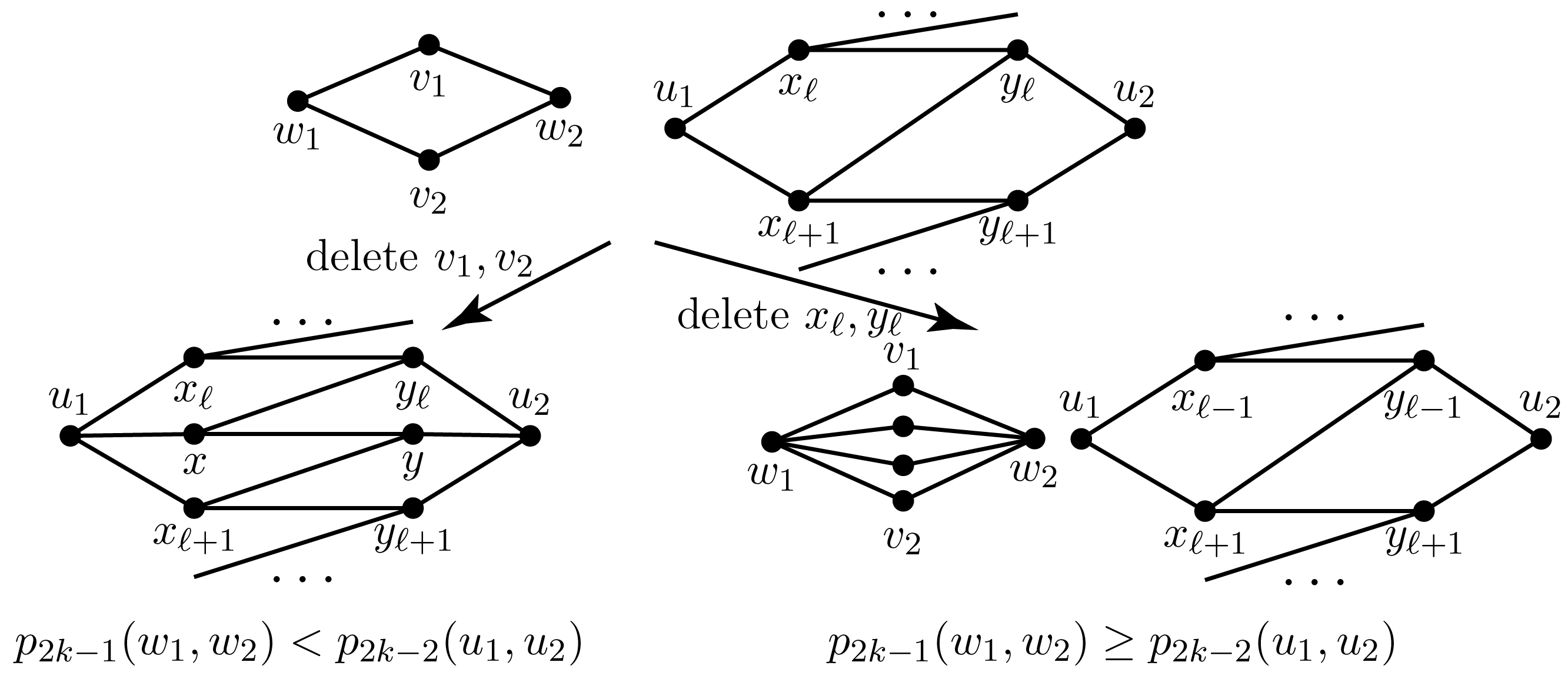}
    \caption{The operations when $G$ does not have property (3).}\label{fig: operation for property 3}
\end{figure}

Throughout the process, the graph remains planar and $\mathcal{C}^o_{<2k+1}$-free.
In both cases, we obtain either a graph with more copies of $C_{2k+1}$, or a graph with the same number of copies but more vertices of degree two.
We can continue this procedure until we obtain a graph with property (3), provided that the graph maintains properties (1) and (2) throughout the process, as stated in the following lemma.

\begin{lemma}\label{lem: procedure 2}
    Let $G$ be a planar graph on $n$ vertices forbidding $\mathcal{C}^o_{<2k+1}$ with property (1) and (2) but not property (3).
    Then there exists a planar graph $G'$ on $n$ vertices forbidding $\mathcal{C}^o_{<2k+1}$ with $f_{k}(G') \ge f_{k}(G)$ and one of the following properties.
    \begin{enumerate}
        \item either $f_{k}(G') > f_{k}(G)$ and $G'$ does not satisfy property (1) or (2);
        \item or $f_{k}(G') = f_{k}(G)$, $G'$ has more vertices of degree two than $G$, and $G'$ does not satisfy property (1) or (2);
        \item or $G'$ has property (1), (2) and (3).
    \end{enumerate}
\end{lemma}

We are now ready to state the main theorem.

\begin{theorem}\label{thm: new main}
    There exists $n_0 = n_0(k)$ such that every planar graph $G$ on $n \ge n_0$ vertices forbidding $\mathcal{C}^o_{<2k+1}$ and satisfying properties (1), (2) and (3) contains at most $h_k(n)$ copies of $C_{2k+1}$.
\end{theorem}

We claim that Theorem~\ref{thm: new main} is sufficient to prove Theorem~\ref{thm: main}.

\noindent
\textbf{Proof of Theorem~\ref{thm: main} assuming Theorem~\ref{thm: new main}.}
    Define $n_0' = n_0 + 2 \ell n_0^{2k+1}$, where $n_0$ is the constant in Theorem~\ref{thm: new main}.
    Let $G$ be a planar graph on $n \ge n_0'$ vertices forbidding $\mathcal{C}^o_{<2k+1}$ and containing the maximum possible number of copies of $C_{2k+1}$.

    We define a procedure starting from $G_n = G$.
    Here the subscript always denotes the number of vertices in the graph.
    % Here we use the subscript to denote the number of vertices in the graph.

    \begin{enumerate}[label=Step \arabic*:]
        \item If $G_i$ has property (1), (2) and (3) or $G_i$ has at most $n_0$ vertices, then stop.
        \item If $G_i$ does not have property (1).
        If there exists a vertex $v$ that violates property (1), then delete $v$ to obtain $G_{i-1}$; in this case, $f_k(G_i) \le f_k(G_{i-1}) + h_k(i) - h_k(i-1) -1$.
        If there exist $2\ell$ vertices $v_1,\ldots,v_{2\ell}$ that violate property (1), then delete $v_1,\ldots,v_{2\ell}$ to obtain $G_{i-2\ell}$; in this case, $f_k(G_i) \le f_k(G_{i-2\ell}) + h_k(i) - h_k(i-2\ell) -1$.
        Then return to Step 1.
        \item If $G_i$ has property (1) but not property (2), then update $G_i$ by the graph obtained from Lemma~\ref{lem: procedure 1}.
        Then return to Step 1.
        % \item If $G_i$ have property (1) but not property (2), then by the previous analysis, we can do some operations to update $G_{i}$ such that the result graph has property (2) without decreasing the number of $C_{2k+1}$.
        \item If $G_i$ has property (1) and (2) but not property (3), then update $G_i$ by the graph obtained from Lemma~\ref{lem: procedure 2}.
        Then return to Step 1.
        % \item If $G_i$ have property (1) and (2) but not property (3), then by the previous analysis, we can do some operations to update $G_i$ either increasing the number of $C_{2k+1}$ or increasing the number of vertices of degree two.
        % Then go back to Step 1.
    \end{enumerate}

    The number of vertices decreases only in Step 2, because of a violation of property (1).
    In Step 3, the updated $G_i$ either has property (1) and (2), or does not satisfy property (1).
    % In the latter case, the number of vertices will decrease due to Step 2.
    In Step 4, either $f_k(G_i)$ strictly increases, or the number of vertices of degree two strictly increases, or $G_i$ has property (1), (2) and (3).
    Hence the procedure eventually stops.
    Suppose the procedure passes through $G_{n_1}, G_{n_2}, \ldots, G_{n_m}$, where $n = n_1 > n_2 > \cdots > n_m$, and stops at $G_{n_m}$.

    If $G_{n_m}$ has property (1), (2) and (3) and at least $n_0$ vertices, then by Theorem~\ref{thm: new main} we have
    \begin{equation*}
    \begin{aligned}
        f_k(G) & \le \sum_{i=1}^{m-1} \left( h_k(n_i) - h_k(n_{i+1}) \right) + f_k(G_{n_m})\\
        & \le h_k(n_1).
    \end{aligned}
    \end{equation*}
    If $G_{n_m}$ has at most $n_0$ vertices, then
    \begin{equation*}
    \begin{aligned}
        f_k(G) & \le \sum_{i=1}^{m-1} \left( h_k(n_i) - h_k(n_{i+1}) -1 \right) + f_k(G_{n_m})\\
        & \le h_k(n_1) - \frac{n-n_0}{2\ell} + f_k(G_{n_m})\\
        & \le h_k(n_1) - n_0^{2k+1} + n_0^{2k+1} \le h_k(n).
    \end{aligned}
    \end{equation*}
    \hfill$\square$\par

    The rest of the paper is devoted to the proof of Theorem~\ref{thm: new main}.

% -------------------------------------------------------------
\section{Proof of Theorem~\ref{thm: new main}}\label{sec: main proof}

Since we work with planar graphs, we fix a plane embedding of the graph throughout the proof.

% -------------------------------------------------------------
\subsection{Find empty regions and vertices of degree two}\label{subsec: faces}

We say that a region is \textbf{empty} if it contains no vertex in its interior.

\begin{lemma}\label{lem: 4-cycle empty}
    Let $G$ be a planar graph on $n$ vertices forbidding $\mathcal{C}^o_{<2k+1}$ with property (1).
    Let $(u_1,u_2)$ be a $(2k-1)$-valid pair, and $v_1,v_2$ be two consecutive distinct vertices in $N(u_1) \cap N(u_2)$~(see Figure~\ref{fig: 4- and 6-cycle empty}).
    Denote by $R$ the region with boundary $u_1 v_1 u_2 v_2 u_1$.
    That is, no vertex from $N(u_1) \cap N(u_2)$ lies inside $R$.
    Then there exists a constant $n_1 = n_1(k)$ such that, if $n \ge n_1$ and $R$ is not empty, then $R$ contains at least $\frac{1}{400k^{k}} n$ vertices.
\end{lemma}

\begin{proof}
    Let $\delta = \frac{1}{400k^{k}}$.
    Suppose otherwise; then there are at most $\delta n$ vertices in $R$.
    Let $z$ be a vertex in $R$ with minimum degree in the planar graph induced by the vertices in the interior of $R$.
    Then $d(z) \le 9$~(including the vertices on the boundary).
    By property (1) and the pigeonhole principle, there exist $z_1, z_2 \in N(z)$ such that the number of $C_{2k+1}$ containing $z,z_1,z_2$ is at least $\frac{{{((n-2k-2)/k)}^{k-1}}}{\binom{9}{2}}$.

    We first consider cycles contained in $R$~(including the boundary).
    The number of $(z_1,z_2)$-paths of length $2k-1$ contained in $R$ is at most ${(\delta n+4)}^{k-1}$ by Lemma~\ref{lem: size lower bound of A}~(we do not apply Lemma~\ref{lem: size lower bound of A} to $G$ directly, but to the graph induced by the vertices contained in $R$).

    Thus most copies of $C_{2k+1}$ containing $z,z_1,z_2$ must contain at least one vertex outside $R$.
    In this case, the cycle crosses the boundary of $R$ at least twice.
    There are only two possibilities since the cycle is induced: the cycle contains both $u_1$ and $u_2$, or the cycle contains both $v_1$ and $v_2$.

    For the first case, let $q$ be the length of $(u_1,u_2)$-path in the cycle without $z$.
    First, $0 < q \le 2k-1$ since $u_1u_2$ is not an edge.
    If $q = 2k-1$, then $\{z_1,z_2\} = \{u_1,u_2\}$ which violates our assumption that $z \notin N(u_1) \cap N(u_2)$.
    Then $q$ must be even, otherwise by combining $u_1v_1u_2$, we have a shorter odd circuit, a contradiction.
    We conclude that $q$ must be two; otherwise the cycle contains a $(u_1,u_2)$-path of odd length at most $2k-3$, which together with $u_1v_2u_2$ gives a shorter odd circuit, a contradiction.
    Therefore, the cycle contains $u_1wu_2$ for some $w \in N(u_1) \cap N(u_2)$ outside $R$, a $(u_1,z_1)$-path of length $p$ and a $(u_2,z_2)$-path of length $(2k-3-p)$~(or switching $z_1$ and $z_2$) for some $p \in \{0,1,\ldots,2k-3\}$.
    For the vertex $w$, there are at most $n$ choices.
    For the $(u_1,z_1)$-path of length $p$, there are ${(\delta n)}^{\lceil (p-1)/2 \rceil}$ choices by Lemma~\ref{lem: size lower bound of A}.
    For the $(u_2,z_2)$-path of length $(2k-3-p)$, there are ${(\delta n)}^{\lceil (2k-4-p)/2 \rceil}$ choices.
    In total, the number of such cycles is at most
    \begin{equation*}
    \begin{aligned}
        2 \sum_{p=0}^{2k-3} n \cdot {(\delta n)}^{\lceil (p-1)/2 \rceil} \cdot {(\delta n)}^{\lceil (2k-4-p)/2 \rceil} \le 8k\delta n^{k-1}.
    \end{aligned}
    \end{equation*}

    For the second case when the cycle contains both $v_1$ and $v_2$.
    The distance between $v_1$ and $v_2$ in the cycle must be two.
    Otherwise, combining it with $v_1u_1v_2$ gives a shorter odd circuit.
    Recall that there exists a $(u_1,u_2)$-path of length $2k-1$ denoted by $P$.
    Due to the existence of $P$, the cycle must contain a $(v_1,v_2)$-path of length $2k-1$ which involves a vertex outside $R$, and the two neighbors of $z$ on the cycle are $v_1$ and $v_2$ (equivalently, the cycle contains the path $v_1 z v_2$).
    Then the cycle contains at least one vertex $w$ in $P$ other than $u_1$ and $u_2$.

    Thus there is a $(v_1,w)$-path of length $p$ and a $(v_2,w)$-path of length $(2k-1-p)$ for some positive integer $p$.
    By definition, there exist $(v_1,w)$-paths~(and $(v_2,w)$-paths) of lengths $q$ and $(2k+1-q)$.

    Combining two of the $(v_1,w)$-paths~(or $(v_2,w)$-paths), we have four circuits with length $p+q$, $p+(2k+1-q)$, $(2k-1-p)+q$, $(2k-1-p)+(2k+1-q)$, respectively.
    If $p+q$ is odd, then one of $(p+q)$ and $(2k-1-p)+(2k+1-q)$ is an odd number smaller than $2k+1$, a contradiction.
    If $p+q$ is even, then one of $(2k-1-p+q)$ and $(2k+1-q+p)$ is an odd number smaller than $2k+1$, a contradiction.
    Therefore, there is no such cycle containing both $v_1$ and $v_2$.

    Therefore, the number of copies of $C_{2k+1}$ containing $z,z_1,z_2$ is at most
    \[
        10k\delta n^{k-1} < {((n-2k-2)/k)}^{k-1}/\binom{9}{2},
    \]
     when $n_1$ is large enough, a contradiction.
\end{proof}

\begin{figure}
    \centering
    \includegraphics[width=0.8\textwidth]{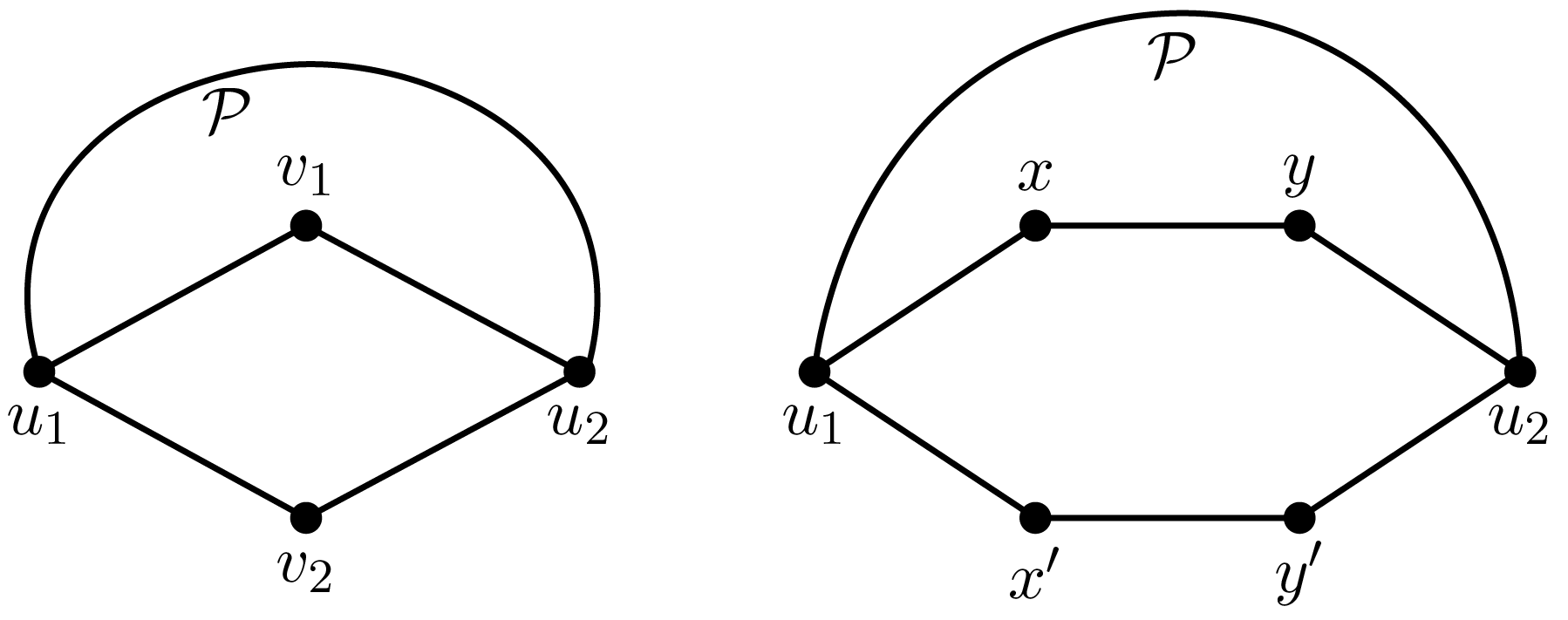}
    \caption{The cycle $u_1 v_1 u_2 v_2$ in Lemma~\ref{lem: 4-cycle empty} and the cycle $u_1 xy u_2 y' x' u_1$ in Lemma~\ref{lem: 6-cycle empty}.}~\label{fig: 4- and 6-cycle empty}
\end{figure}

As a corollary, we obtain the following.

\begin{corollary}\label{cor: large common neighbor}
    Let $G$ be a planar graph on $n$ vertices forbidding $\mathcal{C}^o_{<2k+1}$ with property (1).
    There exists $n_1 = n_1(k)$ such that for all $n \ge n_1$, the following holds.
    Let $(u_1,u_2)$ be a $(2k-1)$-valid pair in $G$.
    If $|N(u_1) \cap N(u_2)| \ge 800k^{k}$, then there exists
    a vertex of degree two in $N(u_1) \cap N(u_2)$.
    % consecutive $v_1,v_2,\ldots,v_{2k+1} \in N(u_1) \cap N(u_2)$ such that there is no vertex in the region $R_1$ with boundary $u_1v_{k}u_2v_{k+1}u_1$ and $R_2$ with boundary $u_1v_{k+1}u_2v_{k+2}u_1$.
    % Therefore, $v_{k+1}$ is a vertex of degree two and by property (1), $p_{2k-1}(u_1,u_2) \ge {((n-2k-2)/k)}^{k-1}$.
\end{corollary}

\begin{proof}
    This follows directly from Lemma~\ref{lem: 4-cycle empty}.
\end{proof}

\begin{lemma}\label{lem: 6-cycle empty}
    Let $G$ be a planar graph on $n$ vertices forbidding $\mathcal{C}^o_{<2k+1}$ with property (1).
    Let $(u_1,u_2)$ be a $(2k-2)$-valid pair, and $\{A_1,A_2\}$ be the $3$-path-decomposition of $(u_1,u_2)$.
    Suppose $x,x' \in A_1$, $y,y' \in A_2$ are four distinct vertices such that $u_1 x y u_2 y' x' u_1$ is a cycle~(see Figure~\ref{fig: 4- and 6-cycle empty}).
    Denote by $R$ the region with boundary $u_1 x y u_2 y' x' u_1$.
    % That is, no vertex from $A_1 \cup A_2$ is inside $R$.
    Assume no vertex from $A_1 \cup A_2$ is inside $R$.
    Then there exists $n_2 = n_2(k)$ such that, if $n \ge n_2$ and $R$ is not empty, then $R$ contains at least $\frac{1}{1000k^{k}} n$ vertices.
\end{lemma}

\begin{proof}
    Let $\delta = \frac{1}{1000k^{k}}$.
    Denote by $P$ the $(u_1,u_2)$-path of length $2k-2$ in $G$~(see Figure~\ref{fig: 4- and 6-cycle empty}).
    Suppose otherwise; then there are at most $\delta n$ vertices inside $R$.
    Let $z$ be a vertex in $R$ with minimum degree in the planar graph induced by the vertices in the interior of $R$.
    Then $d(z) \le 11$~(including the vertices on the boundary).
    By property (1) and the pigeonhole principle, there exist $z_1, z_2 \in N(z)$ such that the number of $C_{2k+1}$ containing $z,z_1,z_2$ is at least $\frac{{{((n-2k-2)/k)}^{k-1}}}{\binom{11}{2}}$.

    By a similar argument as in the proof of Lemma~\ref{lem: 4-cycle empty}, the number of such cycles contained in $R$~(including the boundary) is at most ${(\delta n+6)}^{k-1}$ by Lemma~\ref{lem: size lower bound of A}.
    Thus most copies of $C_{2k+1}$ containing $z,z_1,z_2$ must contain at least one vertex outside $R$.

    Let $C$ be such a cycle, and let $P_C$ be the maximal path in $C$ without vertices in $R$~(excluding the boundary).
    The endpoints of $P_C$ must be two vertices in $\{u_1,u_2,x,x',y,y'\}$.
    There is exactly one such path in $C$: if there were two such paths $P_C$ and $P_C'$, then the four endpoints of $P_C$ and $P_C'$ would be distinct and independent, a contradiction.
    Let $q$ denote the length of $P_C$; then $q \ge 2$ since there is at least one vertex outside $R$.
    We now consider the endpoints of $P_C$.
    There are several possibilities as follows.
    \begin{enumerate}
        \item $u_1$ and $u_2$.
        \item $x$ and $x'$~(or equivalently $y$ and $y'$).
        \item $x$ and $y'$~(or equivalently $y$ and $x'$).
        \item $u_1$ and $y$~(or equivalently $u_1$ and $y'$, $u_2$ and $x$, $u_2$ and $x'$).
    \end{enumerate}

    In the first case, we first claim that $3 \le q \le 2k-2$, since $u_1u_2$ is not an edge and there is no $(u_1,u_2)$-path of length two.
    Also $q \neq 2k-2$, since $z \notin A_1 \cup A_2$.
    The length $q$ must be odd; otherwise, combining the path with $u_1xyu_2$ gives a shorter odd circuit.
    There also exists a $(u_1,u_2)$-path of length $2k+1-q$ in $C$.
    Combining it with $u_2 yx u_1$ gives a circuit of length $(2k+1-q) + 3 = 2k+4 - q$, which is odd.
    The only possibility is $q = 3$.

    Then there are at most $n$ choices for $P_C$ by Lemma~\ref{lem: size lower bound of A}.
    The remaining part of the cycle consists of a $(u_1,z_1)$-path of length $p$ and a $(u_2,z_2)$-path of length $(2k-4 - p)$~(or switching $z_1$ and $z_2$) for some $p \in \{0,1,\ldots,2k-4\}$.
    For the $(u_1,z_1)$-path of length $p$, there are ${(\delta n)}^{\lceil (p-1)/2 \rceil}$ choices by Lemma~\ref{lem: size lower bound of A}.
    For the $(u_2,z_2)$-path of length $(2k-4 - p)$, there are ${(\delta n)}^{\lceil (2k-5 - p)/2 \rceil}$ choices.
    In total, the number of such cycles is at most
    \begin{equation*}
    \begin{aligned}
        2n \sum_{p=0}^{2k-4} \left( {(\delta n)}^{\lceil (p-1)/2 \rceil} \cdot {(\delta n)}^{\lceil (2k-5 - p)/2 \rceil} \right) \le 4k\delta n^{k-1}.
    \end{aligned}
    \end{equation*}

    In the second case, the distance between $x$ and $x'$ in $C$ must be exactly two, otherwise we have a shorter odd circuit by combining $x u_1 x'$.
    That is, either $q = 2$ or $q = 2k-1$.
    When $q=2$, the vertices $x$ and $x'$ have a common neighbor outside $R$.
    By the existence of $P$, the only possible common neighbor is $u_1$, a contradiction since $u_1$ is not outside $R$.
    Then we must have $q = 2k-1$ and thus $\{z_1,z_2\} = \{x,x'\}$.
    Then $C$ must contain at least one vertex $w$ in $P$ other than $u_1$.
    Thus there is a $(x,w)$-path of length $p_1$ and a $(x',w)$-path of length $(2k-1 - p_1)$ for some $p_1 \in \{1,\ldots,2k-2\}$.
    By definition, there exist $(x,w)$-paths~(and $(x',w)$-paths) of length $p_2$ and $(2k+1 - p_2)$ for some $p_2 \in \{1,\ldots,2k\}$.
    An argument similar to the proof of Lemma~\ref{lem: 4-cycle empty} yields a shorter odd circuit, a contradiction.
    Therefore, no cycle of the second type exists.

    In the third case, an argument similar to the second type shows that $q = 2k-2$.
    That is, one of $z_1,z_2$ is in $\{x,y'\}$ and the other one is adjacent to the other one in $\{x,y'\}$.
    If the cycle contains $y$, then there are at most $n^{k-2}$ choices for the $(x,y')$-path of length $2k-3$ which is negligible.
    From now on, we may therefore assume that the cycle does not contain $y$, and equivalently does not contain $x'$.
    Then the cycle must contain a vertex $w$ in $P$.
    Thus there is a $(x,w)$-path of length $p_1$ and a $(y',w)$-path of length $(2k-2 - p_1)$ for some positive integer $p_1$.
    By definition, there exist $(x,w)$-paths of lengths $p_2$ and $(2k+1 - p_2)$.
    There also exist $(y',w)$-paths of lengths $p_2 +1$ and $(2k - p_2)$ for some $p_2$.

    If $p_1+p_2$ is odd, then $p_1+p_2 \ge 2k+1$ by considering the circuit with two $(x,w)$-paths.
    Then the circuit consisting of the $(y',w)$-path of length $2k-2-p_1$ and the $(y',w)$-path of length $2k-p_2$ has odd length $(2k-2-p_1 + 2k-p_2) < 2k+1$, a contradiction.
    If $p_1+p_2$ is even, then one of $p_1 + 2k+1-p_2$ and $(2k-2-p_1) + p_2+1$ is a smaller odd length, leading to a contradiction in the same way.
    Therefore, the number of cycles of the third type is at most $2n^{k-2}$.

    In the fourth case, the same argument as before shows that $q = 2k-1$ or $q=2$.
    If $q=2k-1$, then $\{z_1,z_2\} = \{u_1,y\}$ and thus $z \in A_1$, a contradiction.
    Hence, $q=2$, and the cycle contains $u_1 w y$ for some $w$ outside $R$, a $(u_1,z_1)$-path of length $p$ and a $(y,z_2)$-path of length $(2k-3 - p)$~(or switching $z_1$ and $z_2$) for some $p \in \{0,1,\ldots,2k-3\}$.
    For $w$, there are at most $n$ choices.
    For the $(u_1,z_1)$-path of length $p$, there are ${(\delta n)}^{\lceil (p-1)/2 \rceil}$ choices by Lemma~\ref{lem: size lower bound of A}.
    For the $(y,z_2)$-path of length $(2k-3 - p)$, there are ${(\delta n)}^{\lceil (2k-4 - p)/2 \rceil}$ choices.
    In total, the number of such cycles is at most
    \begin{equation*}
    \begin{aligned}
        2n \sum_{p=0}^{2k-3} \left( {(\delta n)}^{\lceil (p-1)/2 \rceil} \cdot {(\delta n)}^{\lceil (2k-4 - p)/2 \rceil} \right) \le 4k\delta n^{k-1}.
    \end{aligned}
    \end{equation*}

    Thus the total number of copies of $C_{2k+1}$ containing $z,z_1,z_2$ is at most
    \[
        10k\delta n^{k-1} < {((n-2k-2)/k)}^{k-1}/\binom{11}{2},
    \]
    a contradiction.
\end{proof}

Lemmas~\ref{lem: 4-cycle empty} and~\ref{lem: 6-cycle empty} show that, for these special $4$-cycles and $6$-cycles, the interior is either empty or contains many vertices.

% -------------------------------------------------------------
\subsection{Find a vertex of degree two or a special path from large \texorpdfstring{$p_3(u_1,u_2)$}{p3(u1,u2)} }\label{subsec: special path or vertex of degree two}

If we have a vertex $x$ of degree two with neighbors $u_1,u_2$, then property (1) implies that $p_{2k-1}(u_1,u_2)$ is large~($p_{2k-1}(u_1,u_2) \ge {((n-2k-2)/k)}^{k-1}$).
Similarly, Lemma~\ref{lem: with long path} shows that if there is a special path $u_1 x_{\ell} y_{\ell} u_2$, then $p_{2k-2}(u_1,u_2)$ is still large, up to an $\epsilon$ fraction: $p_{2k-2}(u_1,u_2) \ge (1-\epsilon){((n-3\ell)/k)}^{k-1}$.

Both structures will be very useful in the proof.
Corollary~\ref{cor: large common neighbor} states that if $p_2(u_1,u_2)$ is large for a $(2k-1)$-valid pair $(u_1,u_2)$, then there exists a vertex of degree two in $N(u_1) \cap N(u_2)$.

We now prove that if $p_3(u_1,u_2)$ is large for a $(2k-2)$-valid pair $(u_1,u_2)$, then we can find either a vertex of degree two or a special path.

% \begin{lemma}\label{lem: maximum matching}
%     Let $G$ be a planar graph on $n$ vertices forbidding $\mathcal{C}^o_{<2k+1}$ with property (1).
%     Let $(u_1,u_2)$ be a $2k-2$-valid-pair in $G$.
%     Let $\{A_1,A_2\}$ be a $3$-path-decomposition of $(u_1,u_2)$ as (\ref{eq: Ai}).
%     Suppose there is no vertex of degree two in $A_1 \cup A_2$.
%     Suppose there is a maximal matching $M = {\{x_i y_i\}}_{1 \le i \le m}$ in the bipartite graph with parts $A_1$ and $A_2$.
%     Let $R_i$ denote the region with boundary $u_1 x_i y_i u_2  y_{i+1} x_{i+1} u_1$ for each $1 \le i \le m-1$.

%     If $R_i$ is not empty, then there are at least $\delta n$ vertices in $R_i$.

% \end{lemma}

% \begin{proof}

% \end{proof}

\begin{lemma}\label{lem: vertex of degree two or long path}
    Let $G$ be a planar graph on $n$ vertices forbidding $\mathcal{C}^o_{<2k+1}$ with property (1).
    Let $(u_1,u_2)$ be a $(2k-2)$-valid pair in $G$.
    Let $\{A_1,A_2\}$ be the $3$-path-decomposition of $(u_1,u_2)$.
    There exists a constant $n_3 = n_3(k)$ such that the following holds.

    If $n \ge n_3$ and $p_3(u_1,u_2) \ge n^{0.2}$, then either there is a vertex of degree two in $A_1 \cup A_2$, or there exists a path $x_1y_1 \ldots x_{3\ell} y_{3\ell}$ such that $x_i \in A_1$ and $y_i \in A_2$ for all $1 \le i \le 3\ell$, and every $x_i,y_i$ has degree three, that is, a special path $u_1x_\ell y_{\ell} u_2$.
\end{lemma}

\begin{proof}
    Let $A = A_1 \cup A_2$.
    By Lemma~\ref{lem: size lower bound of A}, we have $|A| \ge n^{0.2}$.
    Without loss of generality, we may assume $|A_2| \ge |A_1|$.
    Then $|A_2| \ge n^{0.2}/2$.
    From now on, assume that there is no vertex of degree two in $A_1 \cup A_2$.
    Denote by $\mathcal{P}$ the $(u_1,u_2)$-path of length $2k-2$ in $G$.

    Let $M = {\{x_i y_i\}}_{1 \le i \le m}$ be a maximum matching in the bipartite subgraph induced by $A_1$ and $A_2$.
    Assume that ${\{y_i\}}_{1 \le i \le m}$ occur in counter-clockwise order around $u_2$. Then ${\{x_i\}}_{1 \le i \le m}$ occur in clockwise order around $u_1$~(see Figure~\ref{fig: region R_i}).
    Without loss of generality, if a vertex $y \in A_2$ lies between $y_i$ and $y_{i+1}$, then $y$ is not adjacent to $x_{i+1}$; otherwise, we could replace $x_{i+1}y_{i+1}$ by $x_{i+1}y$.
    Similarly, no vertex $y$ in $A_2$ is before $y_1$ in the counter-clockwise order around $u_2$ that is adjacent to $x_1$, otherwise we can replace $x_1y_1$ by $x_1y$.
    By definition, for a vertex $y$ in $A_2$ between $y_{i}$ and $y_{i+1}$, it must be adjacent to $x_{i}$ since there is no vertex of degree two.
    Similarly, for a vertex $x$ in $A_1$ between $x_{i}$ and $x_{i+1}$, it must be adjacent to at least one of $\{y_i, y_{i+1}\}$.

    Let $\alpha_i$ be the number of vertices in $A_2$ between $y_i$ and $y_{i+1}$, $1 \le i \le m-1$.
    Let $\alpha_m$ be the number of vertices in $A_2$ after $y_m$.
    Since there is no vertex in $A_2$ before $y_1$, we have
    \[
        |A_2| = \sum_{i=1}^{m} \alpha_i + m.
    \]

    If $\alpha_i \ge 800k^{k}$ for some $i$, then by Corollary~\ref{cor: large common neighbor} there exists a vertex of degree two, a contradiction.
    Thus
    \begin{equation*}
    \begin{aligned}
        n^{0.2}/2 \le |A_2| &= \sum_{i=1}^{m} \alpha_i + m \le (800k^{k} + 1)m.
    \end{aligned}
    \end{equation*}
    Hence $m \ge n^{0.1}$ when $n$ is sufficiently large.

    \begin{figure}
        \centering
        \includegraphics[width=0.5\textwidth]{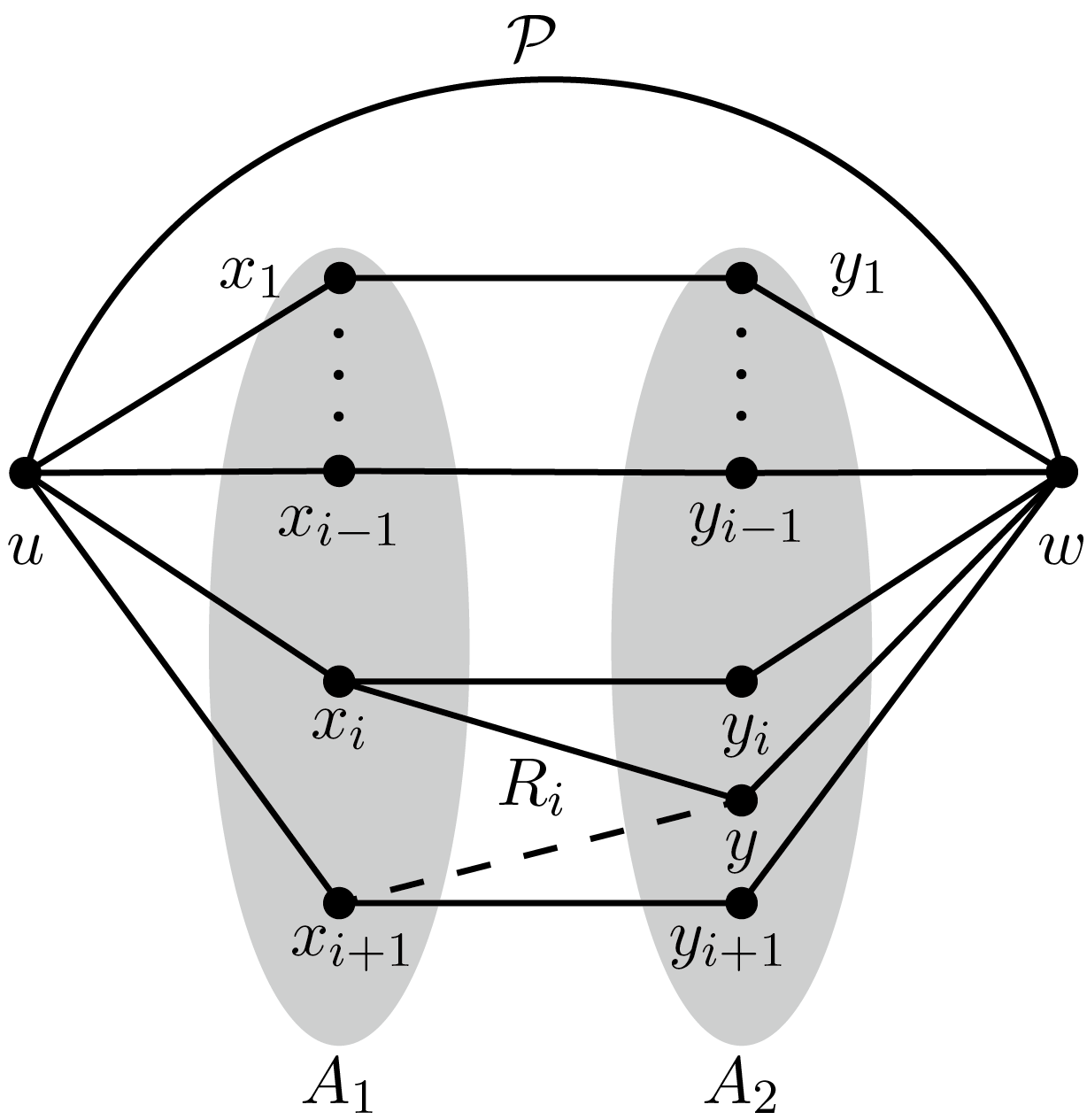}
        \caption{The region $R_i$ and the case when $x_i$ is adjacent to a vertex $y$ in $R_i$ other than $y_i$ and $y_{i+1}$. The dotted line is an impossible edge.}
        \label{fig: region R_i}
    \end{figure}

    Let $R_i$ be the region with boundary $u_1 x_i y_i u_2 x_{i+1} y_{i+1} u_1$~(see Figure~\ref{fig: region R_i}).
    We say that $R_i$ is \textbf{almost empty} if it contains no vertex of $V(G)\setminus (A_1\cup A_2)$.

    For every $R_i$, there are at most $3200k^{k}$ vertices from $A_1 \cup A_2$ inside $R_i$; otherwise Corollary~\ref{cor: large common neighbor} gives a vertex of degree two.
    Moreover, if $R_i$ is not almost empty, that is, if it contains a vertex of $V(G)\setminus (A_1 \cup A_2)$, then the vertices from $A_1 \cup A_2$ further split $R_i$ into smaller regions whose boundaries are either $4$-cycles or $6$-cycles.
    One of these smaller regions is non-empty, and hence, by Lemmas~\ref{lem: 4-cycle empty} and~\ref{lem: 6-cycle empty}, contains at least $\frac{1}{1000k^{k}} n$ vertices.
    Thus, if $R_i$ is not almost empty, then it contains at least $\frac{1}{1000k^{k}} n$ vertices.
    Hence at most $1000k^{k}$ regions are not almost empty.

    Let
    \[
    I = \{i \in [4\ell, m-4\ell] \mid \forall j \in [i-4\ell, i+4\ell], \text{$R_j$ is almost empty}\}
    \]

    We have $|I| \ge m - 8\ell - 1000k^{k} 8\ell = \Omega(n^{0.1})$ when $n$ is large enough.
    Pick any $i \in I$.
    Consider the neighbors of $x_i$.
    By our choice of the maximum matching, $x_i$ has no neighbors in $R_{i-1}$ except $y_{i-1}$ and $y_i$.
    If $x_i$ has a neighbor $y$ in $R_{i+1}$ with $y \neq y_{i}, y_{i+1}$~(see Figure~\ref{fig: region R_i}), then $y\in A_2$ and $y$ must be a vertex of degree two, a contradiction.
    Hence $x_i$ has at most four neighbors, namely $u_1, y_{i-1}, y_i, y_{i+1}$.
    Since $x_i$ is not a vertex of degree two, it must be adjacent to at least one of $y_{i-1}$ and $y_{i+1}$.

    % If $x_i$ is not a vertex of degree two, then $x_i$ is adjacent to at least one of $y_{i-1}$ and $y_{i+1}$.
    Notice that $x_i$ cannot be adjacent to both $y_{i-1}$ and $y_{i+1}$; otherwise $y_i$ would be a vertex of degree two.
    Without loss of generality, we may assume $x_i$ is adjacent to $y_{i+1}$.
    Then $x_{i+1}$ must be adjacent to $y_{i+2}$, since $x_{i+1}y_i$ cannot be an edge by planarity.
    Continuing in this way, we obtain a path $x_i y_{i+1} x_{i+1} y_{i+2} \ldots x_{i+3\ell} y_{i+3\ell+1}$ such that every $x_j,y_j$ has degree three.
\end{proof}

\begin{remark*}
    In the statement of Lemma~\ref{lem: vertex of degree two or long path}, $\{A_1,A_2\}$ is the $3$-path-decomposition of $(u_1,u_2)$.
    If we replace $\{A_1,A_2\}$ by $\{A_1',A_2'\}$ where $A_i'$ is a consecutive subset of $A_i$ and replace $p_3(u_1,u_2)$ by the number of $(u_1,u_2)$-paths of length $3$ using a vertex in $A_1'$ and a vertex in $A_2'$, then the lemma still holds.
    This observation will be used only in the proof of Lemma~\ref{lem: large star of long path}.
\end{remark*}

We now use Lemma~\ref{lem: vertex of degree two or long path} to prove the following stronger statement.
Lemma~\ref{lem: large star of long path} shows that if $p_3(u_1,u_2)$ is large for a $(2k-2)$-valid pair $(u_1,u_2)$, then we can find either a special path or a degree-two vertex together with a rather large star joining $A_1$ and $A_2$.

\begin{lemma}\label{lem: large star of long path}
    Let $\epsilon$ be a fixed constant.
    Let $G$ be a planar graph on $n$ vertices forbidding $\mathcal{C}^o_{<2k+1}$ and satisfying properties (1) and (2).
    There exists $n_3 = n_3(\epsilon,k)$ such that the following holds whenever $n\ge n_3$.
    Let $(u_1,u_2)$ be a $(2k-2)$-valid pair with $p_3(u_1,u_2) \ge \epsilon n$.
    Let $A_1$ and $A_2$ be the $3$-path-decomposition of $(u_1,u_2)$.
    Either
    \[
    \max_{x \in A_1, y \in A_2} \{ |N(x) \cap N(u_2)|, |N(y) \cap N(u_1)|\} \ge p_3(u_1,u_2) - 2\epsilon n,
    \]
 or there exists a path $x_{1}y_{1} \ldots x_{3\ell} y_{3\ell}$ such that $x_i \in A_1$ and $y_i \in A_2$ for all $1 \le i \le 3\ell$, and every $x_i,y_i$ has degree three, that is, a special path.
\end{lemma}

\begin{proof}
    Let $n_3(\epsilon, k)$ be a constant larger than the constant in Lemma~\ref{lem: vertex of degree two or long path}.
    Denote by $\mathcal{P}$ the $(u_1,u_2)$-path of length $2k-2$ in $G$.
    By Lemma~\ref{lem: vertex of degree two or long path}, we may assume that there is no special path. Thus there exists a vertex $x_1 \in A_1$ of degree two. Let $y_1 \in A_2$ be the neighbor of $x_1$ other than $u_1$~(see Figure~\ref{fig: lem large star of long path}).

    \begin{figure}
        \centering
        \includegraphics[width=0.3\linewidth]{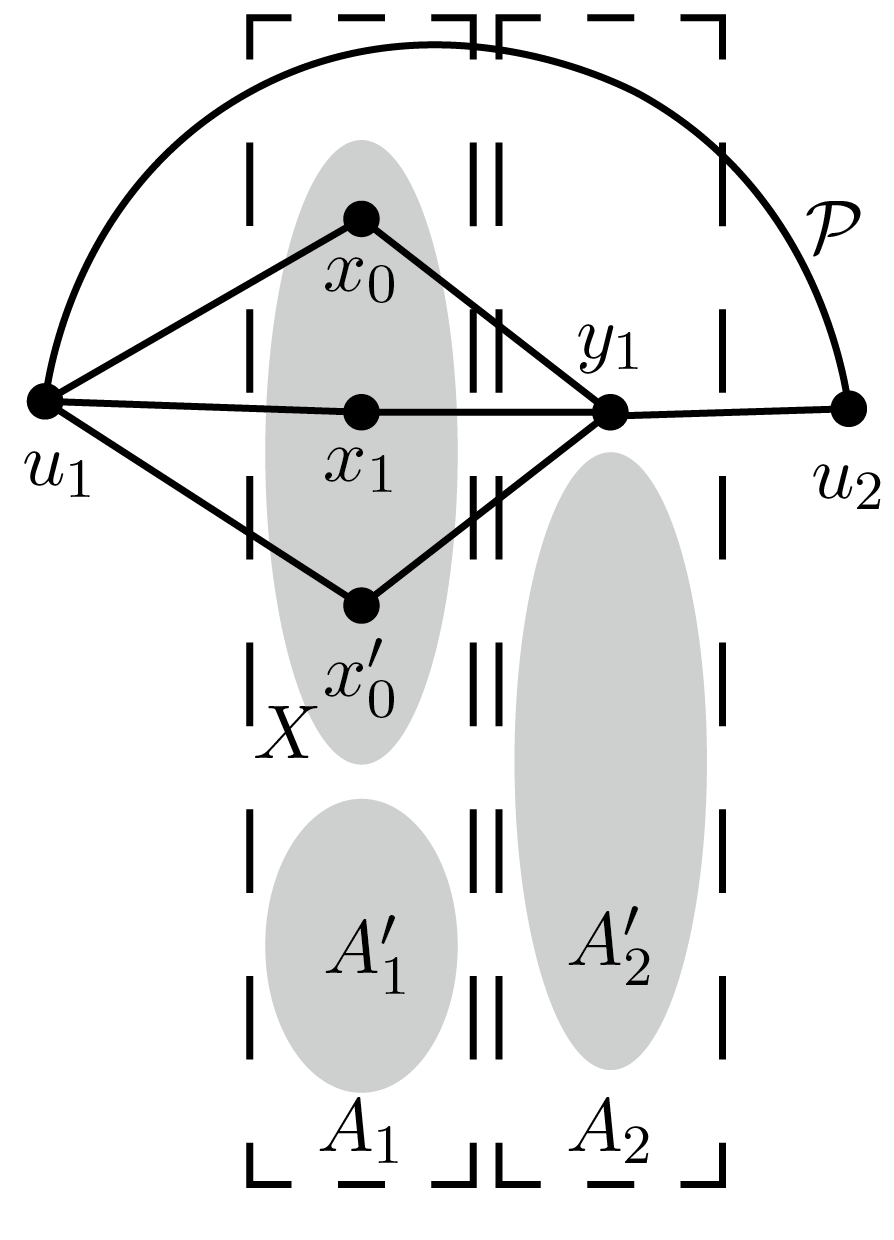}
        \caption{In the proof of Lemma~\ref{lem: large star of long path}.}\label{fig: lem large star of long path}
    \end{figure}

    Let $X = N(u_1) \cap N(y_1)$.
    If $|X| \ge p_3(u_1,u_2) - 2\epsilon n$, then we are done.
    Otherwise, note that $\mathcal{P}$ and $u_1x_1y_1u_2$ form a cycle that divides the plane into two regions.
    We call the bounded region the \textbf{interior} and the unbounded region the \textbf{exterior}.

    Let $x_0$ be the innermost vertex in $X$, that is, there is no vertex of $X$ in the interior of the cycle consisting of $\mathcal{P}$ and $u_1 x_0 y_1 u_2$.
    Let $x_0'$ be the outermost vertex in $X$.
    We first claim that $p_2(x_0,u_2) \le 800k^{k}$.
    Otherwise, there exists a vertex $y_0 \in N(x_0) \cap N(u_2)$ of degree two other than $y_1$ by Corollary~\ref{cor: large common neighbor}.
    By property (2), $x_1$ and $y_0$ must be contained in a common $C_{2k+1}$.
    Such a cycle must contain $u_1,x_1,y_1,u_2,y_0,x_0$, which already form a $C_6$, a contradiction since the cycle must be induced.
    Similarly, $p_2(x_0',u_2) \le 800k^{k}$.
    By planarity, for every vertex in $X$ between $x_0$ and $x_0'$, it cannot be adjacent to other vertices in $A_2$ other than $y_1$.
    Therefore, the number of $(u_1,u_2)$-paths of length three that contain a vertex in $X$ is at most
    \[
        |X| + 800k^{k} \le (p_3(u_1,u_2) - 2\epsilon n + 800k^{k}) < p_3(u_1,u_2) - \epsilon n.
    \]

    Let $A_i'$ be the subset of $A_i\setminus X$ in the exterior of $u_1x_1y_1u_2 \mathcal{P}$ for $i=1,2$.
    Similarly, let $A_i''$ be the subsets of $A_i\setminus X$ in the interior of $u_1x_1y_1u_2 \mathcal{P}$ for $i=1,2$.
    By planarity, a $(u_1,u_2)$-path of length three cannot contain both a vertex in the interior and a vertex in the exterior.

    Let $p_1$~(resp. $p_2$) be the number of $(u_1,u_2)$-paths of length three avoiding $X$ and using only vertices in the interior~(resp.~exterior).
    Then $p_1 + p_2 \ge \epsilon n$.
    Without loss of generality, we may assume $p_2 \ge p_1$ and thus $p_2 \ge \epsilon n/2$.
    The number of $(u_1,u_2)$-paths of length three using a vertex in $A_1'$ and a vertex in $A_2'$ is at least $\epsilon n/2$.
    By Lemma~\ref{lem: vertex of degree two or long path}, there exists a vertex of degree two in $A_1' \cup A_2'$~(here we use the remark after Lemma~\ref{lem: vertex of degree two or long path}).

    If there is a vertex of degree two in $A_1'$, say $x_2$, and if $y_2$ denotes the neighbor of $x_2$ other than $u_1$, then by property (2), $x_1$ and $x_2$ must be contained in a common $C_{2k+1}$.
    Such a cycle must contain $u_1,x_1,y_1,x_2,y_2$, and a $(y_1,y_2)$-path of length $2k-3$.
    Combining $y_1u_2y_2$, we would obtain a shorter odd circuit, a contradiction.
    Hence, the vertex of degree two must be in $A_2'$, say $y_3$.
    Denote the neighbor of $y_3$ apart from $u_2$ by $x_3$.
    Again by property (2), $x_1$ and $y_3$ must be contained in a common $C_{2k+1}$.
    Such a cycle must contain $u_1,x_1,y_1,u_2,y_3,x_3$, which already form a $C_6$, a contradiction since the cycle must be induced.
\end{proof}

% -------------------------------------------------------------
\subsection{Estimate Deltas}\label{subsec: prep Delta 2 Delta 3}

Let $z$ be a vertex of minimum degree in $G$.
By property (1), $2 \le d(z) \le 5$.
There exists $z_1,z_2 \in N(z)$ such that the number of $C_{2k+1}$ containing $z_1zz_2$ is at least $\frac{{{((n-2k-2)/k)}^{k-1}}}{\binom{5}{2}}$.
Thus $\Delta_{2k-1} \ge \frac{{{((n-2k-2)/k)}^{k-1}}}{\binom{5}{2}} = \Omega(n^{k-1})$.
Our first goal is to prove $\Delta_m = \Omega(n^{\lfloor m/2 \rfloor})$ for each $m$.

\begin{lemma}\label{lem: order of deltas}
    Let $\gamma$ and $t$ be fixed constants with $1 \le t \le k-2$.
    Let $G$ be a planar graph on $n$ vertices forbidding $\mathcal{C}^o_{<2k+1}$ with property (1).
    Suppose there exists a $2t$-valid pair $(u_1,u_2)$ with $p_{2k-2t+1}(u_1,u_2) \ge \gamma n^{k-t}$, and let $\{A_1,A_2,\ldots,A_{2k-2t}\}$ be the $(2k-2t+1)$-path-decomposition of $(u_1,u_2)$.
    Then either there exists a vertex $w_1 \in A_2$ such that $p_2(u_1,w_1) \ge \gamma n/2$ and $p_{2k-2t-1}(w_1,u_2) \ge \gamma n^{k-t-1}$, or there exists a vertex $w_2 \in A_{3}$ such that $p_3(u_1, w_2) \ge \gamma n/2$ and $p_{2k-2t-2}(w_2, u_2) \ge \gamma n^{k-t-1}$.
\end{lemma}
\begin{proof}

    Let $E_i$ be the set of edges between $A_{2i-1}$ and $A_{2i}$, $i=1,2,\ldots,k-t$.
    Recall that $E_i$ is a forest by Observation~\ref{obs: forest between Ai and Ai+1} and thus $|E_i| \le |A_{2i-1}| + |A_{2i}|$.
    Let $\mathcal{H}$ be an auxiliary $(k-t)$-partite $(k-t)$-uniform hypergraph with parts ${\{E_i\}}_{1 \le i \le k-t}$.
    For a tuple $(e_1,\ldots,e_{k-t})$ where $e_i \in E_i$, $i=1,2,\ldots, k-t$, we have $(e_1,\ldots,e_{k-t}) \in E(\mathcal{H})$ if and only if $e_1,\ldots,e_{k-t}$ occur along a $(u_1,u_2)$-path of length $2k-2t+1$.

    Then $e(\mathcal{H}) = p_{2k-2t+1}(u_1,u_2) \ge \gamma n^{k-t}$.
    Also, $v(\mathcal{H}) \le \sum_{i=1}^{k-t}|E_i| \le \sum_{i=1}^{2k-2t} |A_i| \le n$.

    Removing the vertices in $\mathcal{H}$ with degree at most $\gamma n^{k-t-1}$ gives a subgraph $\mathcal{H}'$ of $\mathcal{H}$ with minimum degree at least $\gamma n^{k-t-1}$.
    During this process the average degree increases, so the resulting graph $\mathcal{H}'$ is non-empty.

    Select $a_3a_4 \in V(\mathcal{H}') \cap E_2$ arbitrarily.
    Define
    \[
        E_1(a_3a_4) = \{e\in E_1 \mid \exists \mathcal{E} \in E(\mathcal{H}), e,a_3a_4 \in \mathcal{E}\},
    \]
    Thus $E_1(a_3a_4)$ is the set of neighbors of $a_3a_4$ in $V(\mathcal{H}')\cap E_1$.
    The degree of $a_3a_4$ in $\mathcal{H}'$ satisfies
    \[
        \gamma n^{k-t-1} \le d_{\mathcal{H}'}(a_3a_4) \le |E_1(a_3a_4)| \cdot \prod_{i=3}^{k-t} |E_i| \le  |E_1(a_3a_4)| n^{k-t-2}.
    \]
    Thus $|E_1(a_3a_4)| \ge \gamma n$, that is, $p_3(a_3,u_1) \ge \gamma n$.
    Let
    \[
        Y = \{y \in A_2 \mid \exists x \in A_1, xy \in E_1(a_3a_4)\}.
    \]

    \textbf{Case 1:} $|Y| \ge 3$. Choose three vertices $y_1,y_2,y_3 \in Y$. Assume $y_2$ lies between $y_1$ and $y_3$ in the counter-clockwise order around $a_3$, and let $x_2y_2 \in E_1(a_3a_4)$.
    Define
    \[
        \mathcal{E}(x_2y_2) = \{e \in E(\mathcal{H}') \mid x_2y_2 \in e\}.
    \]
    By the minimum degree condition of $\mathcal{H}'$, we have $|\mathcal{E}(x_2y_2)| \ge \gamma n^{k-t-1}$.
    For every hyperedge $e \in \mathcal{E}(x_2y_2)$,
    denote $z = (e \cap E_2)\cap A_3$.
    By the existence of $y_1,y_3$ and planarity, $z$ must be $a_3$.
    See the left figure in Figure~\ref{fig: lem rough p5}.
    This yields $p_{2k-2t-2}(a_3,u_2) \ge \gamma n^{k-t-1}$.
    Combining with $p_3(a_3,u_1) \ge |E_1(a_3a_4)| \ge \gamma n$, we are done by letting $w_2 = a_3$.

    \begin{figure}
        \centering
        \includegraphics[width=0.9\textwidth]{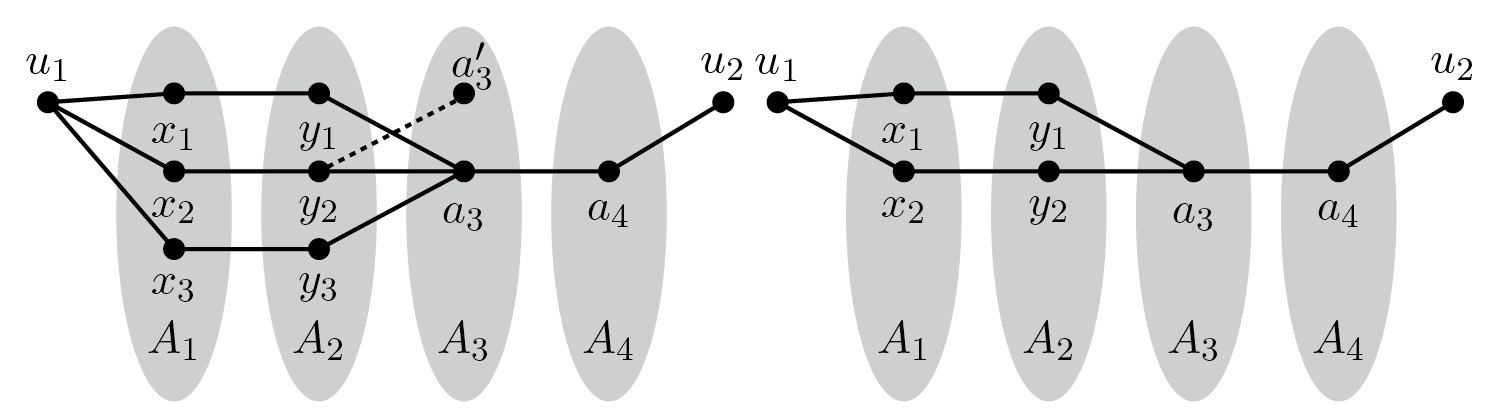}
        \caption{The left figure is for Case 1, and the right figure is for Case 2 in the proof of Lemma~\ref{lem: order of deltas}.}\label{fig: lem rough p5}
    \end{figure}

    \textbf{Case 2:} $|Y| \le 2$.
    Then there are at most two vertices in $Y$, say $y_1,y_2$.
    Since $|E_1(a_3a_4)| = p_2(u_1,y_1) + p_2(u_1,y_2) \ge \gamma n$, we may assume without loss of generality that $p_2(u_1,y_1) \ge \gamma n/2$.
    Also, $p_{2k-2t-1}(y_1,u_2) \ge \gamma n^{k-t-1}$ by $x_1y_1 \in \mathcal{H}'$ for some $x_1$ and the degree condition.
    This completes the proof by taking $w_1 = y_1$.
\end{proof}

\begin{lemma}\label{lem: order of deltas 2}
    Let $\gamma$ and $t$ be fixed constants with $1 \le t \le k-2$.
    Let $G$ be a planar graph on $n$ vertices forbidding $\mathcal{C}^o_{<2k+1}$ with property (1).
    Suppose there exists a $(2t+1)$-valid pair $(u_1,u_2)$ with $p_{2k-2t}(u_1,u_2) \ge \gamma n^{k-t}$, and let $\{A_1,A_2,\ldots,A_{2k-2t-1}\}$ be the $(2k-2t)$-path-decomposition of $(u_1,u_2)$.
    Then there exists a vertex $w_1 \in A_2$ such that $p_2(u_1,w_1) \ge \gamma n/2$ and $p_{2k-2t-2}(w_1,u_2) \ge \gamma n^{k-t-1}$.
\end{lemma}

\begin{proof}
    The proof is similar to that of Lemma~\ref{lem: order of deltas}.
    We replace $E_{k-t}$ by $A_{2k-2t-1}$ in the proof of Lemma~\ref{lem: order of deltas}.
    Then Case 1 is impossible.
    Let $a_3$, $x_2$ and $y_2$ be the vertices appearing in Case 1 of the proof of Lemma~\ref{lem: order of deltas}~(see the left figure in Figure~\ref{fig: lem rough p5}).

    By the minimum degree condition of $\mathcal{H}'$, the degree of $x_2y_2$ is at least $\gamma n^{k-t-1}$.
    Hence the number of $(u_1,u_2)$-paths of length $2k-2t$ using $x_2y_2$ is at least $\gamma n^{k-t-1}$.
    By planarity, all such paths must contain $a_3$.

    Let $E_i'$ be the edge set between $A_{2i}$ and $A_{2i+1}$, $i=1,2,\ldots, k-t-1$.
    By Observation~\ref{obs: forest between Ai and Ai+1}, $E_i'$ is a forest and thus $|E_i'| \le |A_{2i}| + |A_{2i+1}| \le n$.
    The number of $(u_1,u_2)$-paths of length $2k-2t$ using $x_2,y_2,a_3$ is at most $\prod_{i=2}^{k-t-1} |E_i'| \le n^{k-t-2}$, a contradiction.

    The result therefore follows from Case 2 in the proof of Lemma~\ref{lem: order of deltas}.
\end{proof}

Recall that $\Delta_{2k-1} \ge \frac{((n-2k-2)/k)^{k-1}}{\binom{5}{2}} \ge \frac{1}{11k^{k-1}} n^{k-1}$ when $n$ is large.
Applying Lemmas~\ref{lem: order of deltas} and~\ref{lem: order of deltas 2}, we obtain $\Delta_2 \ge \frac{1}{22k^{k-1}} n$.
Let $u_1,u_2$ be a $(2k-2)$-valid pair with $p_{2}(u_1,u_2)  = \Delta_2$.
By Corollary~\ref{cor: large common neighbor}, there exists a vertex $w \in N(u_1) \cap N(u_2)$ of degree two.
By property (1), the number of $(u_1,u_2)$-paths of length $2k-1$ is at least $\left(\frac{n-2k-2}{k}\right)^{k-1} \ge \gamma_0 n^{k-1}$, where $\gamma_0 = \frac{1}{2k^{k-1}}$.
Iteratively applying Lemmas~\ref{lem: order of deltas} and~\ref{lem: order of deltas 2}, starting from $u_1,u_2$, shows that there exists a sequence of vertices $v_1,v_2,\ldots,v_{k}$ such that $p_2(v_i,v_{i+1}) \ge \gamma_0 n/2$ for $i=1,2,\ldots,k-1$ and $p_{3}(v_1,v_k) \ge \gamma_0 n/2$.
Let $\gamma = \gamma_0^k / 2^k$. Then

\begin{equation}\label{eq: lb Delta m}
\begin{aligned}
    &\Delta_m \ge \gamma n^{ \lfloor m/2 \rfloor}, m=2,3,\ldots,2k-1.
\end{aligned}
\end{equation}
By Lemma~\ref{lem: size lower bound of A}, we have the upper bound
\begin{equation}\label{eq: ub Delta m}
\begin{aligned}
    &\Delta_m \le n^{ \lceil (m-1)/2 \rceil} = n^{ \lfloor m/2 \rfloor} , m=2,3,\ldots,2k-1.
\end{aligned}
\end{equation}
As we have shown that $\Delta_m = \Theta(n^{\lfloor m/2 \rfloor})$ for $m=2,3,\ldots,2k-1$, we next estimate the relationship between $\Delta_2,\Delta_3$ and the other $\Delta$ terms.

\begin{lemma}\label{lem: other Delta}
    Let $G$ be a planar graph on $n$ vertices forbidding $\mathcal{C}^o_{<2k+1}$ with property (1), (2) and (3).
    Let $4 \le m \le 2k-3$ be fixed, and let $\epsilon$ be fixed and small. Then there exists $n_4 = n_4(\epsilon, k)$ such that we have
    \begin{equation*}
        \Delta_m \le \left\{ \begin{array}{ll}
            \Delta_2^{m/2} + 2m \epsilon n^{m/2}, & \text{if $m$ is even};\\
            \Delta_2^{(m-3)/2} \Delta_3 + 2m \epsilon n^{(m-1)/2}, & \text{if $m$ is odd}.
        \end{array} \right.
    \end{equation*}
    for $n \ge n_4$.
\end{lemma}

\begin{proof}

    We start with two claims.

    \begin{claim}\label{claim: unite two paths in a common cycle}
        Let $(u_1,u_2)$ be a $(2k + 1 - m)$-valid pair and let $\{A_1,\ldots,A_{m-1}\}$ be the $m$-path-decomposition of $(u_1,u_2)$.
        Suppose that $u_1x_1x_2$ is a path where $x_1 \in A_1$ and $x_2 \in A_2$, and $y_{i}y_{i+1}y_{i+2}$ is a path where $y_j \in A_j$ for $j=i,i+1,i+2$ for some $0 \le i \le m-2$~(let $A_0 = \{u_1\}$ and $A_m = \{u_2\}$). The edges $u_1x_1,x_1x_2, y_{i}y_{i+1}, y_{i+1}y_{i+2}$ are different.

        If $u_1x_1x_2$ and $y_{i}y_{i+1}y_{i+2}$ are contained in a common $C_{2k+1}$, then either there is a $(x_2,y_{i})$-path of length $i-2$, or there is a $(x_2,y_{i+2})$-path of length $(2k-i-3)$.
    \end{claim}

    \noindent
    \textbf{Proof of Claim~\ref{claim: unite two paths in a common cycle}.}
    \begin{figure}
        \centering
        \includegraphics[width=0.8\textwidth]{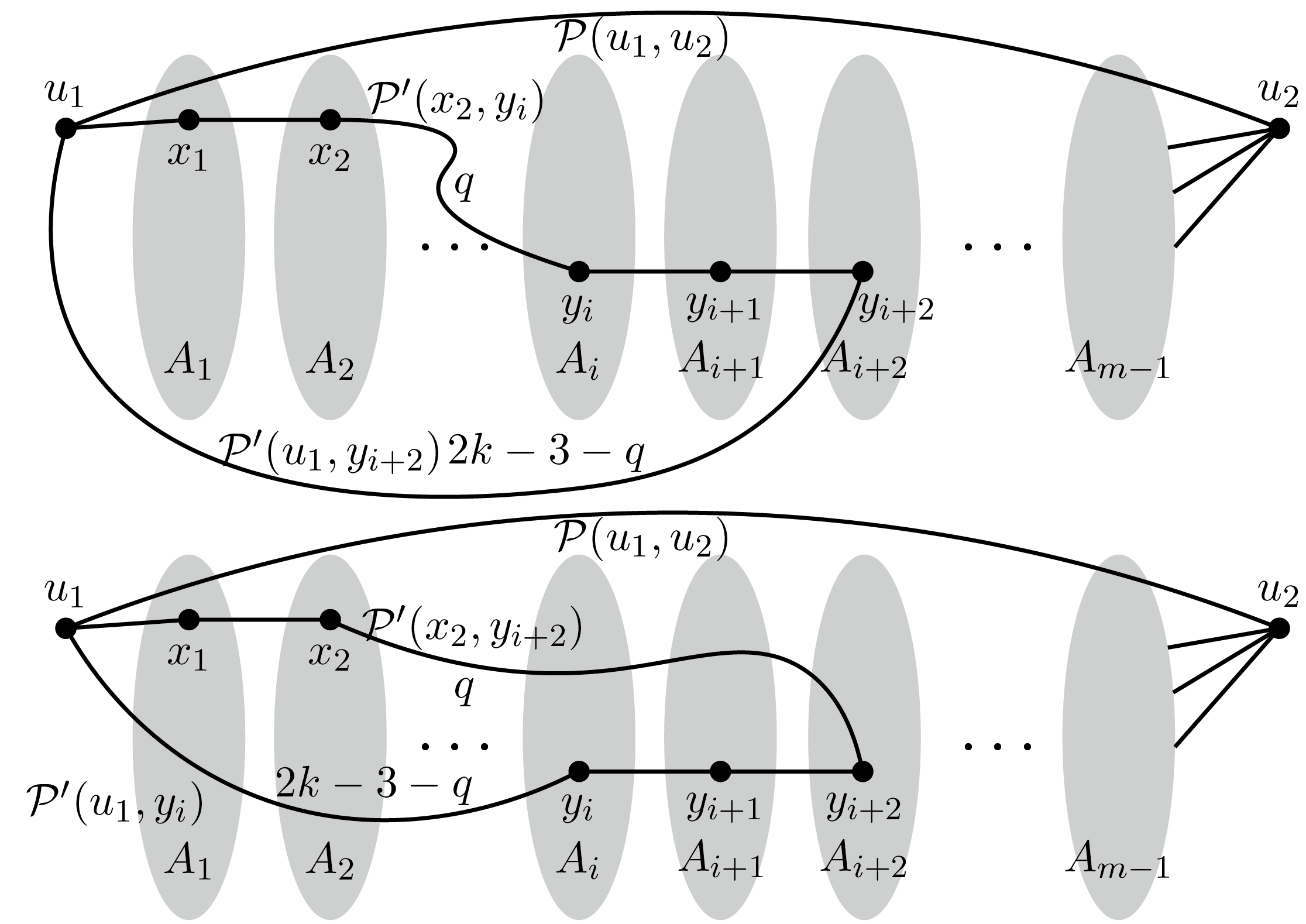}
        \caption{The cycle containing the path $u_1x_1x_2$ and the path $y_{i}y_{i+1} y_{i+2}$ in Claim~\ref{claim: unite two paths in a common cycle}.}\label{fig: claim other Delta 1}
    \end{figure}

    Since $(u_1,u_2)$ is a $(2k+1-m)$-valid pair, there exists a $(u_1,u_2)$-path of length $2k+1-m$, denoted by $\mathcal{P}(u_1,u_2)$.
    By definition, for any vertex $z$ in $A_j$ for a fixed $j$, there is a $(u_1,z)$-path of length $j$ denoted by $\mathcal{P}(u_1,z)$ and a $(z,u_2)$-path of length $m-j$ denoted by $\mathcal{P}(z,u_2)$.

    \textbf{Case 1}: the cycle consists of $u_1x_1x_2$, a $(x_2,y_{i})$-path of length $q$ denoted by $\mathcal{P}'(x_2,y_{i})$, $y_{i} y_{i+1} y_{i+2}$, and a $(u_1,y_{i+2})$-path of length $2k-3-q$ denoted by $\mathcal{P}'(u_1,y_{i+2})$, where $0 \le q \le 2k-4$~($q=0$ means $x_2 = y_{i}$).

    When $q \not\equiv i \pmod{2}$, the circuit consisting of $\mathcal{P}'(u_1,y_{i+2})$, $\mathcal{P}(y_{i+2},u_2)$, $\mathcal{P}(u_1,u_2)$ has odd length $(2k-3-q) + (m-i-2) + (2k+1-m) = 4k-4 -q -i$.
    The circuit consisting of $u_1x_1x_2$, $\mathcal{P}'(x_2,y_{i})$, and $\mathcal{P}(u_1,y_{i})$ also has odd length $q+i+2$.
    Since $(4k-4-q-i) + (q+i+2) = 4k-2 < 2(2k+1)$, at least one of these two circuits has length less than $2k+1$, a contradiction.

    When $q \equiv i \pmod{2}$, the circuit consisting of $\mathcal{P}(u_1,u_2)$, $\mathcal{P}(y_{i}, u_2)$, $\mathcal{P}'(x_2,y_{i})$, $u_1x_1x_2$ has odd length $(2k+1-m) + (m-i)+ q + 2 = 2k + 3 + q -i$.
    Therefore, $2k+3 + q -i \ge 2k+1$, and thus $q \ge i-2$;
    The circuit consisting of $\mathcal{P}(u_1,y_{i+2})$ and $\mathcal{P}'(u_1,y_{i+2})$ also has odd length $(2k-3-q)+(i+2) = 2k-1-q+i$.
    Therefore, $2k-1 -q+i \ge 2k+1$, and thus $q \le i-2$.
    Hence $q = i-2$.
    In particular, we easily get a contradiction when $i < 2$.

    \textbf{Case 2:} the cycle consists of $u_1x_1x_2$, $(x_2,y_{i+2})$-path of length $q$ denoted by $\mathcal{P}'(x_2,y_{i+2})$, $y_{i+2} y_{i+1} y_{i}$, $(u_1,y_{i})$-path of length $2k-3-q$ denoted by $\mathcal{P}'(u_1,y_{i})$.

    When $q \equiv i \pmod{2}$, the circuit consisting of $\mathcal{P}(u_1,u_2)$, $\mathcal{P}(y_{i+2},u_2)$, $\mathcal{P}'(x_2,y_{i+2})$, $u_1x_1x_2$ has odd length $(2k+1-m) + (m-i-2)+ q +2 = 2k + 1 + q -i$.
    Thus $2k+1+q-i \ge 2k+1$, and hence $q \ge i$.
    Also, the circuit consisting of $\mathcal{P}(u_1,y_{i})$ and $\mathcal{P}'(u_1,y_{i})$ has odd length $2k-3-q+i$. Thus $2k-3 - q + i \ge 2k+1$, and hence $q \le i-2$, a contradiction.

    When $q \not\equiv i \pmod{2}$, the circuit consisting of $u_1x_1x_2$, $\mathcal{P}'(x_2,y_{i+2})$, $\mathcal{P}(u_1,y_{i+2})$ has odd length $i+q+4$.
    Thus $i+q+4 \ge 2k+1$, and hence $q \ge 2k - i -3$.
    Also the circuit consisting of $\mathcal{P}(u_1,u_2)$, $\mathcal{P}'(y_{i},u_1)$, $y_{i} y_{i+1} y_{i+2}$, $\mathcal{P}(y_{i+2},u_2)$ has odd length $(2k+1-m) +(2k-3-q) + 2 + (m-i-2) = 4k-2 -i -q$.
    Thus $4k-2 -i -q \ge 2k+1$, and hence $q \le 2k - i -3$.
    Hence $q = 2k - i -3$.
    \hfill$\blacksquare$ \par

    We will use the following consequence.
    \begin{claim}\label{claim: coro unite two paths: degree two and special path}
        Let $(u_1,u_2)$ be a $(2k + 1 - m)$-valid pair and let $\{A_1,\ldots,A_{m-1}\}$ be the $m$-path-decomposition of $(u_1,u_2)$.
        Let $A_0 = \{u_1\}$ and $A_m = \{u_2\}$.
        Suppose $x_1$ is a vertex of degree two in $A_1$ with neighbors $u_1$ and $x_2 \in A_2$.

        If $y_{i+1} \in A_{i+1}$~($y_{i+1} \neq x_1$) is a vertex of degree two with neighbors $y_i \in A_i$ and $y_{i+2} \in A_{i+2}$ such that $u_1x_1x_2$ and $y_iy_{i+1}y_{i+2}$ are contained in a common $C_{2k+1}$, then either there is a $(x_2,y_{i})$-path of length $i-2$, or there is a $(x_2,y_{i+2})$-path of length $(2k-i-3)$.

        If $y_iy_{i+1}y_{i+2}y_{i+3}$ is a special path such that $u_1x_1x_2$ and $y_iy_{i+1}y_{i+2}y_{i+3}$ are contained in a common $C_{2k+1}$, then either there is a $(x_2,y_{i})$-path of length $i-2$, or there is a $(x_2,y_{i+3})$-path of length $(2k-i-4)$.
    \end{claim}

    \begin{proof}

        \begin{figure}
            \centering
            \includegraphics[width=0.8\textwidth]{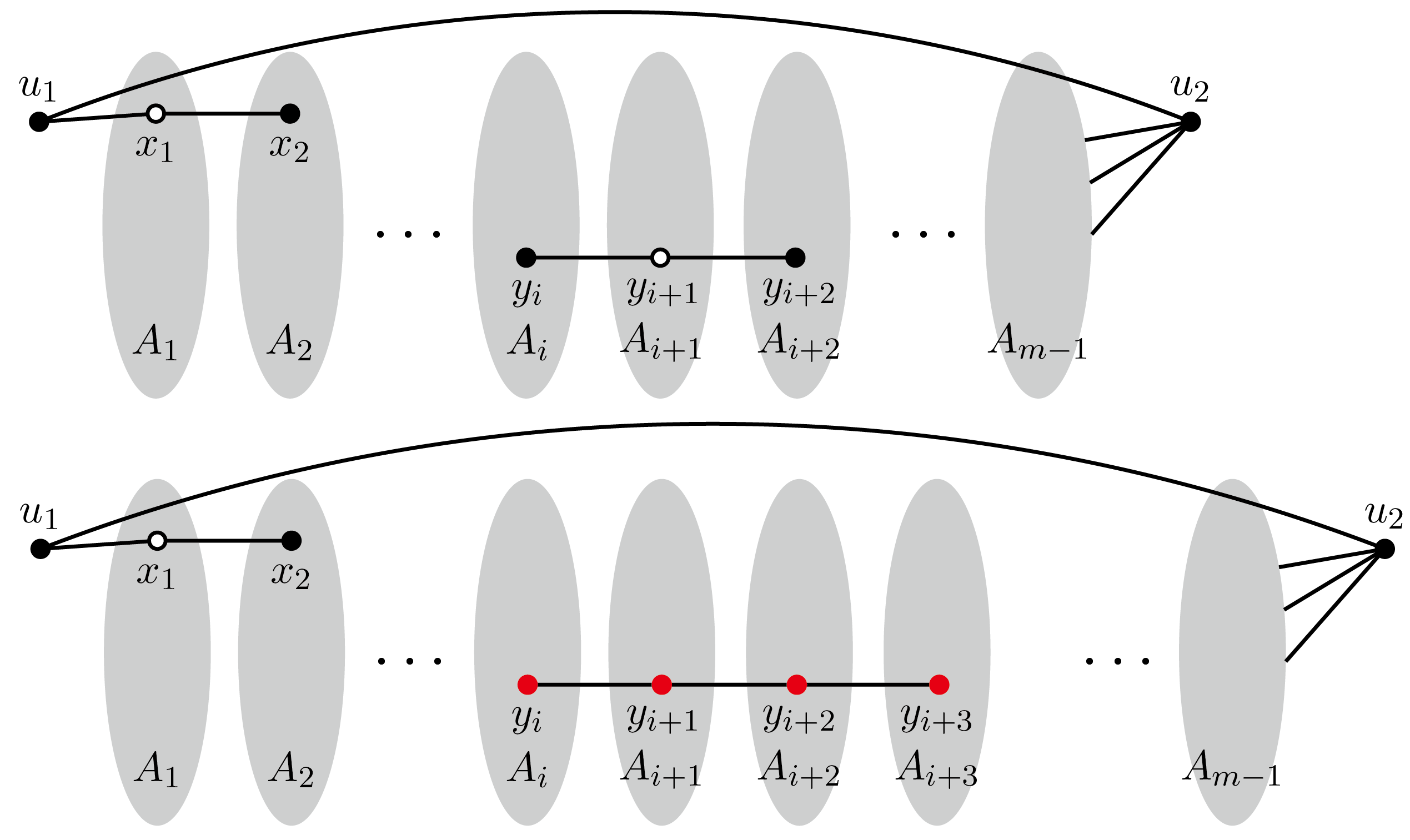}
            \caption{The hollow points represent vertices of degree two. The red vertices form a special path.}\label{fig: claim coro unite two paths: degree two and special path}
        \end{figure}
        In the case when $y_{i+1}$ is a vertex of degree two, the claim follows from Claim~\ref{claim: unite two paths in a common cycle}.
        In the case when $y_iy_{i+1}y_{i+2}y_{i+3}$ is a special path, Claim~\ref{claim: unite two paths in a common cycle} shows that either there is a $(x_2,y_{i})$-path of length $i-2$, or there is a $(x_2,y_{i+2})$-path of length $(2k-i-3)$.
        Since $y_{i+2}$ has degree three, if the neighbor of $y_{i+2}$ in the $(x_2,y_{i+2})$-path of length $(2k-i-3)$ is not $y_{i+3}$, then it must be a vertex in $A_{i+1}$, say $y_{i+1}'$.
        This means that there exists a $(x_2,y_{i+1}')$-path of length $2k-i-4$.
        Combining $u_1x_1x_2$ and the $(u_1, y_{i+1}')$-path of length $i+1$, we have a circuit of length $(2k-i-4) + (i+1) + 2 = 2k-1$, a contradiction.

        Therefore, the neighbor of $y_{i+2}$ in the $(x_2,y_{i+2})$-path of length $(2k-i-3)$ must be $y_{i+3}$.
        Then it corresponds to a $(x_2,y_{i+3})$-path of length $(2k-i-4)$.
    \end{proof}

    We now prove the lemma.
    Let $(u_1,u_2)$ be a $(2k + 1 - m)$-valid pair such that $p_m(u_1,u_2) = \Delta_m$, and let $\{A_1,A_2,\ldots,A_{m-1}\}$ be the $m$-path-decomposition of $(u_1,u_2)$.

    \textbf{Case 1:} $m$ is even and $m=2t$.

    By (\ref{eq: lb Delta m}) and Lemma~\ref{lem: order of deltas 2}, there exists a vertex $x_0 \in A_2$ such that $p_2(u_1,x_0) \ge \frac{\gamma}{2}n$ and thus by Corollary~\ref{cor: large common neighbor}, there exists a vertex $x_1 \in N(u_1) \cap N(x_0)$ of degree two.

    We claim that, for every other $x \in A_2 \setminus \{x_0\}$, either $p_2(u_1,x) \le 800k^{k}$ or $p_{m-2}(x,u_2) \le \epsilon n^{t-1}$.
    Suppose otherwise. By iteratively applying Lemma~\ref{lem: order of deltas} and Corollary~\ref{cor: large common neighbor}, there exists a path $y_1y_2\ldots y_{2t-1}$ such that $y_j \in A_j$ for all $1 \le j \le 2t-1$ and $y_j$ has degree two whenever $j$ is odd~(see the first case in Figure~\ref{fig: long path in lem: other Delta}).

    \begin{figure}
        \centering
        \includegraphics[width=0.9\textwidth]{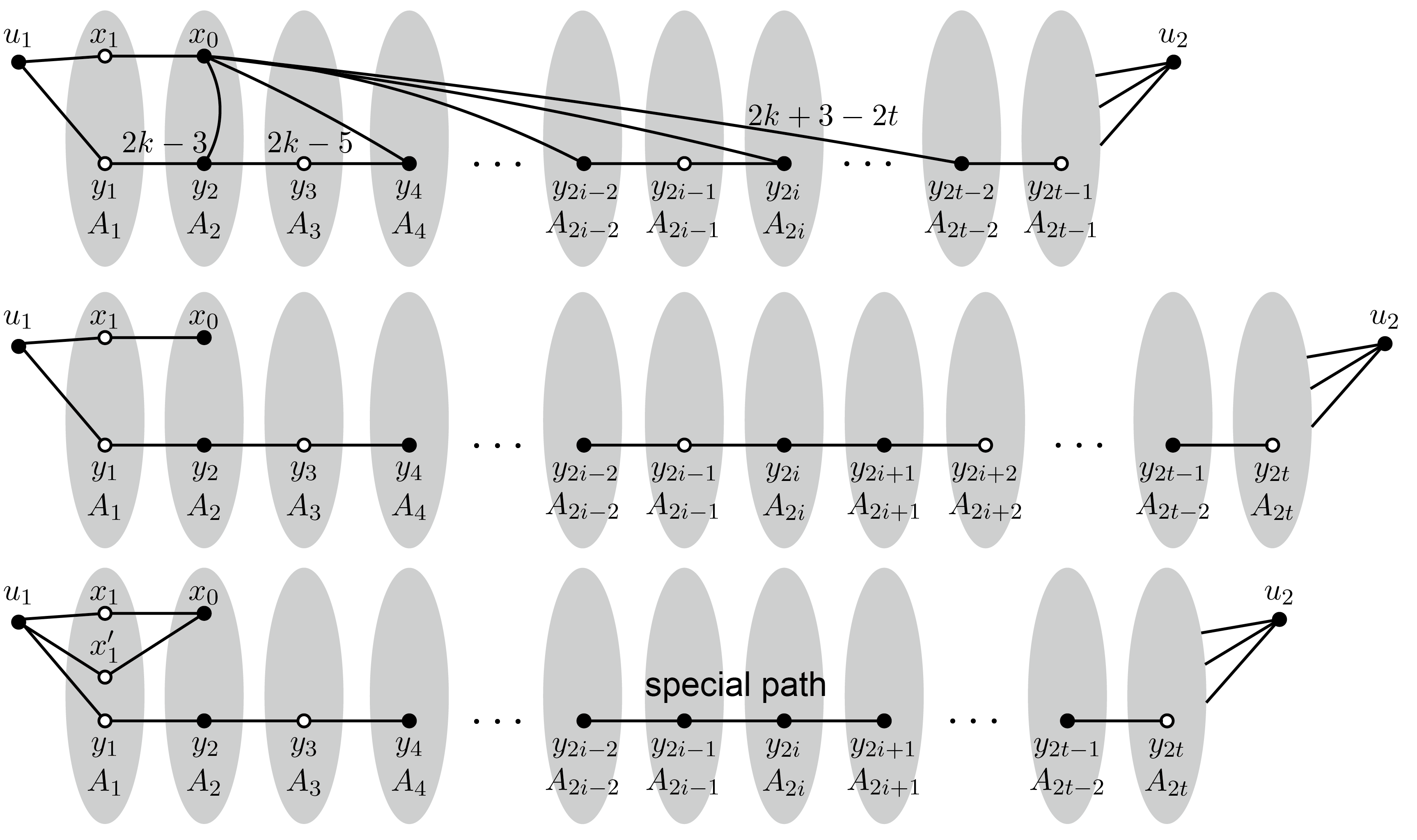}
        \caption{Three cases in the proof of Lemma~\ref{lem: other Delta}. Hollow vertices are of degree two. The path $y_{2i-2}y_{2i-1}y_{2i}y_{2i+1}$ is a special path in the third case.}\label{fig: long path in lem: other Delta}
    \end{figure}

    Since $y_1$ and $x_1$ have different neighbors, by property (2), they must be in a common $C_{2k+1}$.
    Claim~\ref{claim: coro unite two paths: degree two and special path} then gives a $(x_0,y_2)$-path of length $2k-3$.
    Applying Claim~\ref{claim: coro unite two paths: degree two and special path} again to $x_1$ and $y_3$,
    either there exists a $(x_0,y_2)$-path of length zero~(i.e., $x_0 = y_2$), or there exists a $(x_0,y_4)$-path of length $2k-5$.
    The equality $x_0=y_2$ is impossible, because then the $(x_0,y_2)$-path of length $2k-3$ gives a circuit of length $2k-3$, a contradiction.
    Hence, there exists a $(x_0,y_4)$-path of length $2k-5$.
    Repeating this argument, each application of Claim~\ref{claim: coro unite two paths: degree two and special path} to $x_1$ and $y_{2i-1}$ gives a $(x_0,y_{2i})$-path of length $2k-1-2i$.
    At the step $i = t$, we obtain a $(x_0, u_2)$-path of length $2k-1-2t$. Combining it with the $(x_0,u_2)$-path of length $2t-2$ from the definition of $A_2$, we obtain a circuit of length $2k-1$, a contradiction.

    Thus, for every $x \in A_2 \setminus \{x_0\}$, either $p_2(u_1,x) \le 800k^{k}$ or $p_{m-2}(x,u_2) \le \epsilon n^{t-1}$.
    For every $x \in A_2$, let $l(x) = p_2(u_1,x)$ and $r(x) = p_{m-2}(x,u_2)$.
    Let $E_i$ be the edge set between $A_{2i-1}$ and $A_{2i}$ for $1 \le i \le t-1$.
    Let $E_i'$ be the edge set between $A_{2i}$ and $A_{2i+1}$ for $1 \le i \le t-1$.
    We have $\sum_{x \in A_2} l(x) \le |E_1| \le n$ and $\sum_{x \in A_2} r(x) \le  \prod_{i=1}^{t-1} |E_i'| \le n^{t-1}$.
    For sufficiently large $n_4$,

    \begin{equation}\label{eq: Delta_m to Delta_m-2}
    \begin{aligned}
        p_{m}(u_1,u_2) = \Delta_m & = \sum_{x \in A_2} l(x) \cdot r(x) \\
        & \le l(x_0) r(x_0) + 800k^{k} n^{t-1} + \epsilon n^{t-1} \cdot n \\
        & \le \Delta_2 \Delta_{m-2} + 4\epsilon n^t.
    \end{aligned}
    \end{equation}

    \textbf{Case 2:} $m$ is odd and $m = 2t+1$.

    By Lemma~\ref{lem: order of deltas} and Lemma~\ref{lem: order of deltas 2}, either there exists a vertex $x_0 \in A_2$ such that $p_2(u_1,x_0) \ge \frac{\gamma}{2} n$, or there exists a vertex $x_0' \in A_{m-2}$ such that $p_{m-2}(x_0',u_2) \ge \frac{\gamma}{2} n$.
    Without loss of generality, we may assume there exists the vertex $x_0 \in A_2$ with the above property.
    Thus by Corollary~\ref{cor: large common neighbor}, there exist two vertices $x_1, x_1' \in N(u_1) \cap N(x_0)$ of degree two.
    As in Case 1, we claim that, for every other $x \in A_2 \setminus \{x_0\}$, either $p_2(u_1,x) \le 800k^k$ or $p_{m-2}(x,u_2) \le \epsilon n^{t-1}$.
    Suppose otherwise. By Corollary~\ref{cor: large common neighbor} and repeated applications of Lemma~\ref{lem: order of deltas} and Lemma~\ref{lem: vertex of degree two or long path}, we distinguish several cases.

    \textbf{Case 2.1}: There exists a path $y_1y_2\ldots y_{2t}$ such that $y_j \in A_j$ for all $1 \le j \le 2t$ and, $y_1,y_3,\ldots,y_{2i-1}, y_{2i+2}, y_{2i+4}, \ldots, y_{2t}$ are vertices of degree two~(see the second case in Figure~\ref{fig: long path in lem: other Delta}).

    By property (2), $x_1$ and $y_1$ must be in a common $C_{2k+1}$.
    Applying Claim~\ref{claim: coro unite two paths: degree two and special path}, we obtain a $(x_0,y_2)$-path of length $2k-3$.
    Repeating the argument from Case 1 gives a $(x_0, y_{2i})$-path of length $2k-1-2i$.

    By property (2), we may apply Claim~\ref{claim: coro unite two paths: degree two and special path} to $x_1$ and $y_{2i+2}$.
    If there exists a $(x_0,y_{2i+1})$-path of length $2i-1$, then we obtain a circuit of length $(2k-1-2i)+(2i-1) + 1 = 2k-1$ by combining the $(x_0, y_{2i})$-path of length $2k-1-2i$ and $y_{2i}y_{2i+1}$, a contradiction.
    Hence there exists a $(x_0,y_{2i+3})$-path of length $2k-4-2i$.
    Repeating this process and applying the claim to $x_1$ and $y_{2t}$ gives a $(x_0,u_2)$-path of length $2k-1-m$.
    Combining the $(x_0,u_2)$-path of length $m-2$ from the definition of $A_2$, we obtain a circuit of length $2k-3$, a contradiction.

    \textbf{Case 2.2}: There exists a path $y_1y_2\ldots y_{2t}$ such that $y_j \in A_j$ for all $1 \le j \le 2t$ and, $y_1,y_3,\ldots,y_{2i-3}, y_{2i+2}, y_{2i+4}, \ldots, y_{2t}$ are vertices of degree two, and $y_{2i-2}y_{2i-1}y_{2i}y_{2i+1}$ is a special path~(see the third case in Figure~\ref{fig: long path in lem: other Delta}).

    The same process as in Case 1 gives a $(x_0, y_{2i-2})$-path of length $2k+1-2i$.
    By property (3), $x_1$ and the special path $y_{2i-2}y_{2i-1}y_{2i}y_{2i+1}$ must be in a common $C_{2k+1}$.
    Applying Claim~\ref{claim: coro unite two paths: degree two and special path}, either there exists a $(x_0,y_{2i-2})$-path of length $2i-4$, or there exists a $(x_0,y_{2i+1})$-path of length $2k-2i-2$.
    In the first case, combining this path with the $(x_0,y_{2i-2})$-path of length $2k+1-2i$ gives a circuit of length $2k-3$, a contradiction.
    Hence, there exists a $(x_0,y_{2i+1})$-path of length $2k-2-2i$.
    Applying Claim~\ref{claim: coro unite two paths: degree two and special path} to $y_{2i+2}, y_{2i+4}, \ldots, y_{2t}$ as in Case 1, we finally obtain a $(x_0,u_2)$-path of length $2k-1-m$, a contradiction.

    The same argument, using Lemma~\ref{lem: vertex of degree two or long path}, shows that for every vertex $x \in A_3$, either $p_3(u_1,x) \le n^{0.2}$ or $p_{m-3}(x,u_2) \le \epsilon n^{t-1}$.
    Let $H$ be a weighted auxiliary graph induced by $A_2 \cup A_3$.
    Then $H$ is bipartite with parts $A_2$ and $A_3$.
    For every vertex $x \in A_2$, let $w(x) = p_2(u_1,x)$.
    For every vertex $y \in A_3$, let $w(y) = p_{m-3}(y,u_2)$.
    For an edge $xy \in E(H)$, let $w(xy) = w(x) w(y)$.
    Define $w(H) = \sum_{xy \in E(H)} w(xy)$.
    Let $A_2'$ be the set of vertices in $A_2\setminus \{x_0\}$ such that $p_2(u_1,x) \le 800k^k$.
    Let $A_3'$ be the set of vertices in $A_3$ such that $p_{m-3}(x,u_2) \le \epsilon n^{t-1}$.
    For a subset $S$ of $V(H)$, let $w(S) = \sum_{x \in S} w(x)$.
    Let $E_1$ be the edge set between $A_1$ and $A_2$, $E_i'$ be the edge set between $A_{2i-1}$ and $A_{2i}$ for $i=1,2,\ldots, t-1$.
    By definition and Observation~\ref{obs: forest between Ai and Ai+1}, $w(A_2) \le |E_1| \le n$ and $w(A_3) \le \prod_{i=2}^{t-1} |E_i'| \le n^{t-2}$.
    By the previous statements, for any vertex $x \in A_2\setminus \left(A_2' \cup \{x_0\}\right)$, $w(N_H(x)) \le p_{m-2}(x,u_2) \le \epsilon n^{t-1}$.
    Similarly, for any vertex $y \in A_3\setminus A_3'$, $w(N_H(y)) \le p_3(u_1,y) \le n^{0.2}$.
    Let $H'$ be the weighted subgraph of $H$ induced by $A_2' \cup A_3'$.
    Combining Lemma~\ref{lem: bi tree} with the preceding bounds gives
    \begin{equation*}
    \begin{aligned}
        w(H) & \le w(x_0)w(N(x_0)) + \sum_{x \in A_2\setminus \left(A_2' \cup \{x_0\}\right)}w(x)w(N_H(x)) + \sum_{y \in A_3\setminus A_3'}w(y)w(N_H(y)) + w(H')  \\
        & \le \Delta_2 \Delta_{m-2} + 2\epsilon n^t + w(A_2) \epsilon n^{t-1} + w(A_3) 800k^k  \\
        & \le \Delta_2 \Delta_{m-2} + 4\epsilon n^t.
    \end{aligned}
    \end{equation*}

    Thus $\Delta_m \le \Delta_2 \Delta_{m-2} + 4\epsilon n^{\lfloor m/2 \rfloor}$ for both even and odd $m$.
    Combining this inequality with $\Delta_2,\Delta_3 \le n $ from Lemma~\ref{lem: size lower bound of A}, and then applying induction, yields for $4 \le m \le 2k-3$ that
    \begin{equation*}
        \Delta_m \le \left\{ \begin{array}{ll}
            \Delta_2^{m/2} + 2m \epsilon n^{m/2}, & \text{if $m$ is even};\\
            \Delta_2^{(m-3)/2} \Delta_3 + 2m \epsilon n^{(m-1)/2}, & \text{if $m$ is odd}.
        \end{array} \right.
    \end{equation*}
\end{proof}

% -------------------------------------------------------------
\subsection{A large \texorpdfstring{$\Delta_2$}{Δ2}}\label{subsec: large Delta2}

In this subsection, we show that $\Delta_2$ is large enough.

\begin{lemma}\label{lem: big Delta2}
    Let $G$ be a planar graph on $n$ vertices forbidding $\mathcal{C}^o_{<2k+1}$ with property (1), (2) and (3).
    Let $\epsilon, \epsilon'$ be fixed small constants with $\epsilon \ll \epsilon'$.
    Then there exists $n_5 = n_5(\epsilon, \epsilon', k)$ such that when $n \ge n_5$, we have
    \[
        \Delta_2 \ge \frac{1}{k}n - \epsilon' n.
    \]
\end{lemma}

\begin{proof}
    We distinguish two cases.

    \textbf{Case 1:} $\Delta_2 < \Delta_3 - 2\epsilon n$.

    Let $(u_1,u_2)$ be a $(2k-2)$-valid pair with $p_3(u_1,u_2) = \Delta_3$.
    Let $\{A_1, A_2, \ldots, A_{2k-3}\}$ be the $(2k-2)$-path-decomposition of $(u_1,u_2)$.
    Let $\{B_1,B_2\}$ be the $3$-path-decomposition of $(u_1,u_2)$.
    By Lemma~\ref{lem: t-decomp and 2k+1-t-decomp}, $A_i$ and $B_j$ are disjoint for all $1 \le i \le 2k-3$ and $1 \le j \le 2$.
    Let $B = B_1 \cup B_2$. Lemma~\ref{lem: size lower bound of A} gives $|B| \ge \Delta_3$.
    Let $A = \bigcup_{i=1}^{2k-3} A_i$. Then $\Delta_3 + |A| \le |A| + |B| \le n$.
    Let $E_i$ be the edge set between $A_{2i-1}$ and $A_{2i}$, $1 \le i \le k-2$ and denote $e_i = |E_i|$.

    Let $H$ be the weighted bipartite subgraph $G[A_{2k-4}, A_{2k-3}]$.
    For every $x \in A_{2k-4}$, let $w(x) = p_{2k-4}(u_1,x)$.
    For every $y \in A_{2k-3}$, let $w(y) = 1$.
    For a vertex subset $S$, let $w(S) = \sum_{v \in S} w(v)$.
    Define $w(H) = \sum_{xy \in E(H)} w(x) w(y)$.
    By definition, $\sum_{x \in A_{2k-4}} w(x) \le \prod_{i=1}^{k-2} e_i$ and $\sum_{y \in A_{2k-3}} w(y) = |A_{2k-3}|$.
    Moreover, $w(H) = p_{2k-2}(u_1,u_2)$.
    By Lemma~\ref{lem: large star of long path}, either there exists a special path between $B_1$ and $B_2$, or
    \[
        \max_{x \in B_1, y\in B_2} \{|N(x) \cap N(u_2)|, |N(y)\cap N(u_1)|\} \ge p_3(u_1,u_2) - 2\epsilon n = \Delta_3 - 2\epsilon n.
    \]
    The second case is impossible since $\Delta_2 < \Delta_3 - 2\epsilon n$.
    Therefore Lemma~\ref{lem: with long path} gives $p_{2k-2}(u_1,u_2) \ge (1-\epsilon){\left( \frac{n-3\ell}{k} \right)}^{k-1}$.
    By Lemma~\ref{lem: other Delta}, we have the following inequalities.

    \begin{equation*}
        \begin{aligned}
            (1-\epsilon){\left( \frac{n-3\ell}{k} \right)}^{k-1} \le w(H) & \le \sum_{x \in A_{2k-4}} w(x) w(N_H(x)) \\
            & \le \sum_{x \in A_{2k-4}} w(x) \Delta_2  \\
            & \le \Delta_2 \left( e_1 e_2 \ldots e_{k-2} \right),
        \end{aligned}
    \end{equation*}
    and
    \begin{equation*}
    \begin{aligned}
        (1-\epsilon){\left( \frac{n-3\ell}{k} \right)}^{k-1} \le w(H) & \le \sum_{y \in A_{2k-3}} w(y) w(N_H(y)) \\
            & \le \Delta_{2k-3} \sum_{y \in A_{2k-3}} w(y) \\
            & \le |A_{2k-3}| \Delta_2^{k-3}\Delta_3 +  |A_{2k-3}| 4k \epsilon n^{k-2}.
    \end{aligned}
    \end{equation*}
    Also,
    \begin{equation*}
    \begin{aligned}
        \Delta_3 + e_1 + e_2 + \cdots + e_{k-2} + |A_{2k-3}| \le \Delta_3 + |A| \le n.
    \end{aligned}
    \end{equation*}
    Combining the above inequalities and using the AM--GM inequality, we have
    \begin{equation*}
    \begin{aligned}
        {(1-\epsilon)}^{2}   {\left( \frac{n-3\ell }{ k}\right)}^{2(k-1)} \le {w(H)}^2 & \le \Delta_2^{k-2} \Delta_3 (e_1 e_2 \ldots e_{k-2}) |A_{2k-3}| + 4k \epsilon n^{2k-2}\\
        & \le \Delta_2^{k-2} {\left( \frac{n}{k} \right)}^{k} + 4k \epsilon n^{2k-2}.
    \end{aligned}
    \end{equation*}
    Since $\epsilon \ll \epsilon'$, we have
    \begin{equation*}
    \begin{aligned}
        \Delta_2 \ge \frac{1}{k}n - \epsilon' n.
    \end{aligned}
    \end{equation*}

    \textbf{Case 2:} $\Delta_2 \ge \Delta_3 - 2\epsilon n$.

    Let $(u_1,u_2)$ be a $(2k-1)$-valid pair with $p_2(u_1,u_2) = \Delta_2$.
    Then by Corollary~\ref{cor: large common neighbor}, and property (1), $p_{2k-1}(u_1,u_2) \ge {\left(\frac{n-2k-2}{k}\right)}^{k-1}$.
    Let $\{A_1, A_2, \ldots, A_{2k-2}\}$ be the $(2k-1)$-path-decomposition of $(u_1,u_2)$ .
    Let $A = \bigcup_{i=1}^{2k-2} A_i$. The same argument as in Case 1 gives $|A| + \Delta_2 \le n$.
    Let $E_i$ be the edge set between $A_{2i-1}$ and $A_{2i}$, $1 \le i \le k-1$. Let $e_i = |E_i|$.

    Let $H$ be the weighted bipartite subgraph $G[A_2, A_3]$.
    For every $x \in A_2$, let $w(x) = p_2(u_1,x)$.
    For every $y \in A_3$, let $w(y) = p_{2k-4}(y,u_2)$.
    For a vertex subset $S$, let $w(S) = \sum_{v \in S} w(v)$.
    Define $w(H) = \sum_{xy \in E(H)} w(x) w(y)$.
    By definition, we have $w(H) = p_{2k-1}(u_1,u_2) \ge {\left(\frac{n-2k-2}{k}\right)}^{k-1}$.
    As in Case 1, $\sum_{x \in A_2} w(x) \le e_1$ and $\sum_{y \in A_3} w(y) \le \prod_{i=2}^{k-1} e_i$.
    Lemma~\ref{lem: other Delta} gives the following inequalities.

    \begin{equation*}
    \begin{aligned}
        {\left(\frac{n-2k-2}{k}\right)}^{k-1} \le w(H) & \le \sum_{x \in A_2} w(x) w(N_H(x)) \\
            & \le \sum_{x \in A_2} w(x) \cdot \Delta_{2k-3} \\
            & \le e_1 \left( \Delta_2^{k-3}\Delta_3 + 4k\epsilon n^{k-2} \right),
    \end{aligned}
    \end{equation*}
    and
    \begin{equation*}
    \begin{aligned}
        {\left(\frac{n-2k-2}{k}\right)}^{k-1} \le w(H) & \le \sum_{y \in A_3} w(y) w(N_H(y)) \\
            & \le \sum_{y \in A_3} w(y) \cdot \Delta_{3} \\
            & \le \Delta_2 e_2 e_3 \ldots e_{k-1}.
    \end{aligned}
    \end{equation*}
    Also,
    \begin{equation*}
    \begin{aligned}
        \Delta_2 + e_1 + e_2 + \cdots + e_{k-1} \le \Delta_2 + |A| \le n.
    \end{aligned}
    \end{equation*}
    Combining the above inequalities and using the AM--GM inequality, we have
    \begin{equation*}
    \begin{aligned}
         {\left( \frac{n-2k-2 }{ k}\right)}^{2(k-1)} \le {w(H)}^2 & \le \left(e_1 \left( \Delta_2^{k-3}\Delta_3 + 4k\epsilon n^{k-2} \right)\right) \Delta_2 e_2 e_3 \ldots e_{k-1} \\
         &\le \Delta_2^{k-2} {\left( \frac{n}{k}\right)}^{k} + 4k \epsilon n^{2k-2}.
    \end{aligned}
    \end{equation*}
    Since $\epsilon \ll \epsilon'$, we have
    \begin{equation*}
    \begin{aligned}
        \Delta_2 \ge \frac{1}{k}n - \epsilon' n.
    \end{aligned}
    \end{equation*}
\end{proof}

% -------------------------------------------------------------
\subsection{Bootstrap: complete the proof}\label{subsec: Bootstrap}

In Section~\ref{subsec: prep Delta 2 Delta 3}, we determined the order of magnitude of $\Delta_m$ for every $m = 2,3,\ldots,2k-1$.
Combining this with the lower bound $\Delta_2 \ge \frac{1}{k}n - \epsilon' n$, we now obtain much sharper estimates for $\Delta_m$.

\begin{lemma}\label{lem: bootstrap 1}
    Let $\xi$ be a fixed small constant, $1 \le t \le k-2$.
    Let $G$ be a planar graph on $n$ vertices forbidding $\mathcal{C}^o_{<2k+1}$, with property (1), (2) and (3).
    There exists a constant $n_6 = n_6(\xi, k)$ such that, whenever $n \ge n_6$ and $\epsilon' \ll \xi$, the following holds.
    Let $(u_1,u_2)$ be both a $2t$-valid pair and a $(2k-2t+1)$-valid pair with $p_{2k-2t+1}(u_1, u_2) \ge \frac{n^{k-t}}{k^{k-t}} - \epsilon' n^{k-t}$.
    Let $\{A_1, A_2, \ldots, A_{2k-2t}\}$ be the $(2k-2t+1)$-path-decomposition of $(u_1,u_2)$.
    Let $A = \bigcup_{i=1}^{2k-2t}A_i$ and assume $|A| \le \frac{k-t}{k}n + \epsilon' n$.
    Then either there exists a vertex $w$ in $A_2$ such that
    \begin{enumerate}
        \item $p_2(u_1, w) = |N(w) \cap N(u_1)| \ge \frac{1}{k}n - 2\xi n$.
        \item $p_{2k-2t-1}(w, u_2) \ge \frac{n^{k-t-1}}{k^{k-t-1}} - 2\xi n^{k-t-1}$.
        \item $|A'| \le \frac{k-t-1}{k}n + 2\xi n$, where $A'$ is the union of the sets in the $(2k-2t-1)$-path-decomposition of $(w,u_2)$.
    \end{enumerate}
    or there exists a vertex $w'$ in $A_{3}$ such that
    \begin{enumerate}
        \item $p_3(w', u_1) \ge \frac{1}{k}n - 2\xi n$.
        \item $p_{2k-2t-2}(w', u_2) \ge \frac{n^{k-t-1}}{k^{k-t-1}} - 2\xi n^{k-t-1}$.
        \item $|A'| \le \frac{k-t-1}{k}n + 2\xi n$, where $A'$ is the union of the sets in the $(2k-2t-2)$-path-decomposition of $(w',u_2)$.
    \end{enumerate}
    % there exists a vertex $w'$ in $A_{2k-2t-1}$ such that
    % \begin{enumerate}
    %     \item $p_2(w', u_2) = |N(w') \cap N(u_2)| \ge \frac{1}{k}n - 2\xi n$.
    %     \item $p_{2k-2t-1}(u_1, w') \ge \frac{n^{k-t-1}}{k^{k-t-1}} - 2\xi n^{k-t-1}$.
    %     \item $|A'| \le \frac{k-t-1}{k}n + 2\xi n$, where $A'$ is the union of the sets in the $(2k-2t-1)$-path-decomposition of $(u_1, w')$.
    % \end{enumerate}
\end{lemma}

\begin{proof}
    By Lemma~\ref{lem: size lower bound of A}, we have $|A| \ge (k-t) {\left(p_{2k-2t+1}(u_1,u_2)\right)}^{1/{k-t}} \ge (k-t) \left( \frac{1}{k}n - \xi n \right)$.
    Let $E_i$ be the set of edges between $A_{2i-1}$ and $A_{2i}$, $1 \le i \le k-t$.

    \begin{claim}\label{claim: large each E_i}
        $\left| |E_i| - \frac{1}{k}n \right| \le \xi n$, for all $1 \le i \le k-t$.
    \end{claim}

    \noindent
    \textbf{Proof of Claim~\ref{claim: large each E_i}}
        We have $\sum_{i=1}^{k-t}|E_i| \le |A| \le \frac{k-t}{k}n + \epsilon' n$.
        % Thus
        % \begin{equation}\label{eq: E1 dot E2 upper bound}
        % \begin{aligned}
        %     p_5(u_1,u_2) & \le |E_1| \cdot |E_2| \\
        %     & \le {\left(\frac{1}{3}n + \epsilon' n/2\right)}^2
        % \end{aligned}
        % \end{equation}
        Let $\xi' = \xi/k$.
        % We first prove $|E_1| \ge \frac{1}{k} n - \xi' n$, and the proof for other $|E_i|$ is similar.
        Suppose otherwise that $|E_i| \le \frac{1}{k} n - \xi' n$ for some $i$.
        Then
        \begin{equation*}
        \begin{aligned}
            \frac{n^{k-t}}{k^{k-t}} - \epsilon' n^{k-t}& \le p_{2k-2t+1}(u_1,u_2) \le \prod_{j=1}^{k-t}|E_j| \le |E_i| \frac{1}{{(k-t-1)}^{k-t-1}} {\left(\frac{k-t}{k}n + \epsilon' n - |E_i|\right)}^{k-t-1} \\
            & \le \left(\frac{1}{k} n - \xi' n\right) {\left(\frac{1}{k}n + \frac{\epsilon' + \xi' }{k-t-1}n \right)}^{k-t-1} \\
            & \le \frac{n^{k-t}}{k^{k-t}} + \frac{\epsilon' }{k^{k-t-1}}n^{k-t} - \frac{{\xi'}^2}{k-t-1} n^{k-t}.
        \end{aligned}
        \end{equation*}
        This is a contradiction when $\epsilon'$ is relatively small compared to $\xi'$.
        Hence $|E_j| \ge \frac{1}{k}n  - \xi' n$ for every $j$, and therefore $|E_i| \le |A| - \sum_{j \neq i} |E_j| \le \frac{1}{k}n + \epsilon' n + (k-t-1)\xi' n \le \frac{1}{k}n + \xi n$ for all $1 \le i \le k-t$.
    \hfill$\blacksquare$\par

    Let $H$ be a weighted bipartite subgraph of $G$ induced by $A_2$ and $A_3$.
    For each $x \in A_2$, define $w(x) = p_2(u_1,x) = |N(x) \cap N(u_1)|$.
    For each $y \in A_3$, define $w(y) = p_{2k-2t-2}(y,u_2)$.
    For a vertex set $S$, define $w(S) = \sum_{v \in S} w(v)$.
    Then $w(A_2) = |E_1|$ and $w(A_3) \le \prod_{i=2}^{k-t} |E_i|$.
    Assume $x_0 \in A_2$ maximizes $w(x)$, and $y_0 \in A_3$ maximizes $w(y)$.
    \begin{claim}\label{claim: wx0 or wy0 is large}
        Either $w(x_0) \ge \frac{1}{k}n - 2\xi n$ or $w(y_0) \ge \frac{n^{k-t-1}}{k^{k-t-1}} - 2\xi n^{k-t-1}$.
    \end{claim}

    \noindent
    \textbf{Proof of Claim~\ref{claim: wx0 or wy0 is large}}
        Suppose otherwise. By Lemma~\ref{lem: bi tree},
        \begin{equation*}
        \begin{aligned}
            \frac{n^{k-t}}{k^{k-t}} - \epsilon' n^{k-t}  & \le p_{2k-2t+1}(u_1,u_2) \le w(H) \\
             & \le w(A_2)w(A_3) - \left(w(A_2) - w(x_0)\right)\left(w(A_3) - w(y_0)\right) \\
             & \le {\left( \frac{1}{k} n + \frac{\epsilon'}{k-t}n \right)}^{k-t} - \xi n \left( {\left(\frac{n}{k} - \xi n\right)}^{k-t-1} - \frac{n^{k-t-1}}{k^{k-t-1}} + 2\xi n^{k-t-1} \right)\\
             & \le {\left( \frac{1}{k} n + \frac{\epsilon'}{k-t}n \right)}^{k-t} - \xi^2 n^{k-t}
        \end{aligned}
        \end{equation*}

        This is a contradiction when $\epsilon'$ is relatively small compared to $\xi$ and $n_6$ is sufficiently large.
    \hfill$\blacksquare$\par

    \textbf{Case 1:} $w(x_0) \ge \frac{1}{k}n - 2\xi n$.
    We claim that $w(N_H(x_0)) \ge \frac{n^{k-t-1}}{k^{k-t-1}} - 2\xi n^{k-t-1}$. Otherwise,
    \begin{equation*}
    \begin{aligned}
        \frac{n^{k-t}}{k^{k-t}} - \epsilon' n^{k-t}  & \le p_{2k-2t+1}(u_1,u_2) \le w(H) \\
            & \le (w(A_2) - w(x_0))w(A_3) + w(x_0) w(N(x_0))\\
            & \le w(A_2)w(A_3) - w(x_0) (w(A_3) - w(N(x_0)))\\
            & \le {\left( \frac{1}{k} n + \frac{\epsilon'}{k-t}n \right)}^{k-t} - \left(\frac{1}{k}n -2\xi n\right)  \left( {\left(\frac{n}{k} - \xi n\right)}^{k-t-1} - \frac{n^{k-t-1}}{k^{k-t-1}} + 2\xi n^{k-t-1} \right),
    \end{aligned}
    \end{equation*}
    which is a contradiction when $\epsilon'$ is relatively small compared to $\xi$ and $n_6$ is sufficiently large.
    Taking $w = x_0$, we have $|A'| \le |A| - |E_1| \le \frac{k-t-1}{k}n + 2\xi n$.

    \textbf{Case 2:} $w(y_0) \ge \frac{n^{k-t-1}}{k^{k-t-1}} - 2\xi n^{k-t-1}$.
    By the same argument as in Case 1, $w(N_H(y_0)) \ge \frac{1}{k}n - 2\xi n$.
    Otherwise,
    \begin{equation*}
    \begin{aligned}
        \frac{n^{k-t}}{k^{k-t}} - \epsilon' n^{k-t}  & \le p_{2k-2t+1}(u_1,u_2) \le w(H) \\
            & \le (w(A_3) - w(y_0))w(A_2) + w(y_0) w(N(y_0))\\
            & \le w(A_2)w(A_3) - w(y_0) (w(A_2) - w(N(y_0)))\\
            & \le {\left( \frac{1}{k} n + \frac{\epsilon'}{k-t}n \right)}^{k-t} - \left( \frac{n^{k-t-1}}{k^{k-t-1}} - 2\xi n^{k-t-1} \right)  \left( {\left(\frac{n}{k} - \xi n\right)} - \left( \frac{1}{k}n - 2\xi n\right) \right),
    \end{aligned}
    \end{equation*}
    which is a contradiction when $\epsilon'$ is relatively small compared to $\xi$ and $n_6$ is sufficiently large.
    Taking $w' = y_0$, we have $|A'| \le |A| - |E_1| \le \frac{k-t-1}{k}n + 2\xi n$.
\end{proof}

The same argument also gives the following lemma.

\begin{lemma}\label{lem: bootstrap 2}
    Let $\xi$ be a fixed small constant.
    Let $G$ be a planar graph on $n$ vertices forbidding $\mathcal{C}^o_{<2k+1}$, with property (1), (2) and (3).
    Then there exists a constant $n_7 = n_7(\xi, k)$ such that, whenever $n \ge n_7$ and $\epsilon'$ is relatively small compared to $\xi$, the following holds.

    Let $(u_1,u_2)$ be both a $(2t+1)$-valid pair and a $(2k-2t)$-valid pair, where $1 \le t \le k-2$.
    $p_{2k-2t}(u_1, u_2) \ge \frac{n^{k-t}}{k^{k-t}} - \epsilon' n^{k-t}$.
    Let $\{A_1, A_2, \ldots, A_{2k-2t-1}\}$ be the $(2k-2t)$-path-decomposition of $(u_1,u_2)$.
    Let $A = \bigcup_{i=1}^{2k-2t-1}A_i$ and assume $|A| \le \frac{k-t}{k}n + \epsilon' n$.
    Then there exists a vertex $w$ in $A_2$ such that
    \begin{enumerate}
        \item $p_2(u_1,w) = |N(w) \cap N(u_1)| \ge \frac{1}{k}n - 2\xi n$.
        \item $p_{2k-2t-2}(w, u_2) \ge \frac{n^{k-t-1}}{k^{k-t-1}} - 2\xi n^{k-t-1}$.
        \item $|A'| \le \frac{k-t-1}{k}n + 2\xi n$, where $A'$ is the union of the sets in the $(2k-2t-2)$-path-decomposition of $(w,u_2)$.
    \end{enumerate}
\end{lemma}

\begin{proof}
    The proof follows the same argument as Lemma~\ref{lem: bootstrap 1}, except that $E_{k-t}$ is replaced by $A_{2k-2t-1}$.

    Let $E_i'$ be the set of edges between $A_{2i}$ and $A_{2i+1}$, $1 \le i \le k-t-1$.
    By Observation~\ref{obs: forest between Ai and Ai+1}, $E_i'$ is a forest and thus $|E_i'| \le |A_{2i}| + |A_{2i+1}| \le n$.
    Therefore Case 2 in the proof of Lemma~\ref{lem: bootstrap 1} is impossible, because $w(y_0) \le \prod_{i=2}^{k-t-1} |E_i'| \le n^{k-t-2}$.
    The conclusion follows from Case 1 in the proof of Lemma~\ref{lem: bootstrap 1}.
\end{proof}

By Lemma~\ref{lem: big Delta2}, let $(u_1,u_2)$ be a $2$-valid pair with $p_2(u_1,u_2) = \Delta_2 \ge \frac{1}{k}n - \epsilon' n$.
Then by Corollary~\ref{cor: large common neighbor} and property (1), $(u_1,u_2)$ is also $(2k-1)$-valid with $p_{2k-1}(u_1,u_2) \ge \left( \frac{n-2k-1}{k} \right)^{k-1} \ge \frac{n^{k-1}}{k^{k-1}} - \epsilon' n^{k-1}$.
Let $\{A_1, A_2, \ldots, A_{2k-2}\}$ be the $(2k-1)$-path-decomposition of $(u_1,u_2)$ and $A = \bigcup_{i=1}^{2k-2}A_i$.
Let $\{B_1, B_2\}$ be the $3$-path-decomposition of $(u_1,u_2)$ and $B = B_1 \cup B_2$.
Lemma~\ref{lem: size lower bound of A} gives $|B| \ge p_2(u_1,u_2) = \Delta_2$.
By Lemma~\ref{lem: t-decomp and 2k+1-t-decomp}, $|A| \le n-|B| \le \frac{k-1}{k}n + \epsilon' n$.
Then $(u_1,u_2)$ satisfies the conditions of Lemma~\ref{lem: bootstrap 1} for $t=1$.

Starting from $(u_1,u_2)$ and iteratively applying Lemmas~\ref{lem: bootstrap 1} and~\ref{lem: bootstrap 2}, we obtain the following. For every constant $\xi$ relatively small compared to $1/k$, there exists $n_8 = n_8(\xi, k)$ such that, whenever $n \ge n_8$, there are vertices $u_1,\ldots, u_k$ with $|N(u_i) \cap N(u_{i+1})| \ge \frac{1}{k}n - \xi n$ for $i=1,\ldots,k-1$ and $p_{3}(u_1,u_k) \ge \frac{1}{k} n - \xi n$.

It remains to prove the following lemma.
\begin{lemma}\label{lem: finally}
    Let $G$ be a planar graph on $n$ vertices forbidding $\mathcal{C}^o_{<2k+1}$, with property (1), (2) and (3).
    Suppose there exists a small constant $\xi$ and vertices $u_1,u_2,\ldots, u_k$ such that $p_2(u_i,u_{i+1}) \ge \frac{1}{k}n - \xi n$ for $i=1,2,\ldots, k-1$ and $p_3(u_1,u_k) \ge \frac{1}{k}n - \xi n$, then there exists a constant $n_9 = n_9(\xi, k)$ such that the number of copies of $C_{2k+1}$ in $G$ is at most $h_k(n)$ when $n \ge n_9$.
\end{lemma}

\begin{proof}

    Let $W_i = N(u_i) \cap N(u_{i+1})$, $i=1,2,\ldots,k-1$.
    By definition, $\frac{1}{k}n -\xi n \le |W_i| \le \frac{1}{k}n + k\xi n$.
    Let $\{A_1,A_2\}$ be the $3$-path-decomposition of $(u_1,u_k)$.
    Let $A = A_1 \cup A_2$.
    Then $|W_i| \ge \frac{1}{k}n - \xi n$, $i=1,2,\ldots,k-1$, and $|A| \ge \frac{1}{k}n - \xi n$.
    Let $W = \bigcup_{i=1}^{k-1} W_i$.
% We may assume the size of $W\cup A$ is maximized with respect to the choice of $u_1,u_2,\ldots,u_k$.
Choose arbitrary vertices $w_i \in W_i$, $i=1,2,\ldots,k-1$, and an edge between $A_1$ and $A_2$. These choices give a cycle containing $u_1,u_2,\ldots,u_k$.
The cycle divides the plane into two regions.
We call the bounded region the \textbf{interior} and the unbounded region the \textbf{exterior}.
For each $i=1,2,\ldots,k-1$, let $w_i^{\mathrm{int}}$ be the innermost vertex in $W_i$, that is, there is no vertex in $W_i$ in the interior region.
Similarly, define $a_1^{\mathrm{int}} \in A_1$ and $a_2^{\mathrm{int}} \in A_2$ as the innermost vertices in $A_1$ and $A_2$.
Let $R^{\mathrm{int}}$ be the interior region bounded by the cycle $u_1w_1^{\mathrm{int}}u_2 w_2^{\mathrm{int}}\ldots u_k a_2^{\mathrm{int}} a_1^{\mathrm{int}} u_1$.
Similarly, define $w_i^{\mathrm{ext}}$, $a_1^{\mathrm{ext}}$, and $a_2^{\mathrm{ext}}$, and let $R^{\mathrm{ext}}$ be the exterior region~(see Figure~\ref{fig: lemma finally}).

\begin{figure}
    \centering
    \includegraphics[width=0.9\linewidth]{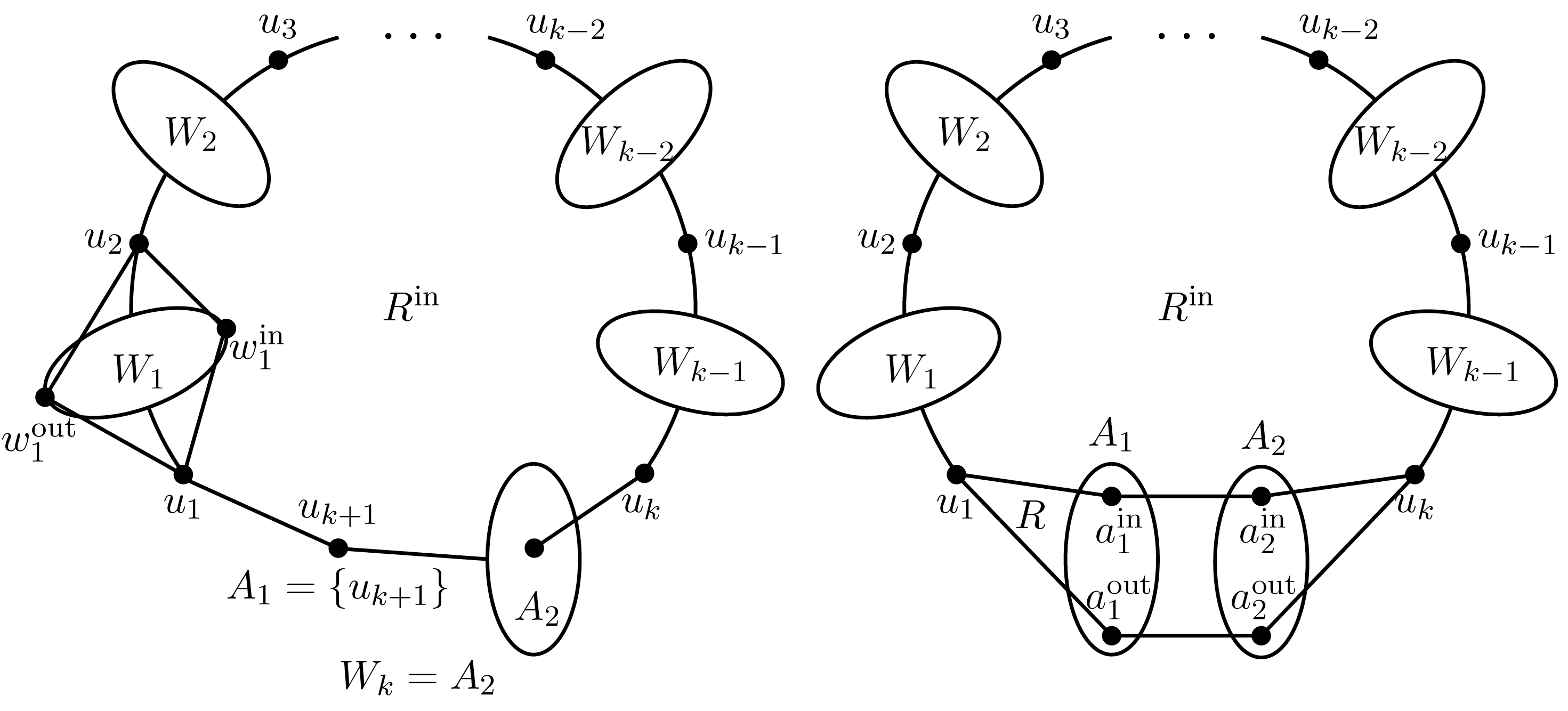}
    \caption{The structure of the graph in Lemma~\ref{lem: finally} and Proposition~\ref{prop: finally}.}~\label{fig: lemma finally}
\end{figure}

Let $U = \{u_1,u_2,\ldots,u_k\}$.
Let $X$ be the set of remaining vertices, i.e.,
\begin{equation}\label{eq: X}
    X = V \setminus \left( U \cup A \cup \bigcup_{i=1}^{k-1} W_i \right).
\end{equation}
We have $|X| \le k\xi n$.

Define $W_i^{\circ} = W_i \setminus \{w_i^{\mathrm{int}}, w_i^{\mathrm{ext}}\}$, $i=1,2,\ldots,k-1$.
Similarly, we can define $A_i^{\circ} = A_i \setminus \{a_i^{\mathrm{int}}, a_i^{\mathrm{ext}}\}$, $i=1,2$.
By Lemma~\ref{lem: 4-cycle empty}, we know the $4$-cycle formed by any two consecutive vertices in $W_i^{\circ}$ and $u_i,u_{i+1}$ is empty.
By Lemma~\ref{lem: 4-cycle empty} and Lemma~\ref{lem: 6-cycle empty}, the regions formed by $u_1,u_k$ and vertices in $A_1,A_2$ are also empty.
Thus all vertices in $X$ lie in either $R^{\mathrm{int}}$ or $R^{\mathrm{ext}}$~(excluding the boundary).

We first prove the following proposition.
\begin{proposition}\label{prop: finally}
    Let $G$ be a planar graph satisfying the structural hypotheses of Lemma~\ref{lem: finally}, but without assuming properties (1), (2) and (3).
    Let $X$ be the set of vertices defined in (\ref{eq: X}) with $|X| \le k\xi n$, and assume all vertices in $X$ lie in either $R^{\mathrm{ext}}$ or $R^{\mathrm{int}}$~(excluding the boundary).
    Let $x$ be a vertex in $R^{\mathrm{int}}$~(the case $x \in R^{\mathrm{ext}}$ is symmetric), and let $x_1,x_2$ be two vertices in $N(x)$.
    Assume $x$ does not lie on any $(a_1^{\mathrm{int}}, u_2)$-path, $(a_1^{\mathrm{ext}}, u_2)$-path, $(a_2^{\mathrm{int}}, u_{k-1})$-path, $(a_2^{\mathrm{ext}}, u_{k-1})$-path of length three.
    Then the number of $C_{2k+1}$ containing $x_1 x x_2$ is at most $\xi {(2k)}^{5k} n^{k-1}$.
\end{proposition}

\noindent
\textbf{Proof of Proposition~\ref{prop: finally}}

\textbf{Case 1:} $|A_1| = 1$~(or equivalently $|A_2| = 1$).
Assume $A_1 = \{u_{k+1}\}$~(see the left part of Figure~\ref{fig: lemma finally}).
Let $W_k = A_2$ and similarly we can define $W_k^{\circ} = A_2\setminus \{a_2^{\mathrm{int}}, a_2^{\mathrm{ext}}\}$.

Let $J$ be a subset of $\{1,2,\ldots,k\}$.
We say a cycle containing $x_1 x x_2$ is of type $J$ if for each $j \in J$, the cycle contains a vertex in $W_j^{\circ}$, and for each $j \notin J$, the cycle does not contain any vertex in $W_j^{\circ}$.
If the cycle contains a vertex in $W_j^{\circ}$, then it must contain $u_j$ and $u_{j+1}$.
Then each cycle contains at most one vertex in each $W_j^{\circ}$.

We estimate the number of cycles of type $J$ for a fixed $J$.
The cycle has length $2k+1$, and the two edges $x_1 x$ and $x x_2$ are already fixed. If $j \in J$, then there are $|W_j^{\circ}|$ choices for the vertex $w_j \in W_j^{\circ}$, and the two edges $u_j w_j$ and $w_j u_{j+1}$ are then fixed.
Thus $2 + 2|J|$ edges are fixed, and clearly $|J| \le k-1$.
There are at most $(2k)!$ cyclic orders for these fixed edges. It remains to choose the other $2k+1 - (2 + 2|J|) = 2(k - |J|) - 1$ edges. Now suppose we have a cycle $c_1c_2\ldots c_{2k+1}$ of type $J$ with fixed $2+2|J|$ edges.
If $c_{i}$ and $c_{i+1}$ are fixed vertices, then the edge $c_{i}c_{i+1}$ is also fixed.
Therefore, the undetermined edges must be disjoint paths of length at least two between the fixed vertices.
Let $P_1,P_2,\ldots,P_s$ be the paths formed by the undetermined edges.
Denote by $|P_i|$ the number of edges in $P_i$.
The middle vertices of each $P_i$ can only be vertices in $X$ and the boundaries of $R^{\mathrm{int}}$ and $R^{\mathrm{ext}}$.
By Lemma~\ref{lem: size lower bound of A}, the number of choices for each $P_i$ is at most ${(k \xi n + 4k)}^{\lceil (|P_i|-1)/2 \rceil}$.
We have $\sum_{i=1}^s |P_i| = 2(k - |J|) - 1$ and $|P_i| \ge 2$ for each $1 \le i \le s$.
Moreover, $\lceil (|P_i|-1)/2 \rceil = |P_i|/2$ only when $|P_i|$ is even.
At least one $P_i$ is of odd length since $\sum_{i=1}^s |P_i|$ is odd.

The total number of choices for all undetermined edges is at most
\begin{equation*}
\begin{aligned}
    \prod_{i=1}^s {(k \xi n+4k)}^{\lceil (|P_i|-1)/2 \rceil} & = {(k \xi n+4k)}^{\sum_{i=1}^s \lceil (|P_i|-1)/2 \rceil} \\
    &\le {(2k \xi n)}^{\frac{\sum_{i=1}^s (|P_i|-1)}{2} + \frac{s-1}{2}} \\
    & \le {(2k \xi n)}^{\frac{2(k - |J|) - 1 - s}{2} + \frac{s-1}{2}} = {(2k \xi n)}^{k - |J| - 1}.
\end{aligned}
\end{equation*}

The total number of cycles of type $J$ with $|J| < k-1$ containing $x_1 x x_2$ is at most
\begin{equation*}
\begin{aligned}
    (2k)! \left( \prod_{j \in J} |W_j| \right) {(2k \xi n)}^{k - |J| - 1} & \le (2k)! {\left( \frac{n}{k} + k\xi n \right)}^{|J|} {(2k \xi n)}^{k - |J| - 1} \\
    & \le (2k)! \xi (2k)^k n^{k-1}.
    % & \le (2k)! \xi \left( \frac{1}{k^{|J|}} + 2^{k} \xi \right) n^{k-1} \\
    % & \le \xi (2k)! k^k n^{k-1}.
\end{aligned}
\end{equation*}

If $|J| = k-1$, then there is exactly one $j_0$ not in $J$.
If $j_0 \in \{2,\ldots,k-1\}$, then the only possibility is $\{x_1,x_2\} = \{u_{j_0}, u_{j_0+1}\}$.
Then $x \in W_{j_0}$, a contradiction.
If $j_0 = k$, then $x$ is contained in a $(u_1,u_{k-1})$-path of length three. Then $x \in A_1\cup A_2$, a contradiction.
If $j_0 = 1$, then $x$ is contained in a $(u_{2},u_{k+1})$-path of length three, which contradicts our assumption.

Summing over all $J \subseteq \{1,2,\ldots,k\}$, the number of cycles containing $x_1 x x_2$ is at most $\xi {(2k)}^{4k} n^{k-1}$, which is stronger than the required bound for this case.

\textbf{Case 2}: $|A_1|, |A_2| \ge 2$.

We distinguish the following three types of such cycles:
\begin{enumerate}[label=Type \arabic*:, leftmargin=4em]
    \item It contains at most one vertex in $A_1$;
    \item It contains at most one vertex in $A_2$;
    \item It contains at least two vertices in both $A_1$ and $A_2$.
\end{enumerate}

For the first type, let $G'$ be the graph obtained from $G$ by identifying all vertices in $A_1$ into a single vertex $a_1$ and adding $|A_1| -1$ common neighbors of $a_1$ and $u_k$.
Equivalently, delete $A_1$, add the vertex $a_1$, join $a_1$ to the neighbors of every vertex $v \in A_1$, and then add $|A_1| -1$ common neighbors of $a_1$ and $u_k$.
Since the regions in $u_1xyu_kx'y'$ only contain vertices in $A_1$ and $A_2$, $G'$ is also a planar graph.
Moreover, every cycle of Type 1 in $G$ corresponds to a cycle containing $x_1xx_2$ in $G'$.
One checks that $G'$ also satisfies the conditions in Proposition~\ref{prop: finally}.
By Case 1, the number of $C_{2k+1}$ of Types 1 is at most $\xi {(2k)}^{4k} n^{k-1}$. Similarly, the number of $C_{2k+1}$ of Types 2 is also at most $\xi {(2k)}^{4k} n^{k-1}$.

\begin{figure}
    \centering
    \includegraphics[width=0.5\linewidth]{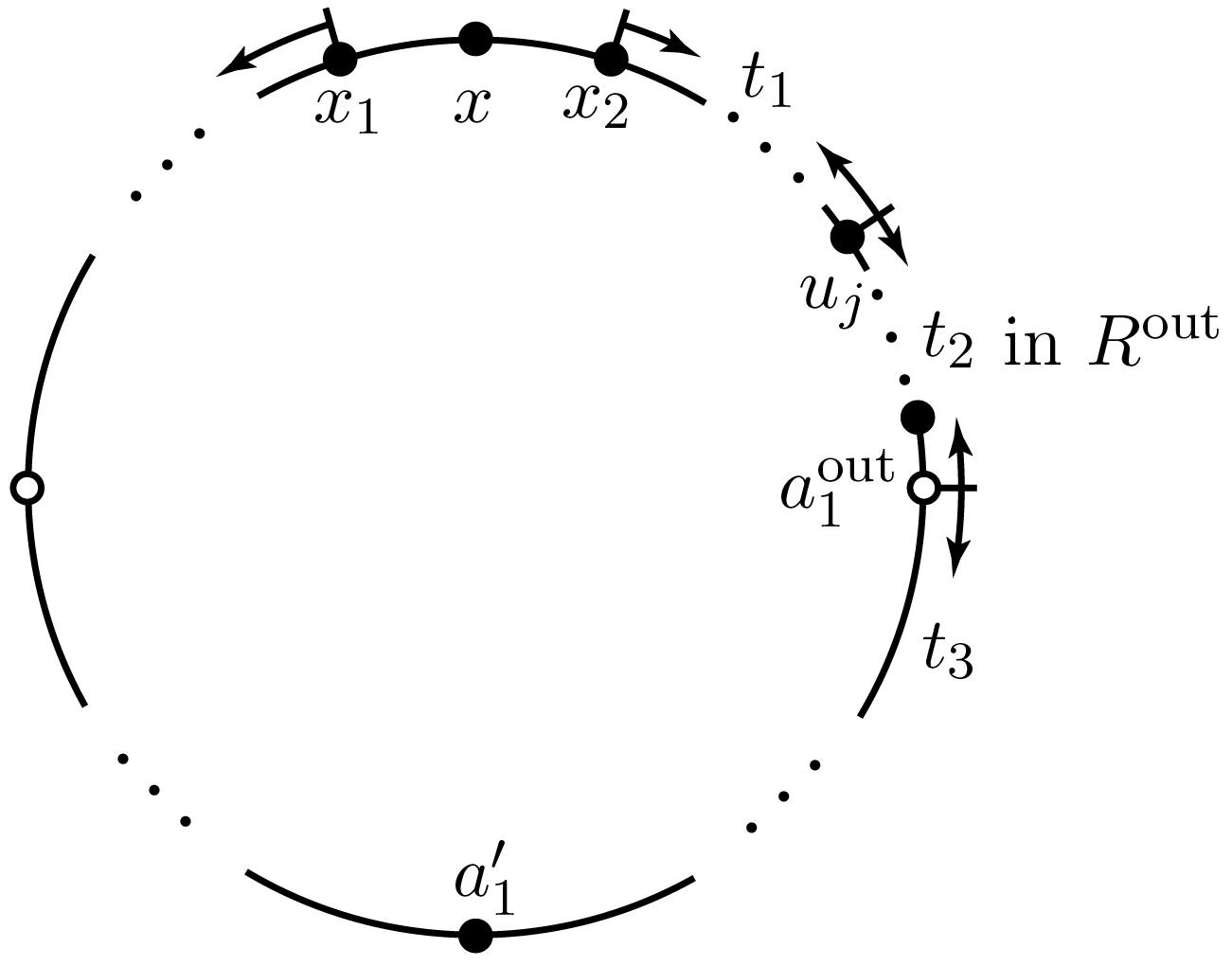}
    \caption{The cycle containing $v$ and $a_1'$. The hollow vertices are special vertices.}~\label{fig: special vertices}
\end{figure}

Now consider cycles of Type 3~(see the right figure in Figure~\ref{fig: lemma finally}).
We claim that such a cycle must contain either $a_1^{\mathrm{ext}}$ or $a_2^{\mathrm{ext}}$.
Consider the region $R$ formed by $u_1a_1^{\mathrm{int}}a_2^{\mathrm{int}}u_k a_2^{\mathrm{ext}} a_1^{\mathrm{ext}} u_1$.
Since there are at least two vertices in $A_1$ on the cycle, at least one of them, say $a_1'$, must lie in the interior of the region~(otherwise $a_1^{\mathrm{ext}}$ is in the cycle and we are done).
Since the cycle contains a vertex in the interior of the region and a vertex $x$ in the exterior of the region, at least two non-adjacent vertices on the boundary of the region must be in the cycle.
The vertices $x$ and $a_1'$ subdivide the cycle into two arcs~(see Figure~\ref{fig: special vertices}).
For each arc, choose a \textbf{special vertex} on the boundary of $R$ that lies on the arc and has maximum distance from $x$ along that arc.
Then these two vertices must be non-adjacent since the cycle is induced.
If $u_1$ is one of the two special vertices, then at least one neighbor of $u_1$ in the cycle must lie in the exterior of the region by the definition of special vertices.
However, since there are at least two vertices in $A_1$ in the cycle, the two neighbors of $u_1$ must be in $A_1$, a contradiction.
Thus, the two special vertices cannot be $u_1$ and $u_k$ by a similar argument.
Then one of $a_1^{\mathrm{ext}}, a_2^{\mathrm{ext}}$ is a special vertex since $a_1^{\mathrm{int}}$ and $a_2^{\mathrm{int}}$ are adjacent to each other.
Without loss of generality, we may assume $a_1^{\mathrm{ext}}$ is a special vertex~(see Figure~\ref{fig: special vertices}).
Then a neighbor of $a_1^{\mathrm{ext}}$ must lie in the exterior of the region, and hence must lie in $R^{\mathrm{ext}}$.
    The arc between $x$ and $a_1^{\mathrm{ext}}$ contains a vertex in $R^{\mathrm{ext}}$, so at some point this arc must cross the boundary of $R^{\mathrm{int}}$; assume it is $u_j$~(it cannot be a vertex in any $W_i$ because $|W_i| \gg 2k$).
Then the cycle consists of three parts~(see Figure~\ref{fig: special vertices}): the arc between $x_2$~(or equivalently $x_1$) and $u_j$ of length $t_1$, the arc between $u_j$ and $a_1^{\mathrm{ext}}$ of length $t_2$ in $R^{\mathrm{ext}}$, and the arc between $a_1^{\mathrm{ext}}$ and $x_1$~(or equivalently $x_2$) of length $t_3$.
We have $t_1 + t_2 + t_3 = 2k-1$ and $t_2 \ge 2$.
By Lemma~\ref{lem: size lower bound of A}, the number of cycles of Type 3 is at most
\begin{equation*}
\begin{aligned}
    &4k \sum_{t_1 + t_2 + t_3 = 2k-1} n^{\lceil (t_1-1)/2 \rceil} \cdot {\left( \xi n\right)}^{\lceil (t_2-1)/2 \rceil} \cdot {n}^{\lceil (t_3-1)/2 \rceil} \\
    \le & (2k)^{5} \xi n^{(t_1+t_2+t_3-1)/2}\\
    \le & (2k)^{5} \xi n^{k-1}.
\end{aligned}
\end{equation*}

Combining this with the bounds for Types 1 and 2 proves the proposition.
\hfill$\blacksquare$\par

\begin{figure}
    \centering
    \includegraphics[width=0.6\linewidth]{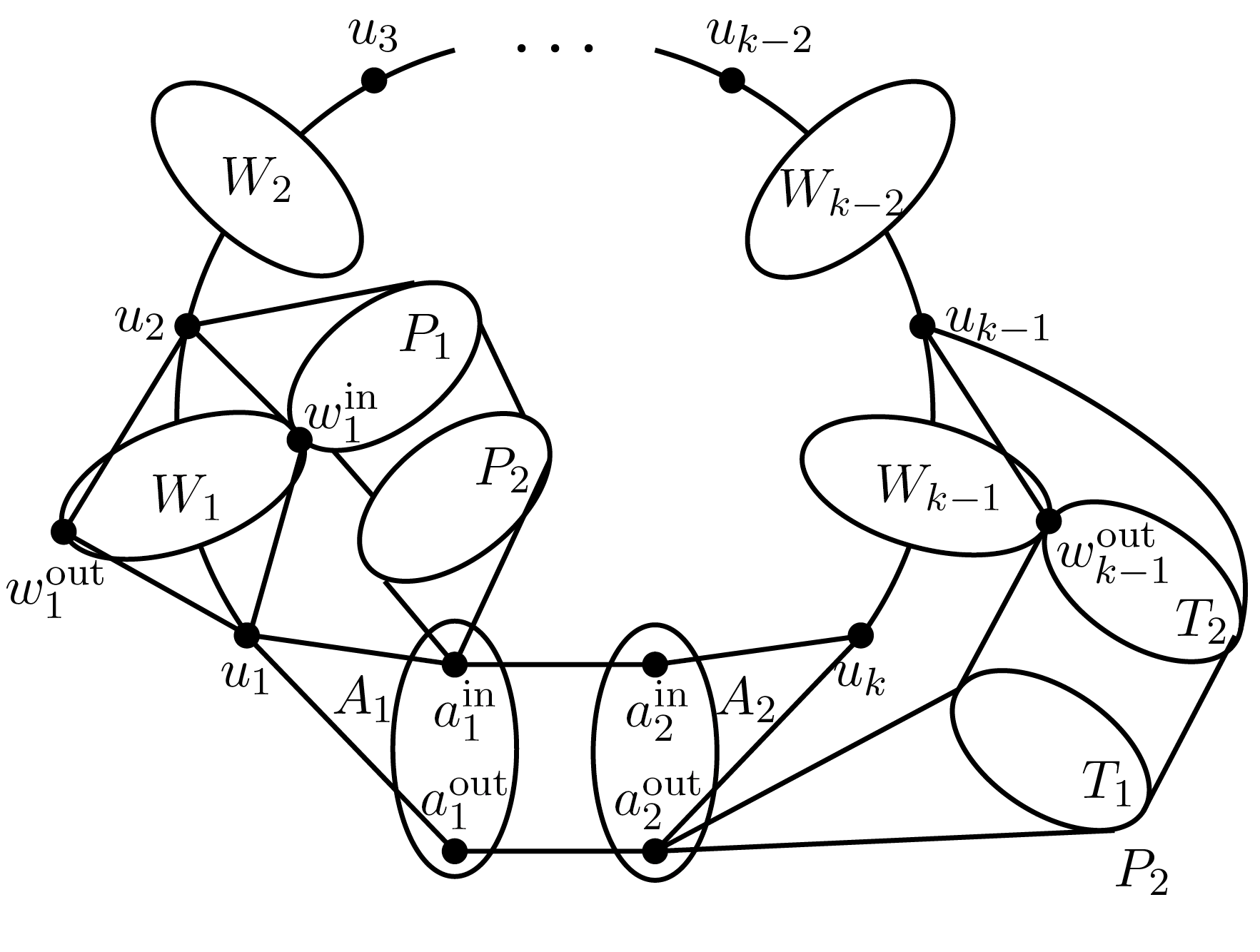}
    \caption{An illustration for the sets $P$ and $T$.}\label{fig: sets P,Q,S,T}
\end{figure}

We now complete the proof of Lemma~\ref{lem: finally}.
Let $\{P_1,P_2\}$ be the $3$-path-decomposition of $(u_2,a_1^{\mathrm{int}})$ in $R^{\mathrm{int}}$ avoiding $u_1$, that is, let $P_i$ be the set of vertices that lie on a path from $u_2$ to $a_1^{\mathrm{int}}$ of length three and avoid $u_1$.
Let $P = P_1 \cup P_2$.
By definition, $P \subseteq X \cup \{w_1^{\mathrm{int}}\}$~(see Figure~\ref{fig: sets P,Q,S,T}).

Similarly, let $\{Q_1,Q_2\}$ be the $3$-path-decomposition of $(a_2^{\mathrm{int}},u_{k-1})$ in $R^{\mathrm{int}}$ avoiding $u_k$.
Let $\{S_1,S_2\}$ be the $3$-path-decomposition of $(u_2,a_1^{\mathrm{ext}})$ in $R^{\mathrm{ext}}$ avoiding $u_1$, and let $\{T_1,T_2\}$ be the $3$-path-decomposition of $(a_2^{\mathrm{ext}},u_{k-1})$ in $R^{\mathrm{ext}}$ avoiding $u_k$.
Let $Q = Q_1 \cup Q_2$, $S = S_1 \cup S_2$, $T = T_1 \cup T_2$.
Then $Q \subseteq X \cup \{w_{k-1}^{\mathrm{int}}\}$, $S \subseteq X \cup \{w_1^{\mathrm{ext}}\}$, and $T \subseteq X \cup \{w_{k-1}^{\mathrm{ext}}\}$.
We claim that $X' = X \setminus (P \cup Q \cup S \cup T)$ is empty.
Otherwise, let $x$ be a vertex in $X'$ with minimum degree in the graph $G[X']$.
Then $d_{G[X']}(x) \le 5$ and $d_{G}(x) \le 2k+6$.
By property (1) and the pigeonhole principle, there exist $x_1,x_2$ in $N(x)$ such that the number of $C_{2k+1}$ containing $x_1 x x_2$ is at least $ {( \frac{n-2k-2}{k} )}^{k-1} / \binom{2k+6}{2}$.
A contradiction to Proposition~\ref{prop: finally} when $\xi$ is small enough.

It is easy to see that $P\cap Q = \emptyset$, otherwise, we get a circuit of odd length smaller than $2k+1$, a contradiction.
Similarly, we have $S \cap T = \emptyset$.
By definition, the sets $P,Q,S,T$ are disjoint.
Let $p' = \max\{|P|-1, 0 \}$, $q' = \max\{|Q|-1, 0 \}$, $s' = \max\{|S|-1, 0 \}$, $t' = \max\{|T|-1, 0 \}$.
Note that if $P$ is non-empty, then $|P| \ge 2$ by the definition.
Moreover, the edges between $P_1$ and $P_2$ form a forest by Observation~\ref{obs: forest between Ai and Ai+1}, thus there are at most $p' = |P|-1$ edges between $P_1$ and $P_2$.
The same bounds hold for $q',s',t'$.
We have now characterized all vertices.
It remains only to count. The number of induced $C_{2k+1}$ in $G$ is at most
\begin{equation}\label{eq: finally}
\begin{aligned}
    & (|A|-1)\prod_{i=1}^{k-1} |W_i| + (p' + s')|A_2| \prod_{i=2}^{k-1} |W_i| + (q' + t')|A_1| \prod_{i=1}^{k-2} |W_i| + (p' + s')(q'+t') \prod_{i=2}^{k-2} |W_i| \\
    \le {} & \prod_{i=2}^{k-2} |W_i| \cdot \big( (|A|-1) |W_1| |W_{k-1}| + (p' + s')|A_2| |W_{k-1}| + (q' + t')|A_1||W_1| + (p' + s')(q'+t') \big),
\end{aligned}
\end{equation}
subject to the constraint that $p'+q'+s'+t'+|A| + \sum_{i=1}^{k-1} |W_i| \le n-k$.
The right-hand side is maximized when $p' = s' = q' = t' = 0$: for instance, if $q' > 0$, moving the value of $q'$ to $|W_1|$ gives a larger value.
The same argument applies to $p'$, $s'$, and $t'$.
Therefore, the maximum number of induced $C_{2k+1}$ in $G$ is
\begin{equation*}
    (|A|-1)\prod_{i=1}^{k-1} |W_i|,
\end{equation*}
subject to the constraint that $(|A|-1) + \sum_{i=1}^{k-1} |W_i|  = n-k-1$, which is at most $h_k(n)$ by the definition of $h_k(n)$.
This completes the proof of Lemma~\ref{lem: finally}.
\end{proof}

This proves Theorem~\ref{thm: main}.

\section{Conclusion}

After a long, sophisticated proof, we have proved Theorem~\ref{thm: main}.
Although the proof is lengthy, much of the work is technical; the main ideas can be summarized as follows.

\begin{itemize}
    \item We work under the three properties introduced in Theorem~\ref{thm: new main}. Property (1) is a standard progressive-induction assumption, while properties (2) and (3) reflect symmetry considerations. These ideas are widely used, but making them compatible in the present setting requires the technical work in Section~\ref{sec: three properties}.
    \item In Section~\ref{subsec: faces}, we prove lemmas that identify empty regions and vertices of degree two. This idea is inspired by the work of Ghosh, Gy\H{o}ri, Janzer, Paulos, Salia and Zamora~\cite{ghosh2022maximum}, who determined the maximum number of induced copies of $C_5$ in a planar graph. For $C_5$, this idea is already close to sufficient; for $C_{2k+1}$, it is not.
    \item Section~\ref{subsec: special path or vertex of degree two} is a key part of the proof. It treats vertices of degree two and special paths in a unified way. One inherent difficulty is that the extremal graph is not unique: it contains a forest, whose two natural extremes are a star and a path.
    The lemmas in Section~\ref{subsec: special path or vertex of degree two} handle these two extreme cases.
    \item In Sections~\ref{subsec: prep Delta 2 Delta 3} and~\ref{subsec: large Delta2}, we prove that $\Delta_2$ is large. This determines a large part of the extremal graph and allows us to uncover the remaining structure. The inequalities in these sections are the main reason the proof works.
    \item With these lemmas in place, the bootstrap argument in Section~\ref{subsec: Bootstrap} is conceptually straightforward: we build the whole structure step by step from a large $\Delta_2$.
\end{itemize}

We do not characterize the extremal graph in this paper.
Although (\ref{eq: finally}) in the proof of Lemma~\ref{lem: finally} shows that a graph with properties (1), (2), and (3) attains the maximum number of $C_{2k+1}$ only when it is one of the extremal graphs described in Theorem~\ref{thm: main}, this only determines the extremal graphs within that restricted class.
To obtain a full characterization of the extremal graph, one would also need to track the operations in Section~\ref{sec: three properties} that transform an arbitrary graph into one with properties (1), (2), and (3).
This seems promising, but for the sake of brevity, we do not pursue it here.

The ultimate goal is to count induced copies after removing the assumption that all shorter odd cycles are forbidden, as in Conjecture~\ref{conj: induced odd cycle}.
Without this assumption, however, our method fails entirely because Lemma~\ref{lem: Ai Aj disjoint} no longer holds.

\section*{Acknowledgement}

The authors used GPT-5.5 only for grammar and style checking. It was not used to develop, verify, or modify any proof or mathematical argument in this paper. 
After using this tool, the authors reviewed and edited the content as needed and take full responsibility for the content of the article

% ---------------- bib -------------------------------
\bibliography{ref.bib}
\bibliographystyle{wyc4}

\end{document}